\documentclass[aos]{imsart}

\RequirePackage{amsthm,amsmath,amsfonts,amssymb}
\RequirePackage[numbers]{natbib}
\RequirePackage[colorlinks,citecolor=blue,urlcolor=blue]{hyperref}
\RequirePackage{graphicx}
\usepackage{float}
\usepackage{epsfig,array,multirow}
\usepackage[usenames]{color}

\usepackage{xr}

\startlocaldefs

\theoremstyle{plain}
\newtheorem{theorem}{Theorem}[section]	
\newtheorem{lemma}{Lemma}[section]
\newtheorem{corollary}{Corollary}[section]
\newtheorem{proposition}{Proposition}[section]
\theoremstyle{definition}

\newtheorem{remark}{Remark}[section]

\newcommand{\R}{\mathbb{R}}

\renewcommand{\qed}{\hfill{\tiny \ensuremath{\blacksquare} }}%

\newcommand{\Ep}{{\mathrm{E}}}

\renewcommand{\Pr}{{\mathrm{P}}}

\newcommand{\red}[1]{\textcolor{red}{#1}}

\def\ben#1{\begin{equation}#1\end{equation}}

\def\ba#1{\begin{align*}#1\end{align*}}

\newcommand{\wh}[1]{\hat{#1}}

\endlocaldefs
\allowdisplaybreaks

\begin{document}

\begin{frontmatter}
\title{Improved Central Limit Theorem and Bootstrap Approximations in High Dimensions}
\runtitle{Improved CLT and bootstrap in high dimensions}

\begin{aug}
\author[A]{\fnms{Victor} \snm{Chernozhuokov}\ead[label=e1]{vchern@mit.edu}},
\author[B]{\fnms{Denis} \snm{Chetverikov}\ead[label=e2,mark]{chetverikov@econ.ucla.edu}}
\author[C]{\fnms{Kengo} \snm{Kato}\ead[label=e3,mark]{kk976@cornell.edu}} \\
\and
\author[D]{\fnms{Yuta} \snm{Koike}\ead[label=e4,mark]{kyuta@ms.u-tokyo.ac.jp}}
\address[A]{
Department of Economics and Operations Research Center, MIT,
\printead{e1}}

\address[B]{
Department of Economics, UCLA,
\printead{e2}}

\address[C]{
Department of Statistics and Data Science, Cornell University,
\printead{e3}}

\address[D]{Department,
Mathematics and Informatics Center and Graduate School of Mathematical Sciences, The University of Tokyo,
\printead{e4}}

\end{aug}

\begin{abstract}
This paper deals with the Gaussian and bootstrap approximations to the distribution of the max statistic in high dimensions. This statistic takes the form of the maximum over components of the sum of independent random vectors and its distribution plays a key role in many high-dimensional estimation and testing problems. Using a novel iterative randomized Lindeberg method, the paper derives new bounds for the distributional approximation errors. These new bounds substantially improve upon existing ones and simultaneously allow for a larger class of bootstrap methods.
\end{abstract}

\begin{keyword}[class=MSC2020]
\kwd{60F05, 62E17}
\end{keyword}

\begin{keyword}
\kwd{bootstrap}
\kwd{central limit theorem}
\kwd{iterative randomized Lindeberg method}
\kwd{Stein kernel}
\end{keyword}

\end{frontmatter}

\section{Introduction}

Let $X_1,\dots,X_n$ be independent random vectors in $\mathbb R^p$ such that
$\Ep[X_{i j}] = \mu_{j}$ for all $i=1,\dots,n$ and $j=1,\dots,p$,
where $X_{i j}$ denotes the $j$th component of the vector $X_i$. We are interested in approximating the distribution of the maximum coordinate of the centered sample mean of $X_{1},\dots,X_{n}$, i.e., 
\begin{equation}\label{eq: statistic first}
T_n = \max_{1\leq j\leq p}\frac{1}{\sqrt n}\sum_{i=1}^n(X_{i j} - \mu_{j}).
\end{equation}
The distribution of $T_n$ plays a particularly important role in many high-dimensional settings, where $p$ is potentially larger or much larger than $n$. For example, it appears in selecting the regularization parameters for the Lasso estimator and the Dantzig selector (\cite{CCK13}), in carrying out reality checks for data snooping and testing superior predictive ability (\cite{W00,H05}), in constructing model confidence sets (\cite{HLN11}), in testing conditional and/or many unconditional moment inequalities (\cite{BSS19, C18, CCK19, KB19}), in multiple testing with the family-wise error rate control (\cite{BCCHK18}), in constructing simultaneous confidence intervals for high-dimensional parameters (\cite{BCCW18}), in adaptive testing of regression and stochastic monotonicity (\cite{C19,CWK18}), in carrying out inference on generalized instrumental variable models (\cite{CR19}), and in constructing Lepski-type procedures for adaptive estimation and inference in nonparametric problems (\cite{CCK14}); more references can be found in \cite{DZ17} and especially in \cite{BCCHK18}. It is therefore of great interest to develop methods for obtaining feasible and accurate approximations to the distribution of $T_n$, allowing for the high-dimensional $p \gg n$ case.


Toward this goal, the first three authors of this paper obtained the following Gaussian approximation result in \cite{CCK13, CCK17}. Let $G = (G_1,\dots,G_p)'$ be a Gaussian random vector in $\mathbb R^p$ with mean $\mu = (\mu_1,\dots,\mu_p)'$ and covariance matrix $n^{-1}\sum_{i=1}^n\Ep[(X_i - \mu)(X_i - \mu)']$ and let the critical value $c_{1-\alpha}$ be the $(1-\alpha)$th quantile of $\max_{1\leq j\leq p}G_j$. Then under mild regularity conditions,
\begin{equation}\label{eq: basic approximation}
\Big|\Pr(T_n > c_{1-\alpha}) - \alpha \Big|\leq C\left(\frac{\log^7(p n)}{n}\right)^{1/6},
\end{equation}
where $C$ is a constant that is independent of $n$ and $p$. This result is important because the right-hand side of the bound \eqref{eq: basic approximation} depends on $p$ only via the logarithm of $p$, and hence it shows that the Gaussian approximation holds if $\log  p = o(n^{1/7})$, which allows $p$ to be much larger than $n$. Besides, building upon this result, the same authors have proved bounds similar to (\ref{eq: basic approximation}) for the critical values obtained by the Gaussian multiplier and empirical bootstraps in \cite{CCK17}.

Gaussian approximation of the form (\ref{eq: basic approximation}) allows us to develop powerful inference methods for high-dimensional data in applications discussed above and has stimulated further developments into dependent data \cite{ZW17,ZC17,CCK19}, $U$-statistics \cite{C18,CK19a,CK19b}, Malliavin calculus \cite{C19}, and homogeneous sums \cite{K19a}. Despite such rapid developments, the literature has left much to be desired on  coherent understanding of sharpness of the bound (\ref{eq: basic approximation}) for the Gaussian or bootstrap critical values since the first appearance of \cite{CCK17} in 2014 on arXiv. The problem can be decomposed into two parts: (i) sharpness of the bound in terms of dependence on $n$ and (ii) sharpness of the bound in terms of dependence on $p$.

There are two important developments toward the question of sharpness of the bound (\ref{eq: basic approximation}) that should be mentioned.  
First, Deng and Zhang \cite{DZ17} considered direct bootstrap approximation without taking the root of Gaussian approximation, and proved the following bound for the critical value $c_{1-\alpha}$ obtained by the empirical or  third-order matching (or Mammen's \cite{M93}) multiplier bootstraps: 
\begin{equation}\label{eq: dz bound}
\Big|\Pr(T_n > c_{1-\alpha}) - \alpha \Big|\leq C\left(\frac{\log^5(p n)}{n}\right)^{1/6}.
\end{equation}
Their bound improves the power of the logs in the previous bound (\ref{eq: basic approximation}), showing that the empirical and Mammen's bootstraps are consistent to approximate the distribution of $T_{n}$ if $\log p = o(n^{1/5})$ instead of $\log p = o(n^{1/7})$. 
Second, the recent preprint by the fourth author \cite{K19b} shows that the same bound (\ref{eq: dz bound}) indeed holds for the Gaussian critical value as well. 

In turn, in this paper, we show that in fact a much larger improvement is possible: under mild regularity conditions, we prove that
\begin{equation}\label{eq: empirical bootstrap introduction}
\Big|\Pr(T_n > c_{1-\alpha}) - \alpha \Big|\leq C\left(\frac{\log^5(p n)}{n}\right)^{1/4},
\end{equation}
both for the Gaussian and  bootstrap critical values $c_{1-\alpha}$. In comparison with the Gaussian approximation result \eqref{eq: basic approximation}, our new bound improves not only the power of the logs but also the power of the sample size $n$. Moreover, regarding the bootstrap types, we allow for not only the empirical and third-order matching multiplier bootstrap methods, but also for general multiplier bootstrap methods (with i.i.d weights), which match only two moments of the data, such as the multiplier bootstrap methods with Gaussian and Rademacher weights.

We remark that several authors have recently pointed out that an additional structural assumption on the covariance matrices of $X_i$'s can improve the bound \eqref{eq: empirical bootstrap introduction}. In particular, Fang and Koike \cite{FK20} showed that the right-hand side of \eqref{eq: empirical bootstrap introduction} can be improved to $C(\log^4(pn)/n)^{1/3}$ when the covariance matrices are non-degenerate and can be further improved to $C(\log^3(p)/n)^{1/2}\log n$ when we additionally assume that $X_i$'s have log-concave densities. The latter result is based on the fact that random vectors with log-concave densities admit Stein kernels with sub-Weibull entries, which is established by Fathi in \cite{Fa19}. Moreover, building on the important results by Lopes in \cite{L20} and Kuchibhotla and Rinaldo in \cite{KR20}, \cite{CCK20} showed that the bound $C(\log^3(p)/n)^{1/2}\log n$ can be achieved even without the assumption of log-concave densities (non-degenerate covariance matrices are still required; \cite{L20} and \cite{KR20} were the first to obtain dependence on $n$ via $1/\sqrt n$ in \eqref{eq: empirical bootstrap introduction} without requiring log-concave densities). In addition, Lopes, Lin and M\"uller \cite{LLM20} showed that the right-hand side of \eqref{eq: empirical bootstrap introduction} can be improved to $C n^{-1/2+\delta}$ for any $\delta>0$ when the coordinates of $X_i$'s have decaying variances. Compared to these results, our bound requires neither non-degenerate covariance matrices nor decaying variances.

In addition, we prove that if the distribution of the random vectors $X_1,\dots,X_n$ is symmetric around the mean, then even better approximation to the distribution of $T_n$ is possible:
\begin{equation}\label{eq: randomization introduction}
\Big|\Pr(T_n > c_{1-\alpha}) - \alpha \Big|\leq C\left(\frac{\log^3(p n)}{n}\right)^{1/2}
\end{equation}
as long as the critical value $c_{1-\alpha}$ is obtained via the multiplier bootstrap method with Rademacher weights. This new bound makes Rademacher weights particularly appealing in the high-dimensional settings, at least from a theoretical perspective.

We also consider bootstrap approximations with incremental factors, previously used by Andrews and Shi in \cite{AS13} in the context of testing conditional moment inequalities. Specifically, for a small but fixed constant $\eta > 0$, called an incremental factor, we derive the following bounds:
\begin{equation}\label{eq: or 1}
\Pr(T_n > c_{1-\alpha} + \eta) - \alpha \leq C\left(\frac{\log^3(p n)}{n}\right)^{1/2}
\end{equation}
if $c_{1-\alpha}$ is obtained via either the empirical or the third-order matching multiplier bootstrap methods and
\begin{equation}\label{eq: or 2}
\Pr(T_n > c_{1-\alpha} + \eta) - \alpha \leq C\left(\frac{\log^5(p n)}{n}\right)^{1/2}
\end{equation}
if $c_{1-\alpha}$ is obtained via general multiplier bootstrap methods, where the constant $C$ may depend on $\eta$. Even though these are one-sided bounds, they are useful because they show that in any test based on the statistic $T_n$, increasing the critical value $c_{1-\alpha}$ by an incremental factor $\eta$ may substantially reduce the sample complexity for over-rejection. Namely,  assuming $\log p \gtrsim \log n$ for simplicity, for the over-rejection probability to be less than or equal to a given level $0 < \Delta < 1-\alpha$, the empirical bootstrap or multiplier bootstrap (without incremental factor) requires $n \gtrsim \Delta^{-4} \log^5 p$, while adding a constant incremental factor reduces the sample complexity to $n \gtrsim (\Delta^{-2} \log^3p) \vee \log^5 p$ if we use the empirical or third-order matching bootstrap. 
It is worth noting that, given that in high-dimensional settings, where $p$ is rapidly increasing together with $n$, $c_{1-\alpha}$ is typically also getting large as we increase $n$, adding an incremental factor $\eta$ may not have a large impact on the power properties of the test.

In fact, all our results apply to a more general version of the statistic $T_n$:
\begin{equation}\label{eq: statistic second}
T_n = \max_{1\leq j\leq p}\frac{1}{\sqrt n}\sum_{i=1}^n(X_{i j} - \mu_{j} + a_j),
\end{equation}
where $a=(a_1,\dots,a_p)'$ is a vector in $\mathbb R^p$, which reduces to \eqref{eq: statistic first} if we set $a = 0_p$. In most applications mentioned above, the former version \eqref{eq: statistic first} is sufficient but there are some applications where the more general version \eqref{eq: statistic second} is required; for example, the latter was used by Bai, Shaikh, and Santos in \cite{BSS19} to extend the method of testing moment inequalities proposed in \cite{RSW14} for the case of a small number of inequalities to the case of a large number of inequalities. For the rest of the paper, we will therefore work with the more general version \eqref{eq: statistic second} of the statistic $T_n$. In addition, we emphasize that our results can be equally applied with
$$
T_n = \max_{1\leq j\leq p}\left|\frac{1}{\sqrt n}\sum_{i=1}^n(X_{i j} - \mu_j + a_j)\right|
$$
by replacing the $p$-dimensional vectors $X_i - \mu + a$ with the $2p$-dimensional vectors whose first $p$ components are equal to $X_i - \mu + a$ and the last $p$ components are equal to $-(X_i - \mu + a)$. 

To prove \eqref{eq: empirical bootstrap introduction}, we develop a novel and iterative version of the randomized Lindeberg method. A key feature of our approach is that we carry out a careful analysis of the coefficients in the Taylor expansion underlying the Lindeberg method. In particular, we apply the Lindeberg method iteratively in combination with an anti-concentration inequality for maxima of Gaussian processes to bound these coefficients, which substantially improves upon the original randomized Lindeberg method proposed in \cite{DZ17}. In addition, we sharpen the Gaussian approximation bounds for the multiplier processes developed in \cite{K19b} using Stein's kernels. In turn, to prove \eqref{eq: randomization introduction}, we establish a new connection between the Rademacher bootstrap and the randomization tests, as discussed in \cite{LR05}, using a recent result from the computer science literature on pseudo-random number generators by O'Donnell, Servedio, and Tan \cite{OST18}, which provides an anti-concentration inequality for maxima of Rademacher processes. Finally, to prove error bounds  \eqref{eq: or 1} and \eqref{eq: or 2}, we apply the original randomized Lindeberg method as developed in \cite{DZ17}.

Finally, we conduct a small scale simulation study. Our simulation study shows that (i) all bootstrap methods considered in this paper perform reasonably well in high dimensions; (ii) for asymmetric distributions, the empirical and
the third-order matching multiplier bootstrap methods outperform the multiplier bootstrap
methods with Gaussian and Rademacher weights; and (iii) for symmetric distributions, the multiplier bootstrap with Rademacher weights performs the best, which is consistent with Theorem \ref{thm: gauss and rademacher} ahead. See the Supplementary Material for details.

The rest of the paper is organized as follows. In the next section, we present our main results. In Section \ref{sec: main arguments}, we develop the iterative randomized Lindeberg method, which is the first key component in deriving our main results. In Section \ref{sec: stein kernels}, we provide new bounds for the Gaussian approximations using Stein's kernels, which is the second key component in deriving our main results. In Section \ref{sec: proofs}, we give proofs of the main results. In the Supplemental Material, we collect additional derivations and conduct a small simulation study.


\subsection{Notation} For any vectors $x,y\in\mathbb R^p$ and any scalar $c\in\mathbb R$, we write $x\leq y$ if $x_j \leq y_j$ for all $j=1,\dots,p$ and  write $x + c$ to denote the vector in $\mathbb R^p$ whose $j$th component is $x_j + c$ for all $j = 1,\dots,p$. Also, for any sequences of scalars $\{a_n\}_{n\geq 1}$ and $\{b_n\}_{n\geq 1}$ we write $a_n\lesssim b_n$ if $a_n \leq C b_n$ for all $n\geq 1$  for some constant $C$. Recall that,  for any random variable $T$ and a constant $\gamma\in(0,1)$, the $\gamma$th quantile of $T$ is defined as  $\inf\{t\in\mathbb R\colon  \Pr(T\leq t) \ge \gamma \}$. Finally, we use the notation $X_{1:n} = (X_1,\dots,X_n)$.

\section{Main Results}\label{sec: main results}
In this section, we present our main results. We first formally define all the critical values $c_{1-\alpha}$ to be used throughout the paper. We then discuss the required regularity conditions and present the results.

\subsection{Gaussian and Bootstrap Critical Values}

First, define the Gaussian critical value $c_{1-\alpha}^G$ as the $(1-\alpha)$th quantile of 
\begin{equation}\label{eq: gaussian analog}
T_n^G = \max_{1\leq j\leq p}(G_j + a_j),
\end{equation}
where $G$ is a centered Gaussian random vector in $\mathbb R^p$ with the covariance matrix
\begin{equation}\label{eq: sigma definition}
\Sigma_n = \frac{1}{n}\sum_{i=1}^n\Ep[(X_i-\mu)(X_i-\mu)'],
\end{equation}
which coincides with the covariance matrix of $\sqrt{n}(\bar{X}_n-\mu)$.
Second, define the bootstrap critical value $c_{1-\alpha}^B$ as the $(1-\alpha)$th quantile of the conditional distribution of
\begin{equation}\label{eq: bootstrap statistic}
T_n^* = \max_{1\leq j\leq p}\frac{1}{\sqrt n}\sum_{i=1}^n(X_{i j}^* + a_j)
\end{equation}
given the data $X_1,\dots,X_n$, where $X_1^*,\dots,X_n^*$ is a (not necessarily empirical) bootstrap sample. We consider the following types of the bootstrap:
\begin{itemize}
\item  Empirical bootstrap: let $X_1^*,\dots,X_n^*$ be a sequence of i.i.d. random variables sampled from the uniform distribution on $\{X_1 - \bar X_n,\dots,X_n - \bar X_n\}$, where $\bar X_n = n^{-1}\sum_{i=1}^n X_i$ denotes the sample mean of the data $X_1,\dots,X_n$. 
\item Multiplier bootstrap: let $e_1,\dots,e_n$ be a sequence of i.i.d. random variables with mean zero and unit variance, referred to as weights, which are independent of $X_1,\dots,X_n$. Define $X_i^* = e_i(X_i - \bar X_n)$ for all $i = 1,\dots,n$. 
\end{itemize}

For the multiplier bootstrap, we will assume throughout the paper that the weights $e_1,\dots,e_n$  are such that
\begin{equation}\label{eq: multiplier bootstrap simplification}
\begin{array}{l}
\text{$e_i=e_{i,1}+e_{i,2}$, where $e_{i,1}$ and $e_{i,2}$ are independent, $e_{i,1}$ has}\\ 
\text{the $N(0,\sigma_e^2)$ distribution for some $\sigma_e\geq0$, and $|e_{i,2}|\leq3$.}
\end{array}
\end{equation}
Condition (\ref{eq: multiplier bootstrap simplification}) is mild and covers many commonly used weights, such as:
\begin{itemize}
\item Gaussian weights: $e_{i,1} \sim N(0,1)$ and $e_{i,2} = 0$.
\item Rademacher weights: $e_{i,1} = 0$ (i.e., $\sigma_{e} =0$) and $\Pr (e_{i,2} = \pm 1) = 1/2$.
\item Mammen's weights \cite{M93}: $e_{i,1} = 0$  and
\[
\Pr\left (e_{i,2} =  \frac{1 \pm \sqrt{5}}{2} \right ) =\frac{\sqrt{5} \mp 1}{2\sqrt{5}}.
\]
\end{itemize}
See Remark \ref{rem: weight discussion} for further discussion on Condition (\ref{eq: multiplier bootstrap simplification}).

Occasionally, we will also consider the weights with unit third moment, namely, 
\begin{equation}\label{eq: third order matching multipliers}
\Ep[e_i^3] = 1,\quad\text{for all }i=1,\dots,n.
\end{equation}
The weights satisfying Condition (\ref{eq: third order matching multipliers}) correspond to  the third-order matching multiplier bootstrap mentioned in the Introduction. We note that Mammen's weights satisfy both Conditions (\ref{eq: multiplier bootstrap simplification}) and (\ref{eq: third order matching multipliers}), but neither Rademacher nor Gaussian weights satisfy Condition \eqref{eq: third order matching multipliers}. 
See Lemma \ref{lem: multipliers} in the Supplemental Material, where we provide a more general class of distributions for the weights satisfying both  Conditions \eqref{eq: multiplier bootstrap simplification} and \eqref{eq: third order matching multipliers}.




Before proceeding to the regularity conditions, we also note that the multiplier bootstrap critical value $c_{1-\alpha}^B$ with Gaussian weights can be regarded as a feasible version of the Gaussian critical value $c_{1-\alpha}^G$. Indeed, it is easy to see that the former can be alternatively defined as the $(1-\alpha)$th quantile of the distribution of 
$$
T_n^{\hat G} = \max_{1\leq j\leq p}(\hat G_j + a_j),$$
where 
$
\hat G\sim N(0_p,\widehat\Sigma_n)$ and $\widehat \Sigma_{n}$ is the empirical covariance matrix
\begin{equation}\label{eq: sigma estimator definition}
\widehat\Sigma_n = n^{-1}\sum_{i=1}^n(X_i-\bar X_n)(X_i-\bar X_n)'. 
\end{equation}
For brevity, we sometimes refer to both quantities as the Gaussian critical values. 

\begin{remark}[On Condition (\ref{eq: multiplier bootstrap simplification})]
\label{rem: weight discussion}
Condition (\ref{eq: multiplier bootstrap simplification}) is technical and can be weakened depending on the moment conditions on $X_i$. A key step in the proof of Theorem \ref{cor: rejection probabilities} is to apply Theorem \ref{cor: max} ahead to approximate the conditional distribution of $T_n^*$ with that of the multiplier bootstrap statistic with weights following a Beta distribution that matches the moments of $e_i$ up to the third order (to be precise, we first replace the Gaussian components $e_{i,1}$ by bounded weights in the proof of Theorem \ref{cor: rejection probabilities}). Condition (\ref{eq: multiplier bootstrap simplification}) will be used to verify Conditions V, P, and B when we apply Theorem \ref{cor: max} there. 
If, e.g., $X_i$ are bounded by $B_n$, then the conclusion of Theorem \ref{cor: rejection probabilities} continues to hold for sub-exponential weights. Since current Condition (\ref{eq: multiplier bootstrap simplification})  already covers many commonly used bootstrap weights, however, we do not pursue this generality of the weights to keep our presentation reasonable concise.
\end{remark}

\subsection{Regularity Conditions}
First, observe that given the construction of the statistic $T_n$ in \eqref{eq: statistic second} and its Gaussian and bootstrap analogs in \eqref{eq: gaussian analog} and \eqref{eq: bootstrap statistic}, it is without loss of generality to assume that $\mu_{j} = 0$ for all $j = 1,\dots,p$, which is what we do for the rest of the paper. Also, all our results follow immediately if $n= 2$, so we assume $n\geq 3$,  which  in particular implies $\log(p n)\geq 1$. In addition, since we are primarily interested in the case with large $p$, we assume $p\geq 2$.

Second, let $b_1$ and $b_2$ be some strictly positive constants such that $b_1\leq b_2$ and let $\{B_n\}_{n\geq 1}$ be a sequence of constants such that $B_n \geq 1$ for all $n\geq 1$. Here, the sequence $\{B_n\}_{n\geq 1}$ can diverge to infinity as the sample size $n$ increases. 


\medskip

\noindent
{\bf Condition E:} {\em For all $i=1,\dots,n$ and $j = 1,\dots,p$, we have 
\[
\Ep[\exp(|X_{i j}|/B_n)]\leq 2.
\]
}

\noindent
{\bf Condition M:} {\em For all $j=1,\dots,p$, we have 
$$
b_1^2\leq \frac{1}{n}\sum_{i=1}^n \Ep[X_{i j}^2]\quad\text{and}\quad \frac{1}{n}\sum_{i=1}^n \Ep[X_{i j}^4]\leq B_n^2 b_2^2.
$$
}

\noindent
{\bf Condition S:} {\em For all $i=1,\dots,n$, the distribution of $X_i$ is symmetric in the sense that $X_i$ and $-X_i$ are identically distributed.}

\medskip

Condition E implies that the random variables $X_{i j}$ are sub-exponential with the Orlicz $\psi_1$-norm bounded by $B_n$; see \cite{V11} for details. The same sub-exponential condition was assumed in e.g. \cite{CCK17} and \cite{DZ17}; see Condition (E.1) in \cite{CCK17} and (E.1) in \cite{DZ17}.
The first part of Condition M, which we refer to as the variance lower bound condition,  requires that each component of the random vectors $X_i$ is scaled properly. The variance lower bound condition is needed to apply the anti-concentration inequalities (cf.  Lemmas \ref{lem: anticoncetration} and \ref{lem: rademacher anticoncentration} in the Supplemental Material) but can be dropped in Theorem \ref{thm: infinitesimal factors} ahead. Also, at least for Theorems \ref{thm: gaussian approximation main} and \ref{cor: rejection probabilities}, it can be relaxed by using Theorem 10 in \cite{DZ17}. However, to consistently state all the results, we work with the present assumption. Given the first part, the second part of Condition M holds if, for example, all random variables $X_{i j}$ are bounded  by $B_n$ and $n^{-1}\sum_{i=1}^n\Ep[X_{ij}^2]\leq b_2^2$ for all $j=1,\dots,p$. 
Condition S means that the distribution of each $X_i$ is symmetric around the mean. 
Importantly, none of these conditions restrict the correlation matrices of $X_i$, and so our results do not follow from the classical results in  empirical process theory.

In what follows, we will always  maintain Conditions E and M and will assume Condition S only in Theorem \ref{thm: gauss and rademacher}, which shows that imposing the symmetric distributions improves the approximation bound for the multiplier bootstrap with Rademacher weights.

\subsection{Main Results} We first present a non-asymptotic bound on the error of the Gaussian approximation to the distribution of the statistic $T_n$:
\begin{theorem}[Gaussian Approximation]\label{thm: gaussian approximation main}
Suppose that Conditions E and M are satisfied. Then
\begin{equation}\label{eq: gaussian bound main}
\left|\Pr\left(T_n > c^G_{1-\alpha}\right) - \alpha\right| \leq C\left(\frac{B_n^2\log^5(p n)}{n}\right)^{1/4},
\end{equation}
where $C$ is a constant depending only on $b_1$ and $b_2$. 
\end{theorem}
This result improves upon the bound in \cite{K19b}, who obtained a similar result with the rate $1/6$ instead of $1/4$. Since $a\in\mathbb R^p$ in the definition of $T_n$ in \eqref{eq: statistic second} is arbitrary, the bound \eqref{eq: gaussian bound main} can be equivalently stated as
$$
\sup_{A\in\mathcal A}\left|\Pr\left(\frac{1}{\sqrt n}\sum_{i=1}^n X_i \in A\right) - \Pr(G\in A)\right|\leq C\left(\frac{B_n^2\log^5(p n)}{n}\right)^{1/4},
$$
where $G\sim N(0_p,\Sigma_n)$ and $\mathcal A$ is the class of all hyper-rectangles in $\mathbb R^p$, i.e. sets of the form
$$
A = \Big\{ w = (w_1,\dots,w_p)'\in\mathbb R^p\colon a_{lj} \leq w_j \leq a_{rj}\text{ for all }j=1,\dots,p \Big\},
$$
for some constants $-\infty\leq a_{lj} \leq a_{rj} \leq \infty$ with $j=1,\dots,p$. This gives a quantitative Central Limit Theorem (CLT) over the hyper-rectangles in high dimensions.


The proof of Theorem \ref{thm: gaussian approximation main}, which is deferred to Section \ref{sec: proofs}, is fairly complicated and goes somewhat backward: (i) we first compare the conditional distribution of a third-order matching bootstrap statistic $T_n^*$ with that of the Gaussian multiplier bootstrap statistic $T_n^{\hat{G}}$, and then compare the conditional distribution of $T_n^{\hat{G}}$ with the distribution of $T_n^G$. These two comparisons rely on the Gaussian approximation via Stein kernel (Theorem \ref{thm: stein kernel}). Then, (ii) we use the preceding comparison between $T_n^*$ and $T_n^{G}$ to verify the anti-concentration for $T_n^*$ to invoke Theorem \ref{cor: max} and compare the conditional distribution of $T_n^*$ with the distribution of $T_n$. The proof of Theorem \ref{cor: max} relies on a novel technique which we call the \textit{iterative randomized Lindeberg method}. 
The conclusion of Theorem \ref{thm: gaussian approximation main} follows from combining the results in Steps (i) and (ii) and the triangle inequality.

Comparison of the the Gaussian multiplier bootstrap statistic $T_n^{\hat{G}}$ with $T_n^G$ relies on  the following Gaussian-to-Gaussian comparison inequality, which can be of independent interest and whose proof is presented in Section \ref{sec: stein kernels} as a consequence of Theorem \ref{thm: stein kernel}:

\begin{proposition}[Gaussian-to-Gaussian Comparison]\label{coro:g-g-comparison}
If $Z_1$ and $Z_2$ are centered Gaussian random vectors in $\mathbb R^p$ with covariance matrices $\Sigma^1$ and $\Sigma^2$, respectively, and $\Sigma^2$ is such that $\Sigma^2_{jj}\geq c$ for all $j=1,\dots,p$ for some constant $c>0$, then
$$
\sup_{y\in\mathbb R^p}\Big|\Pr(Z_1\leq y) - \Pr(Z_2\leq y)\Big| \leq C\Big(\Delta \log^2 p\Big)^{1/2},
$$
where $C$ is a constant depending only on $c$ and $\Delta = \max_{1\leq j,k\leq p}|\Sigma^1_{jk} - \Sigma^2_{jk}|$. 
\end{proposition}
\begin{remark}
Two comments on Proposition \ref{coro:g-g-comparison} are warranted. First, Proposition \ref{coro:g-g-comparison} improves upon Theorem 2 in \cite{CCK15}, which shows that
$$
\sup_{x\in\mathbb R}\left|\Pr\left(\max_{1\leq j\leq p}Z_{1j}\leq x\right) - \Pr\left(\max_{1\leq j\leq p}Z_{2j}\leq x\right)\right| \leq C\Big(\Delta \log^2 p\Big)^{1/3},
$$
under the same conditions. Second, the bound in this proposition is sharp in the sense that there exists a constant $c>0$ such that for infinitely many values of $p$, there exist centered Gaussian random vectors $Z_1$ and $Z_2$ in $\mathbb R^p$ such that the covariance matrix $\Sigma^2$ of $Z_2$ satisfies $\Sigma_{jj}^2 = 1$ for all $j=1,\dots,p$ and
$$
\sup_{y\in\mathbb R^p}\Big|\Pr(Z_1\leq y) - \Pr(Z_2\leq y)\Big| \geq c\Big(\Delta \log^2 p\Big)^{1/2}.
$$
The latter claim is proven in Appendix \ref{sec: sharpness} of the Supplemental Material.\qed
\end{remark}

Comparison of the conditional distribution of the third-order matching bootstrap statistic $T_n^*$ with that of $T_n$ (Theorem \ref{cor: max}) relies on the iterative randomized Lindeberg method. 
An intuition behind the iterative randomized Lindeberg method goes as follows. Recall that, for any smooth function $g\colon\mathbb R^p\to\mathbb R$ and any two sequences of independent random vectors $X_1,\dots,X_n$ and $Y_1,\dots,Y_n$ in $\mathbb R^p$, in order to approximate $\Ep[g(X_1+\dots+X_n)]$ by $\Ep[g(Y_1+\dots+Y_n)]$, the original Lindeberg method constructs an interpolation path from $\Ep[g(X_1+\dots+X_n)]$ to $\Ep[g(Y_1+\dots+Y_n)]$ by replacing $X_i$'s with $Y_i$'s one-by-one in a given order and uses Taylor's expansion to show that the change in the expectation at each step is sufficiently small; see \cite{C06} for example. The randomized Lindeberg method, introduced in \cite{DZ17}, is similar to the original Lindeberg method but it replaces $X_i$'s with $Y_i$'s in a randomly selected order. It turns out that this randomization may bring substantial benefits to the final bound. In turn, to improve upon this version of the randomized Lindeberg method, we carry out a careful analysis of the coefficients in the Taylor's expansions underlying the method. In particular, given that $k$th order coefficients take the form of $\Ep[g^{(k)}(Z_1+\dots+Z_n)]$, up to some approximation error, where $g^{(k)}$ is a vector of the $k$th partial derivatives of $g$ and $Z_1,\dots,Z_n$ is a sequence such that some of its elements are given by $X_i$'s and others by $Y_i$, and using the fact that it is easier in our setting to bound $\Ep[g^{(k)}(Y_1+\dots+Y_n)]$, we apply the randomized Lindeberg method once again to approximate $\Ep[g^{(k)}(Z_1+\dots+Z_n)]$ by $\Ep[g^{(k)}(Y_1+\dots+Y_n)]$. Here, since a new application of the method will bring new Taylor's coefficients, we apply the same method over and over again until the approximation error becomes sufficiently small. We demonstrate that this iterative use of the randomized Lindeberg method gives further substantial benefits to the final bound. See also the discussion before Lemma \ref{lem: main} concerning comparisons of the iterative randomized Lindeberg method with the randomized Lindeberg method used in \cite{DZ17} and the related Slepian-Stein method used in our earlier work \cite{CCK13,CCK17}. 

Our second main result gives a non-asymptotic bound on the deviation of the bootstrap rejection probabilities $\Pr(T_n > c_{1-\alpha}^B)$ from the nominal level $\alpha$ for the empirical and the multiplier bootstrap methods: 
\begin{theorem}[Bootstrap Approximation]\label{cor: rejection probabilities}
Suppose that Conditions E and M are satisfied and that $c_{1-\alpha}^B$ is obtained via either the empirical bootstrap or the multiplier bootstrap with weights satisfying \eqref{eq: multiplier bootstrap simplification}. Then 
\begin{equation}\label{eq: empirical bootstrap main}
\left|\Pr\left(T_n > c_{1-\alpha}^B\right) - \alpha\right| \leq C\left(\frac{B_n^2\log^5(p n)}{n}\right)^{1/4},
\end{equation}
where $C$ is a constant depending only on $b_1$ and $b_2$. 
\end{theorem}

This theorem improves upon the bounds in \cite{DZ17}, who obtained a similar result with the rate $1/6$ instead of $1/4$. In addition, we allow for a larger class of multiplier bootstrap methods. In particular, we do not require the weights $e_1,\dots,e_n$ to satisfy \eqref{eq: third order matching multipliers}. The proof of this theorem is given in Section \ref{sec: proofs}. 

Our third main result gives a non-asymptotic bound on the deviation of the bootstrap rejection probabilities from the nominal level for the multiplier bootstrap method with Rademacher weights in the case of symmetric distributions:

\begin{theorem}[Rademacher Bootstrap Approximation in Symmetric Case]\label{thm: gauss and rademacher}
Suppose that Conditions E, M, and S are satisfied and that $c_{1-\alpha}^B$ is obtained via the multiplier bootstrap with Rademacher weights. Then
\begin{equation}\label{eq: rademacher symmetric}
\left|\Pr\left(T_n > c_{1-\alpha}^B\right) - \alpha\right| \leq C\left(\frac{B_n^2\log^3(p n)}{n}\right)^{1/2},
\end{equation}
where $C$ is a constant depending only on $b_1$ and $b_2$.
\end{theorem}
This theorem implies that the multiplier bootstrap with Rademacher weights is very accurate in the symmetric case. To prove it, we note that under the assumption of symmetric distributions, one can construct the randomization critical value $c_{1-\alpha}^R$ such that $\Pr(T_n > c_{1-\alpha}^R) = \alpha$, up to possible mass points in the distribution of $T_n$. Thus, given that the critical value based on the multiplier bootstrap with Rademacher weights turns out to be a feasible version of this randomization critical value and the two are close to each other, \eqref{eq: rademacher symmetric} follows if we can show that the distribution of $T_n$ is not too concentrated. To this end, we use an anti-concentration inequality for maxima of Rademacher processes derived in \cite{OST18}. The proof of Theorem \ref{thm: gauss and rademacher} is given in Appendix \ref{sec: proof of theorem 23} of the Supplemental Material. 




Our fourth and final result shows that one-sided bounds in the bootstrap approximation can be substantially improved if we allow for incremental factors:

\begin{theorem}[Bootstrap Approximation with Incremental Factors]\label{thm: infinitesimal factors}
Suppose that Conditions E and M are satisfied and let $\eta > 0$ be a constant that may depend on $n$ and $p$. Then there exists a constant C depending only $b_1$ and $b_2$ such that the following hold. 
\begin{enumerate}
\item[(i)] If $B_n^2\log^5(pn)\leq n$ and $c_{1-\alpha}^B$ is obtained via either the empirical bootstrap or the multiplier bootstrap with weights satisfying \eqref{eq: multiplier bootstrap simplification} and \eqref{eq: third order matching multipliers}, then we have 
$$
\Pr(T_n > c_{1-\alpha}^B + \eta) \leq \alpha + C(1\vee\eta^{-4})\left(\frac{B_n^2\log^3(p n)}{n}\right)^{1/2}.
$$
\item[(ii)] If $n^{-1}\sum_{i=1}^n\Ep[X_{ij}^2]\leq b_2^2$ for all $j=1,\dots,p$ and $c_{1-\alpha}^B$ is obtained via the multiplier bootstrap with weights satisfying \eqref{eq: multiplier bootstrap simplification}, then 
$$
\Pr(T_n > c_{1-\alpha}^B + \eta) \leq \alpha + C(1\vee\eta^{-4})\left(\frac{B_n^2\log^5(p n)}{n}\right)^{1/2}.
$$
\end{enumerate}
\end{theorem}

Theorem \ref{thm: infinitesimal factors} allows $\eta$ to (slowly) decrease with $n$ and/or $p$. For example, if we choose $\eta \sim (\log n)^{-1}$, then the over-rejection probability is of order $n^{-1/2}$ in $n$ up to log factors, while only requiring  $p$ to be $\log p = o\big(n^{1/3}/\text{polylog} (n)\big)$ in (i) and $\log p = o\big(n^{1/5}/\text{polylog}(n)\big)$ in (ii) provided that $B_n$ is bounded in $n$. 

To prove this theorem, we use the randomized Lindeberg method but with an important simplification that the incremental factor $\eta$ now absorbs all the terms arising from smoothing the functions of the form $x\mapsto 1\{\max_{1\leq j\leq p}x_j > c\}$, which is used in the Lindeberg method. As discussed in the Introduction, Theorem \ref{thm: infinitesimal factors} is useful if one is concerned with the finite-sample over-rejection of tests based on the statistic $T_n$ as it says that adding an incremental factor $\eta$ to the critical value $c_{1-\alpha}^B$ may substantially reduce over-rejection, with a minimal effect on the power of the test. The proof of Theorem \ref{thm: infinitesimal factors} is given in Appendix \ref{sec: proof of theorem 24} of the Supplemental Material.

We conclude this section with a few remarks on cases with 
 approximate sample means. 
 In many applications (such as simultaneous inference for high-dimensional statistical models; cf. \cite{BCK15}), the statistic $T_n$ can only be asymptotically approximated by the maximum coordinate of the sample mean of independent random vectors. Also, those random vectors, often corresponding to the influence functions, may not be directly observable but have to be estimated. We emphasize here  that all our results can be extended to such approximate sample mean cases using the same arguments as those used in \cite{BCCHK18}; however, we have opted not to carry out the extension here for brevity of the paper.

\subsection{Gaussian and Bootstrap Approximations under Polynomial Moment Conditions}\label{sec: polynomial moment conditions}

So far we have assumed the sub-exponential condition (Condition E) for $X_i$. It turns out that combining some elements of the proof of Lemma \ref{lem: main} below and a truncation argument leads to analogs of the Gaussian and bootstrap approximation results  under polynomial moment conditions, which are given next. The proofs of Theorems \ref{cor: ga under polynomial conditions} and \ref{cor: boot app pol mom cond} can be found in Appendix \ref{sec: proof of poly moment} in the Supplementary Material.


\begin{theorem}[Gaussian Approximation under Polynomial Moment Conditions]\label{cor: ga under polynomial conditions}
Suppose that Condition M is satisfied and that for some $q>2$, we have 
\begin{equation}
\Ep\left[\max_{1\leq j\leq p}|X_{ij}|^q\right]\leq B_n^q
\label{eq: poly}
\end{equation}
for all $i=1,\dots,n$. Then
$$
\left|\Pr\left(T_n > c^G_{1-\alpha}\right) - \alpha\right| \leq C\left\{\left(\frac{B_n^2\log^5 p}{n}\right)^{1/4}+\sqrt{\frac{B_n^2(\log p)^{3-2/q}}{n^{1-2/q}}}\right\},
$$
where $C$ is a constant depending only on, $q$, $b_1$, and $b_2$.
\end{theorem}
This theorem improves on the corresponding result obtained by \cite{K19b} by the fourth author. For bootstrap approximation, we focus on the Gaussian multiplier and empirical bootstraps for simplicity. 



\begin{theorem}[Bootstrap Approximation under Polynomial Moment Conditions]\label{cor: boot app pol mom cond}
Suppose that Condition M is satisfied and that Condition (\ref{eq: poly}) holds for all $i=1,\dots,n$ for some $q>2$.  Let $c_{1-\alpha}^B$ be the critical value obtained via either the empirical bootstrap or the Gaussian multiplier bootstrap. Then
\begin{equation*}
\left|\Pr\left(T_n > c_{1-\alpha}^B\right) - \alpha\right| \leq C\left\{\left(\frac{B_n^2\log^5(p n)}{n}\right)^{1/4}+\sqrt{\frac{B_n^2\log^{3-2/q}(p n)}{n^{1-2/q}}}\right\},
\end{equation*}
where $C$ is a constant depending only on $q$,  $b_1$, and $b_2$. 
\end{theorem}

This theorem improves on the error bound for the empirical bootstrap given in \cite{DZ17} under the polynomial moment condition. 

\section{Main Theoretical Arguments}

\subsection{Iterative Randomized Lindeberg Method}\label{sec: main arguments}
In this section, we derive a distributional approximation result, Theorem \ref{cor: max}, using a novel proof technique, which we call the iterative randomized Lindeberg method. We will use this result in Section \ref{sec: proofs} to prove our main results on the Gaussian and bootstrap approximations in high dimensions, as stated in Section \ref{sec: main results}.

Let $V_1,\dots,V_n, Z_1,\dots,Z_n$ be a sequence of independent random vectors in $\mathbb R^p$ such that $\Ep[V_{i j}] = \Ep[Z_{i j}] = 0$ for all $i = 1,\dots,n$ and $j = 1,\dots,p$, where $V_{i j}$ and $Z_{i j}$ denote the $j$th components of $V_{i j}$ and $Z_{i j}$, respectively. We will assume that these vectors obey the following conditions:

\medskip
\noindent
{\bf Condition V:} {\em There exists a constant $C_v>0$ such that for all $j = 1,\dots,p$, we have
$$
\frac{1}{n}\sum_{i=1}^n\Ep\left[V_{i j}^4 + Z_{i j}^4\right] \leq C_v B_n^2.
$$

}
\noindent
{\bf Condition P:} {\em There exists a constant $C_p\geq 1$ such that for all $i = 1,\dots,n$, we have
$$
\Pr\Big(\|V_i\|_{\infty}\vee\|Z_i\|_{\infty} > C_p B_n\log(p n)\Big) \leq 1/n^4.
$$

}
\medskip
\noindent
{\bf Condition B:} {\em There exists a constant $C_b>0$ such that for all $i = 1,\dots,n$, we have
$$
\Ep\Big[\|V_i\|_{\infty}^8 + \|Z_i\|_{\infty}^8\Big]\leq C_bB_n^8\log^8(p n).
$$

}
\medskip
\noindent
{\bf Condition A:} {\em 
There exist constants $C_a > 0$ and $\delta\geq0$ such that for all $(y,t)\in\mathbb R^p\times (0,\infty)$, we have
$$
\Pr\left(\frac{1}{\sqrt n}\sum_{i=1}^n Z_i \leq y + t \right) - \Pr\left(\frac{1}{\sqrt n}\sum_{i=1}^n Z_i \leq y\right) \leq C_a\left(t\sqrt{\log p}+\delta\right).
$$}

\noindent
Note that the constants $C_v$, $C_p$, $C_b$, and $C_a$ appearing in these conditions are not supposed to be dependent on their indices, e.g. $C_p$ here is not allowed to change with $p$; the indices are introduced with the only goal to differentiate between the constants. 

The following is the main result of this section:

\begin{theorem}[Distributional Approximation via Iterative Randomized Lindeberg Method]\label{cor: max}
Suppose that Conditions V, P, B, and A are satisfied. In addition, suppose that
\begin{equation}\label{eq: bn bounds 1}
\max_{1\leq j,k\leq p}\left| \frac{1}{\sqrt n}\sum_{i=1}^n(\Ep[V_{i j}V_{i k}] - \Ep[Z_{i j}Z_{i k}]) \right| \leq C_mB_n\sqrt{\log(p n)}
\end{equation}
and
\begin{equation}\label{eq: bn bounds 2}
\max_{1\leq j,k,l\leq p}\left| \frac{1}{\sqrt n}\sum_{i=1}^n (\Ep[V_{i j}V_{i k}V_{i l}] - \Ep[Z_{i j}Z_{i k}Z_{i l}]) \right| \leq C_m B_n^2\sqrt{\log^3(p n)}
\end{equation}
for some constant $C_m$. Then
\[
\sup_{y\in\mathbb R^p}\left| \Pr\left(\frac{1}{\sqrt n}\sum_{i=1}^n V_i \leq y\right) - \Pr\left(\frac{1}{\sqrt n}\sum_{i=1}^n Z_i \leq y\right) \right| 
 \le C\left(\left(\frac{B_n^2\log^5(p n)}{n}\right)^{1/4}+\delta\right),
\]
where $C$ is a constant depending only on $C_v$, $C_p$, $C_b$, $C_a$, and $C_m$.
\end{theorem}

\begin{remark}[On Sharpness of Theorem \ref{cor: max}]
We do not claim sharpness of Theorem \ref{cor: max} in the high-dimensional case $p\gg n$ (when $p$ is fixed, the theorem is not sharp in view of the classical Berry-Esseen bound). On one hand, classical Edgeworth expansions in the low-dimensional case suggest that conditions like \eqref{eq: bn bounds 2} should lead to better distributional approximation results than the corresponding Gaussian approximation results, which we do not observe in Theorem \ref{cor: max} since Theorem \ref{thm: gaussian approximation main} gives the same dependence on both $n$ and $p$ for the Gaussian approximation. On the other hand, to the best of our knowledge, there exist no analogs of  Edgeworth expansions in high dimensions. The question whether conditions like \eqref{eq: bn bounds 2} can be used to improve distributional approximations (relative to the Gaussian approximations) thus remains open.
\qed
\end{remark}

To prove this result, we will need additional notation. For all $\epsilon\in\{0,1\}^n$, define
\begin{equation}\label{eq: rho eps definition}
\varrho_{\epsilon} = \sup_{y\in\mathbb R^p}\left|\Pr\left(S_{n,\epsilon}^V \leq y\right) - \Pr\left(S_n^Z \leq y\right)\right|,
\end{equation}
where
$$
S_{n,\epsilon}^V = \frac{1}{\sqrt n}\sum_{i=1}^n (\epsilon_iV_i + (1 - \epsilon_i)Z_i)\quad\text{and}\quad S_n^Z = \frac{1}{\sqrt n}\sum_{i=1}^n Z_i.
$$
We will replace $\epsilon$ with a certain sequence of random vectors $\epsilon^0,\dots,\epsilon^D \in \{ 0,1 \}^n$,  independent of $V_1,\dots,V_n,Z_1,\dots,Z_n$, and derive  recursive bounds for $\rho_{\epsilon^d}$ for $d=0,\dots,D$, which lead to the desired bound in Theorem \ref{cor: max}. Such a sequence of random vectors $\epsilon^0,\dots,\epsilon^D \in \{ 0,1 \}^n$ is constructed as follow:
\begin{itemize}
\item Set $D = [4\log n]+1$ and initialize $\epsilon^0 = (1,\dots,1)$.
\item Let $U_1,\dots,U_D$ be a sequence of independent uniform $[0,1]$ random variables that are independent of $V_1,\dots,V_n,Z_1,\dots,Z_n$.
\item For $d=1,\dots,D$: conditionally on $\epsilon^{d-1}$ and $U_1,\dots,U_D$, set $\epsilon^d_i = 0$  if $\epsilon^{d-1}_i = 0$, and generate $\{\epsilon^d_i\}_{i\in I_{d-1}}$ with $I_{d-1} = \{ i=1,\dots, n : \epsilon^{d-1}_i = 1 \}$  as i.i.d.~Bernoulli($U_d$) random variables. 
\end{itemize}
It is not difficult to see that for each $d=1,\dots,D$, the random vector $\epsilon^d$ satisfies the following properties:
\begin{enumerate}
\item[(i)]  for all $i = 1,\dots,n$, $\epsilon^d_i = 0$ if $\epsilon^{d-1}_i = 0$, and 
\item[(ii)] for $I_{d-1} = \{i=1,\dots,n\colon \epsilon^{d-1}_i = 1\}$, the random variables $\{\epsilon^d_i\}_{i\in I_{d-1}}$ are exchangeable conditional on $\epsilon^{d-1}$ and satisfy
\begin{equation}\label{eq: recursive epsilon}
\Pr\left( \sum_{i\in I_{d-1}}\epsilon^d_i = s \mid \epsilon^{d-1} \right) = \frac{1}{|I_{d-1}| + 1},\quad\text{for all }s = 0,\dots,|I_{d-1}|.
\end{equation}
\end{enumerate}
Indeed, to see that (\ref{eq: recursive epsilon}) holds, observe that,  conditional on $\epsilon^{d-1}$ and $U_d$,  $\sum_{i \in I_{d-1}} \epsilon_i^d$ follows the binomial distribution with parameters $|I_{d-1}|$ and (success probability) $U_d$, so that 
\[
\begin{split}
\Pr\left( \sum_{i\in I_{d-1}}\epsilon^d_i = s \mid \epsilon^{d-1} \right) &=
\binom{|I_{d-1}|}{s}  \int_0^1 u^{s} (1-u)^{|I_{d-1}|-s} du \\
&=\binom{|I_{d-1}|}{s} \frac{s! (|I_{d-1}|-s)!}{(|I_{d-1}|+1)!}  = \frac{1}{|I_{d-1}|+1}. 
\end{split}
\]
Also, two properties (i) and (ii) ensure that $S_{n,\epsilon^d}^V-n^{-1/2}\sum_{i\notin I_{d-1}}Z_i$ is the randomized Lindeberg interpolant between $n^{-1/2}\sum_{i\in I_{d-1}}V_i$ and $n^{-1/2}\sum_{i\in I_{d-1}}Z_i$; see Lemma \ref{lem: rand lind} and the discussion at the beginning of Step 1 of the proof of Lemma \ref{lem: main}.


Further, for all $i=1,\dots,n$ and $j,k,l = 1,\dots,p$, define
$$
\mathcal E_{i, j k}^V = \Ep[V_{i j}V_{i k}], \ \mathcal E_{i, j k l}^V = \Ep[V_{i j} V_{i k}V_{i l}],
$$ 
$$
\mathcal E_{i, j k}^Z = \Ep[Z_{i j}Z_{i k}], \ \mathcal E_{i, j k l}^Z = \Ep[Z_{i j} Z_{i k}Z_{i l}].
$$
For all $n\geq 1$ and $d = 0,\dots,D$, let $\mathcal B_{n,1,d}$ and $\mathcal B_{n,2,d}$ be some strictly positive constants, and define the event $\mathcal A_d$ by 
\[
\begin{split}
\mathcal A_d &= \left \{ 
\max_{1\leq j,k\leq p}\left| \frac{1}{\sqrt n}\sum_{i=1}^n\epsilon^d_i(\mathcal E_{i, j k}^V - \mathcal E_{i, jk}^Z) \right| \leq \mathcal B_{n,1,d} \right \} \\
&\quad \bigcap 
\left \{ 
\max_{1\leq j,k,l\leq p}\left| \frac{1}{\sqrt n}\sum_{i=1}^n \epsilon_i^d (\mathcal E_{i,j k l}^V - \mathcal E_{i, j k l}^Z) \right| \leq \mathcal B_{n,2,d}
\right \}.
\end{split}
\]

The proof of Theorem \ref{cor: max} proceeds as follows. In Lemma \ref{lem: main} and Corollary \ref{cor: main}, we establish a recursive inequality for $\Ep[\varrho_{\epsilon^d}1\{\mathcal A_d\}]$, $d=0,\dots,D$. Next, we show in Lemma \ref{lem: closing} that $\Ep[\varrho_{\epsilon^D}1\{\mathcal A_D\}]$ is bounded by $1/n$. Then, we use an induction argument backward to derive a bound for $\Ep[\varrho_{\epsilon^0}1\{\mathcal A_0\}]$. Since $\epsilon^0_i=1$ for all $i$, this gives the claim of the theorem once we appropriately choose the constants $\mathcal B_{n,1,d}$ and $\mathcal B_{n,2,d}$. The proof of Lemma \ref{lem: main} is long and is given in Appendix \ref{sec: proof of main lemma} of the Supplemental Material.

The derivation of the recursive inequality is based on connecting $S_{n,\epsilon^d}^V$ with $S_n^Z$ by the randomized Lindeberg method originally developed by \cite{DZ17}. 
A similar approach was used in \cite{CCK13,CCK17} to connect $S_{n,\epsilon^0}^V$ with $G$, where the Slepian--Stein method was applied instead. Unlike the latter approach, the randomized Lindeberg method allows us to match the moments of $S_{n,\epsilon^d}^V$ and $S_n^Z$ up to the third order rather than the second order. This leads to improvement on the power of $\log(pn)$ factors. In addition, we incorporate a smoothing effect induced by $Z_i$ via Condition A into our argument. This along with the higher-order moment matching lead to improvement on the power of the sample size $n$.

\begin{lemma}\label{lem: main}
Suppose that Conditions V, P, B, and A are satisfied. Then for any $d = 0,\dots,D-1$ and any constant $\phi>0$ such that 
\begin{equation}\label{eq: phi restriction}
C_pB_n\phi \log^2(p n)\leq \sqrt n,
\end{equation}
we have on the event $\mathcal A_d$, 
\begin{align*}
\varrho_{\epsilon^d} &\lesssim \frac{\sqrt{\log p}}{\phi} + \delta + \frac{B_n^2\phi^4\log^5(p n)}{n^2} +  \left( \Ep[\varrho_{\epsilon^{d+1}}\mid \epsilon^d] + \frac{\sqrt{\log p}}{\phi}+\delta \right)\\
&\quad \times \left( \frac{\mathcal B_{n,1,d}\phi^2\log p}{\sqrt n} + \frac{\mathcal B_{n,2,d}\phi^3\log^2p}{n} + \frac{B_n^2\phi^4\log^3(p n)}{n} \right)
\end{align*}
up to a constant depending only on $C_v$, $C_p$, $C_b$, and $C_a$.
\end{lemma}

\begin{remark}[Choice of $\phi$]
We will choose $\phi$ to depend on $n$ via $n^{1/4}$ when applying this lemma.\qed
\end{remark}

\begin{corollary}\label{cor: main}
Suppose that all assumptions of Lemma \ref{lem: main} are satisfied. Then there exists a constant $K>0$ depending only on $C_v$, $C_p$, and $C_b$ such that for all $d = 0,\dots,D-1$, if $\mathcal B_{n,1,d+1} \geq \mathcal B_{n,1,d} + KB_n\log^{1/2}(p n)$ and $\mathcal B_{n,2,d+1} \geq \mathcal B_{n,2,d} + KB_n^2\log^{3/2}(pn)$, then for any constant $\phi>0$ satisfying \eqref{eq: phi restriction}, we have
\begin{align}
\Ep[\varrho_{\epsilon^d}1\{\mathcal A_d\}] &\lesssim \frac{\sqrt{\log p}}{\phi} + \delta+ \frac{B_n^2\phi^{4}\log^5(p n)}{n^2} + \left( \Ep[\varrho_{\epsilon^{d+1}} 1\{\mathcal A_{d+1}\}] + \frac{\sqrt{\log p}}{\phi} +\delta\right)  \nonumber\\
&\quad\times \left( \frac{\mathcal B_{n,1,d}\phi^2\log p}{\sqrt n} + \frac{\mathcal B_{n,2,d}\phi^3\log^2p}{n} + \frac{B_n^2\phi^4\log^3(p n)}{n} \right)\label{eq: max induction}
\end{align}
up to a constant depending only on $C_v$, $C_p$, $C_b$, and $C_a$.
\end{corollary}

\begin{proof}
Since we assume throughout the paper that $p\geq 2$, the conclusion is trivial if $\phi<1$. We will therefore assume in the proof that $\phi\geq 1$. In turn, $\phi\geq 1$ together with \eqref{eq: phi restriction} imply that
\begin{equation}\label{eq: lazy condition 1}
C_p B_n \log^2(p n)\leq \sqrt n.
\end{equation}
This condition will be useful in the proof. 

Fix $d = 0,\dots,D-1$. Then, given that $\mathcal A_d$ depends only on $\epsilon^d$, we have by Lemma \ref{lem: main} that
\begin{align*}
\Ep[\varrho_{\epsilon^d}1\{\mathcal A_d\}] &\lesssim \frac{\sqrt{\log p}}{\phi} + \delta + \frac{B_n^2\phi^{4}\log^5(pn)}{n^2} +  \left( \Ep[\varrho_{\epsilon^{d+1}} 1\{\mathcal A_{d}\}] + \frac{\sqrt{\log p}}{\phi} + \delta\right)\\
&\quad \times \left( \frac{\mathcal B_{n,1,d}\phi^2\log p}{\sqrt n} + \frac{\mathcal B_{n,2,d}\phi^3\log^2p}{n} + \frac{B_n^2\phi^4\log^3(p n)}{n} \right)
\end{align*}
up to a constant depending only on $C_v$, $C_p$, $C_b$, and $C_a$. Thus, given that \eqref{eq: phi restriction} implies that $\sqrt{\log p}/\phi\geq 1/n$, the conclusion of the corollary will follow if we can show that
\begin{equation}\label{eq: event switch}
\Ep[\varrho_{\epsilon^{d+1}}1\{\mathcal A_d\}] \leq \Ep[\varrho_{\epsilon^{d+1}}1\{\mathcal A_{d+1}\}] + 4/n.
\end{equation}
To this end, we first observe that , as $\varrho_{\epsilon^{d+1}}\in[0,1]$, 
\begin{equation}
\begin{split}
\Ep[\varrho_{\epsilon^{d+1}}1\{\mathcal A_d\}]
&= \Ep[\varrho_{\epsilon^{d+1}}1\{\mathcal A_d\}1\{\mathcal A_{d+1}\}] + \Ep[\varrho_{\epsilon^{d+1}}1\{\mathcal A_d\}(1 - 1\{\mathcal A_{d+1}\})]\\
& \leq \Ep[\varrho_{\epsilon^{d+1}}1\{\mathcal A_{d+1}\}]
+ \Ep[1\{\mathcal A_d\}(1 - 1\{\mathcal A_{d+1}\})]\\
&=\Ep[\varrho_{\epsilon^{d+1}}1\{\mathcal A_{d+1}\}] + \underbrace{\Pr(\mathcal A_d) - \Pr(\mathcal A_d \cap \mathcal A_{d+1})}_{=\Pr(\mathcal A_d) \left ( 1-\Pr(\mathcal A_{d+1}\mid \mathcal A_d) \right)} \\
& \leq \Ep[\varrho_{\epsilon^{d+1}}1\{\mathcal A_{d+1}\}]
+ 1 - \Pr(\mathcal A_{d+1}\mid \mathcal A_d).
\end{split}
\label{eq: event switch proof}
\end{equation}
Moreover, by Lemma \ref{lem: exponential inequality} in the Supplemental Material, for all $j,k = 1,\dots,p$ and $t > 0$, we have
\begin{align*}
&\Pr\left(\left| \frac{1}{\sqrt n}\sum_{i=1}^n\epsilon_i^{d+1}(\mathcal E_{i,jk}^V - \mathcal E_{i,jk}^Z) \right| > \left| \frac{1}{\sqrt n}\sum_{i=1}^n\epsilon_i^{d}(\mathcal E_{i,jk}^V - \mathcal E_{i,jk}^Z) \right| + t \mid \epsilon^d \right)\\
&\quad \leq 2\exp\left(-\frac{n t^2}{32\sum_{i=1}^n(\mathcal E_{i,jk}^V - \mathcal E_{i,jk}^Z)^2}\right) \leq 2\exp\left(-\frac{t^2}{128 C_v B_n^2}\right),
\end{align*}
where the second inequality follows from Condition V. Applying this inequality with $t = 8B_n\sqrt{6C_v\log(p n)}$ and using the fact that
$$
\max_{1\leq j,k\leq p}\left|\frac{1}{\sqrt n}\sum_{i=1}^n \epsilon_i^d(\mathcal E_{i,jk}^V - \mathcal E_{i,jk}^Z)\right| \leq \mathcal B_{n,1,d} \quad \text{on $\mathcal A_d$},
$$
 we have by the union bound that for any $\mathcal B_{n,1,d+1} \geq \mathcal B_{n,1,d} + t$,
$$
\Pr\left(\max_{1\leq j,k\leq p}\left|\frac{1}{\sqrt n}\sum_{i=1}^n \epsilon_i^{d+1}(\mathcal E_{i,jk}^V - \mathcal E_{i,jk}^Z)\right| > \mathcal B_{n,1,d+1}\mid \mathcal A_d\right) \leq \frac{2p^2}{(pn)^3}\leq \frac{2}{n}.
$$
In addition, for all $i=1,\dots,n$ and $j,k,l=1,\dots,p$, setting $\tilde V_i = 1\{\|V_i\|_{\infty}\leq C_p B_n\log(pn)\}$, we have that
\begin{equation}
\begin{split}
|\mathcal E_{i,jkl}^V| &\leq \Ep[|V_{ij}V_{ik}V_{il}|] = \Ep\Big[\tilde V_i |V_{ij}V_{ik}V_{il}|\Big] + \Ep\Big[(1 - \tilde V_i) |V_{ij}V_{ik}V_{il}|\Big]\\\
&\leq C_pB_n\log(p n)\Ep[|V_{ij}V_{ik}|] + (\Ep[1-\tilde V_i])^{1/2}(\Ep[\|V_i\|_{\infty}^6])^{1/2}\\
&\leq C_pB_n\log(p n)\Ep[|V_{ij}V_{ik}|] + C_b^{3/8}B_n^3\log^{3}(p n)/n^2
\end{split}
\label{eq: details moment calculations}
\end{equation}
and similarly
$$
|\mathcal E_{i,jkl}^Z| \leq C_pB_n\log(p n)\Ep[|Z_{ij}Z_{ik}|] + C_b^{3/8}B_n^3\log^{3}(p n)/n^2
$$
by Conditions P and B. Hence, by Condition V and \eqref{eq: lazy condition 1}, there exists a constant $C$ depending only on $C_v$, $C_p$, and $C_b$ such that
$$
\frac{32}{n}\sum_{i=1}^n(\mathcal E_{i,jkl}^V - \mathcal E_{i,jkl}^Z)^2 \leq C B_n^4\log^2(p n).
$$
Thus, by the same argument as above, for all $j,k,l=1,\dots,p$ and $t>0$,
\begin{align*}
&\Pr\left(\left| \frac{1}{\sqrt n}\sum_{i=1}^n\epsilon_i^{d+1}(\mathcal E_{i,jkl}^V - \mathcal E_{i,jkl}^Z) \right| > \left| \frac{1}{\sqrt n}\sum_{i=1}^n\epsilon_i^{d}(\mathcal E_{i,jkl}^V - \mathcal E_{i,jkl}^Z) \right| + t \mid \epsilon^d \right)\\
&\quad \leq 2\exp\left(-\frac{n t^2}{32\sum_{i=1}^n(\mathcal E_{i,jkl}^V - \mathcal E_{i,jkl}^Z)^2}\right) \leq 2\exp\left(-\frac{t^2}{C B_n^4\log^2(p n)}\right).
\end{align*}
Applying this inequality with $t = \sqrt{3C}B_n^2\log^{3/2}(p n)$ shows that for any $\mathcal B_{n,2,d+1} \geq \mathcal B_{n,2,d} + t$, we have
$$
\Pr\left(\max_{1\leq j,k,l\leq p}\left|\frac{1}{\sqrt n}\sum_{i=1}^n \epsilon_i^{d+1}(\mathcal E_{i,jkl}^V - \mathcal E_{i,jkl}^Z)\right| > \mathcal B_{n,2,d+1} \mid \mathcal A_d\right) \leq \frac{2p^3}{(p n)^3}\leq \frac{2}{n}.
$$
Thus,
$
1 - \Pr(\mathcal A_{d+1}\mid \mathcal A_d) \leq 4/n,
$
which in combination with \eqref{eq: event switch proof} implies \eqref{eq: event switch} and completes the proof.
\end{proof}

\begin{lemma}\label{lem: closing}
For any constant $\phi > 0$ such that \eqref{eq: phi restriction} holds, we have $\Ep[\varrho_{\epsilon^D}1\{\mathcal A_D\}] \leq 1/n$. 
\end{lemma}
\begin{proof}
Recall that $D = [4\log n] + 1$ and note that $\varrho_{\epsilon^D} = 0$ if $\epsilon^D = (0,\dots,0)'$. Moreover, by Markov's inequality,
\begin{align*}
&\Pr(\epsilon^D \neq (0,\dots,0)') =  \Pr\left(\sum_{i=1}^n\epsilon^D_i \geq 1\right) \leq \Ep\left[\sum_{i=1}^n\epsilon^D_i\right] = \Ep\left[\Ep\left[\sum_{i=1}^n\epsilon^D_i \mid \sum_{i=1}^n\epsilon^{D-1}_i\right]\right] \\
&\quad    = \Ep\left[\frac{1}{2}\sum_{i=1}^n\epsilon^{D-1}_i\right]
 = \dots = \Ep\left[\frac{1}{2^D}\sum_{i=1}^n\epsilon^0_i\right] = \frac{n}{2^D} \leq \frac{n}{2^{4\log n}} \leq \frac{1}{n},
\end{align*}
where the equalities on the second line follow from \eqref{eq: recursive epsilon}. Hence,
$$
\Ep[\varrho_{\epsilon^D}1\{\mathcal A_D\}] \leq \Ep[\varrho_{\epsilon^D}] \leq \Pr(\epsilon^D \neq (0,\dots,0)') \leq 1/ n,
$$
as desired. 
\end{proof}
\begin{proof}[Proof of Theorem \ref{cor: max}]
Throughout the proof, we will assume that
\begin{equation}\label{eq: natural bound}
C_p^4B_n^2\log^5(p n)\leq n
\end{equation}
since otherwise the conclusion  of the theorem is trivial.

Let $K$ be the constant from Corollary \ref{cor: main} and for all $d = 0,\dots,D$, define
\begin{equation}\label{eq: fancy b bounds}
\mathcal B_{n,1,d}= C_1(d+1)B_n\log^{1/2}(p n) \quad \text{and} \quad \mathcal B_{n,2,d}=C_1(d+1)B_n^2\log^{3/2}(p n),
\end{equation}
where $C_1 = C_m + K$, so that $\mathcal A_0$ holds by \eqref{eq: bn bounds 1} and \eqref{eq: bn bounds 2} and, in addition, the requirements of Corollary \ref{cor: main} on $\mathcal B_{n,1,d}$ and $\mathcal B_{n,2,d}$ also hold.

Now, for all $d = 0,\dots,D$, define
$$
f_d = \inf\left\{x\geq 1\colon \Ep[\varrho_{\epsilon^d}1\{\mathcal A_d\}] \leq  x\left(\left(\frac{B_n^2\log^5(p n)}{n}\right)^{1/4}+\delta\right)\right\}.
$$
Note that $f_d<\infty$ because $\varrho_{\epsilon^d}\leq1$. Then, for all $d = 0,\dots,D-1$, apply Corollary \ref{cor: main} with
$$
\phi = \phi_d = \frac{n^{1/4}}{B_n^{1/2}\log^{3/4}(p n)((d+1)f_{d+1})^{1/3}},
$$
which satisfies the required condition \eqref{eq: phi restriction} since we assume \eqref{eq: natural bound}. Since
\begin{align*}
\frac{B_n^2\phi_d^4\log^{5}(pn)}{n^2}
& \leq\frac{\log^{2}(p n)}{n} 
\leq\frac{\log^{1/4}(p n)}{n^{1/4}}
\leq \frac{C_p B_n^{1/2} \log^{1/4}(p n)}{n^{1/4}}\\
& \leq\frac{C_p\sqrt{\log p}}{\phi_d} \leq C_p((d+1)f_{d+1})^{1/3}\left(\frac{B_n^2\log^5(p n)}{n}\right)^{1/4}, \\
\frac{\mathcal B_{n,1,d}\phi_d^2\log p}{\sqrt n} &\leq \frac{C_1(d+1)}{((d+1)f_{d+1})^{2/3}}, \quad \text{and} \\
\frac{\mathcal B_{n,2,d}\phi^3_d\log^2p}{n} &\bigvee \frac{B_n^2\phi_d^4\log^3(p n)}{n} \leq \frac{C_1\vee 1}{f_{d+1}},
\end{align*}
we have by  Corollary \ref{cor: main}
$$
\Ep[\rho_{\epsilon^d}1\{\mathcal A_d\}] \leq C_2\Big(f_{d+1}^{2/3} + (d+1)^{2/3} + 1\Big)\left(\left(\frac{B_n^2\log^5(p n)}{n}\right)^{1/4}+\delta\right)
$$
for some constant $C_2\geq 1$ depending only on $C_v$, $C_p$, $C_b$, $C_a$, and $C_m$. Hence,
$$
f_d \leq C_2\Big(f_{d+1}^{2/3} + (d+1)^{2/3} + 1\Big),\quad\text{for all }d = 0,\dots,D-1.
$$
Here, we have $f_D = 1$ by Lemma \ref{lem: closing} since $B_n\geq 1$ by assumption. Therefore, by a simple induction argument, we conclude that there exists a constant $C\geq 1$ depending only on $C_2$ such that
$$
f_d \leq C(d+1),\quad\text{for all }d = 0,\dots,D.
$$
In particular, it follows that
$$
\varrho_{\epsilon^0}1\{\mathcal A_0\} = \Ep[\varrho_{\epsilon^0}1\{\mathcal A_0\}] \leq 
C\left(\left(\frac{B_n^2\log^5(p n)}{n}\right)^{1/4}+\delta\right).
$$
Since $\mathcal A_0$ holds by construction, so that $1\{\mathcal A_0\} = 1$, the desired bound follows by combining this inequality and the definition of $\varrho_{\epsilon^0}$.
\end{proof}

\subsection{Stein Kernels and Gaussian Approximation}\label{sec: stein kernels}
Let $C_b^2(\mathbb R^p)$ be the class of twice continuously differentiable functions $\varphi$ on $\mathbb R^p$ such that $\varphi$ and all its partial derivatives up to the second order are bounded where $p\geq 2$. 
Let $V$ be a centered random vector in $\mathbb R^p$ and assume that there exists a measurable function $\tau:\mathbb R^p\to\mathbb R^{p\times p}$ such that
$$
\sum_{j=1}^p\Ep[\partial_j \varphi(V)V_j] = \sum_{j,k=1}^p\Ep[\partial_{jk}\varphi(V)\tau_{jk}(V)]
$$
for all $\varphi\in C_b^2(\mathbb R^p)$. This function $\tau$ is called a \textit{Stein kernel} for the random vector $V$.  Also, let $Z$ be a centered Gaussian random vector in $\mathbb R^p$ with covariance matrix $\Sigma$.
\begin{theorem}[Gaussian Approximation via Stein Kernels]\label{thm: stein kernel}
If $\Sigma_{jj}\geq c$ for all $j=1,\dots,p$ and some constant $c>0$, then
$$
\sup_{y\in\mathbb R^p}\Big|\Pr(V\leq y) - \Pr(Z\leq y)\Big| \leq C\Big(\Delta \log^2 p\Big)^{1/2},
$$
where $C$ is a constant depending only on $c$ and
$
\Delta = \Ep\left[\max_{1\leq j,k\leq p}|\tau_{jk}(V) - \Sigma_{jk}|\right].
$
\end{theorem}
\begin{remark}
This theorem improves upon Proposition 4.1 in \cite{K19a}, which shows that
$$
\sup_{y\in\mathbb R^p}\Big|\Pr(V\leq y) - \Pr(Z\leq y)\Big| \leq C\Big(\Delta \log^2 p\Big)^{1/3}
$$
under the same conditions.\qed
\end{remark}

Theorem \ref{thm: stein kernel} is proven in Appendix \ref{sec: proof of kernel result} of the Supplemental Material. It has two important corollaries. 
The first  is Proposition \ref{coro:g-g-comparison}, a sharp Gaussian-to-Gaussian comparison inequality stated in Section \ref{sec: main results}:

\begin{proof}[Proof of Proposition \ref{coro:g-g-comparison}]
If $V$ is a centered Gaussian random vector, then by the multivariate Stein identity, its Stein kernel coincides with its covariance matrix. Hence, Theorem \ref{thm: stein kernel} immediately implies the conclusion of Proposition \ref{coro:g-g-comparison}.
\end{proof}

Second, combining Theorem \ref{thm: stein kernel} with Lemma 4.6 in \cite{K19b} gives the following result:
\begin{corollary}[Multiplier-Bootstrap-to-Gaussian Comparison]\label{coro:beta-comparison}
Let $a_1,\dots,a_n$ be vectors  in $\mathbb R^p$ such that
$$
\min_{1\leq j\leq p}\frac{1}{n}\sum_{i=1}^n a_{ij}^2 \geq c\quad\text{and}\quad\max_{1\leq j\leq p}\frac{1}{n}\sum_{i=1}^na_{ij}^4 \leq B^2
$$
for some constants $c,B>0$. 
Also, let $\varepsilon_1,\dots,\varepsilon_n$ be independent $N(0,1)$ random variables. 
Moreover, for some constants $\alpha,\beta>0$, let $e_1,\dots,e_n$ be independent standardized Beta$(\alpha,\beta)$ random variables so that
\begin{equation}\label{eq: beta distribution mean variance}
\Ep[e_i]=0 \quad\text{and}\quad \Ep[e_i^2]=1,\quad\text{for all }i=1,\dots,n.
\end{equation}
Then, for the random vectors
$$
V = \frac{1}{\sqrt n}\sum_{i=1}^n e_i a_i\quad\text{and}\quad Z = \frac{1}{\sqrt n}\sum_{i=1}^n \varepsilon_i a_i
$$
we have 
\begin{equation}\label{eq: beta distribution comparison}
\sup_{y\in\mathbb R^p}\Big|\Pr(V\leq y) - \Pr(Z\leq y)\Big| \leq C\left(\frac{B^2\log^5 p}{n}\right)^{1/4},
\end{equation}
where $C$ is a constant depending only on $c$, $\alpha$ and $\beta$.
\end{corollary}
\begin{proof}
Recall that $\eta\sim\text{Beta}(\alpha,\beta)$ has  density function $f_{\alpha,\beta}(x) \propto x^{\alpha - 1}(1-x)^{\beta - 1}$ for $x\in[0,1]$, mean $\mu = \alpha/(\alpha + \beta)$, and variance $\sigma^2 = \alpha\beta/((\alpha + \beta)^2(\alpha + \beta + 1))$. 
By definition, the common distribution of the random variables $e_1,\dots,e_n$ equals that of $(\eta - \mu)/\sigma$.

Define
$$
\tau(x) = -\frac{\int_{-\mu/\sigma}^xsf(s)ds}{f(x)} = \frac{\int_x^{(1-\mu)/\sigma}sf(s)ds}{f(x)} \quad \text{for} \ x\in\left(-\frac{\mu}{\sigma},\frac{1-\mu}{\sigma}\right),
$$
where $f(x)=\sigma f_{\alpha,\beta}(\sigma x + \mu)$ for $x\in\big (-\frac{\mu}{\sigma},\frac{1-\mu}{\sigma}\big)$
is the density function of $(\eta - \mu)/\sigma$. From L'Hospital's rule, there exists a constant $C_1$ depending only on $\alpha$ and $\beta$ such that
$
|\tau(x)|\leq C_1$ for all $x\in\big(-\frac{\mu}{\sigma}, \frac{1-\mu}{\sigma}\big)$. Also, by integration by parts, 
$
\Ep[e_1\varphi(e_1)]=\Ep[\varphi'(e_1)\tau(e_1)]
$
for any continuously differentiable function $\varphi\colon\mathbb R\to\mathbb R$. Then, by Lemma 4.6 in \cite{K19b}, a Stein kernel $\tau^V$ for the random vector $V$ satisfies
$$
\Ep\left[ \max_{1\leq j,k\leq p}\left|\tau_{jk}^V(V) - \frac{1}{n}\sum_{i=1}^n a_{ij}a_{ik} \right| \right] \leq C_2\sqrt{\frac{\log p}{n}}\times\max_{1\leq j\leq p}\sqrt{\frac{1}{n}\sum_{i=1}^n a_{i j}^4}
$$
for some constant $C_2$ depending only on $C_1$. The desired conclusion  \eqref{eq: beta distribution comparison} follows from combining this bound with Theorem \ref{thm: stein kernel} and observing that $\Ep[Z_{j}Z_{k}] =n^{-1}\sum_{i=1}^na_{ij}a_{ik}$ for all $j,k=1,\dots,p$.
\end{proof}

\section{Proofs of Theorems \ref{thm: gaussian approximation main} and \ref{cor: rejection probabilities}}\label{sec: proofs}
In this section, we provide proofs of Theorems \ref{thm: gaussian approximation main} and \ref{cor: rejection probabilities}. Proofs of Theorems \ref{thm: gauss and rademacher} and \ref{thm: infinitesimal factors} will be given in Appendices \ref{sec: proof of theorem 23} and \ref{sec: proof of theorem 24} of the Supplemental Material. To simplify notation, we write
\[
\delta_n=\left(\frac{B_n^2\log^5(pn)}{n}\right)^{1/4}\quad\text{and}\quad
\upsilon_n=\sqrt{\frac{B_n^2\log^3(pn)}{n}}.
\]

Our proof strategy for Theorems \ref{thm: gaussian approximation main} and \ref{cor: rejection probabilities} is summarized as follows. First, we consider the multiplier bootstrap statistic $T_n^*$ with the weights $e_i$ constructed from the standardized Beta($\alpha$,$\beta$) distribution and parameters $\alpha$ and $\beta$ chosen so that $\Ep[e_i^3] = 1$. Thanks to Corollary \ref{coro:beta-comparison} and Proposition \ref{coro:g-g-comparison}, we have Gaussian approximation to this statistic with the rate $\delta_n$. This implies that 
Condition A in Section \ref{sec: main arguments} is satisfied with $Z_i=e_i(X_i-\bar X_n)$ and $\delta=\delta_n$ due to the Gaussian anti-concentration inequality in Lemma \ref{lem: anticoncetration} of the Supplemental Material. In turn, the latter allows us to invoke Theorem \ref{cor: max}, which gives the approximation to $T_n$ by $T_n^*$ with the rate $\delta_n$. (Note that having $\Ep[e_i^3]=1$ is important here since otherwise Theorem \ref{cor: max} would give a slower approximation rate.) Combining this result with the aforementioned Gaussian approximation for $T_n^*$, we obtain the Gaussian approximation for $T_n$ with the rate $\delta_n$. This is done in Lemma \ref{coro: gaussian approximation} and gives Theorem \ref{thm: gaussian approximation main}.

Second, we consider the empirical bootstrap statistic $T_n^*$. Since we now have the Gaussian approximation for $T_n$ with the rate $\delta_n$, it follows that Condition A is satisfied with $Z_i=X_i$ and $\delta=\delta_n$. Hence, applying Theorem \ref{cor: max} with $V_i=X^*_i$ and $Z_i=X_i$, we can verify the empirical bootstrap approximation for $T_n$ with the rate $\delta_n$. This is done in Lemma \ref{thm: empirical bootstrap} and gives one part of Theorem \ref{cor: rejection probabilities}.

Third, we consider the multiplier bootstrap statistic $T_n^*$ with arbitrary weights $e_i$  satisfying \eqref{eq: multiplier bootstrap simplification}. By choosing parameters $\alpha$ and $\beta$ appropriately, we can match the first three moments of these weights by weights constructed from the standardized  Beta($\alpha$,$\beta$) distribution. Thus, yet another application of Theorem \ref{cor: max} allows us to link the distribution of any multiplier bootstrap statistic to the distribution of the multiplier bootstrap statistic with weights constructed from the standardized  Beta($\alpha$,$\beta$) distribution and further, via Corollary \ref{coro:beta-comparison} and Proposition \ref{coro:g-g-comparison}, to the Gaussian distribution. This leads to the Gaussian approximation for the multiplier bootstrap statistic $T_n^*$ with the rate $\delta_n$. This is done in Lemma \ref{eq: second order matching multiplier bootstrap ks metric} and gives the other part of Theorem \ref{cor: rejection probabilities}.

Before proceeding to the main body of the proofs, we present a few auxiliary results.
\begin{lemma}\label{lem: subgaussian growth}
Suppose that Condition E is satisfied. Then
\begin{equation}\label{eq: subgaussian bound on x}
\max_{1\leq i\leq n}\|X_i\|_{\infty} \leq 5B_n\log(p n)
\end{equation}
with probability at least $1 - 1/(2n^4)$. In addition,
$$
\max_{1\leq i\leq n}\Ep\left[\|X_i\|_{\infty}^8\right] \leq CB_n^8\log^8(p n),
$$
where $C$ is a universal constant.
\end{lemma}
\begin{proof}
By the union bound, Markov's inequality, and Condition E, we have for any $x>0$ that
\begin{align*}
&\Pr\left(\max_{1\leq i\leq n}\max_{1\leq j\leq p}|X_{i j}| > x\right)
 \leq pn \max_{1\leq i\leq n}\max_{1\leq j\leq p} \Pr(|X_{i j}| > x)\\
&\quad  
\leq pn \max_{1\leq i\leq n}\max_{1\leq j\leq p} \frac{\Ep[\exp(|X_{i j}| / B_n)]}{\exp(x/B_n)} \leq 2pn \exp(-x/B_n). 
\end{align*}
Substituting here $x = 5B_n\log(p n)$ gives the first asserted claim. The second asserted claim follows from combining Condition E, inequalities on page 95 in \cite{VW96}, and Lemma 2.2.2 in \cite{VW96}.
\end{proof}


\begin{lemma}\label{lem: all conditions}
Suppose that Conditions E and M are satisfied and set $\tilde X_i = X_i - \bar X_n$ for all $i = 1,\dots,n$. 
Then there exist a universal constant $c\in(0,1]$ and constants $C>0$ and $n_0\in\mathbb N$ depending only on $b_1$ and $b_2$ such that for all $n\geq n_0$, if the inequality
\begin{equation}\label{eq: bn restriction once again}
B_n^2\log^5(p n)\leq c n
\end{equation}
holds, then the following events hold jointly with probability at least $1 - 1/n-3\upsilon_n$:
\begin{align}
&\frac{b_1^2}{2}\leq \frac{1}{n}\sum_{i=1}^n \tilde X_{i j}^2\quad \text{and} \quad \frac{1}{n}\sum_{i=1}^n \tilde X_{i j}^4 \leq 2B_n^2 b_2^{2}, \quad \text{for all }j=1,\dots,p,
\label{eq: upper and lower bounds for second moment} \\
&\max_{1\leq j,k\leq p}\left|\frac{1}{\sqrt n}\sum_{i=1}^n(\tilde X_{i j}\tilde X_{i k} - \Ep[X_{ij}X_{i k}])\right| \leq CB_n\sqrt{\log(p n)}, \label{eq: second order deviation} \\
&\max_{1\leq j,k,l\leq p}\left|\frac{1}{\sqrt n}\sum_{i=1}^n(\tilde X_{i j}\tilde X_{i k}\tilde X_{i l} - \Ep[X_{ij}X_{i k}X_{i l}])\right| \leq CB_n^2\sqrt{\log^3(p n)}. \label{eq: third order deviation}
\end{align}
\end{lemma}

The proof of this lemma is rather standard but long, and so is deferred to Appendix \ref{sec: proof lemma 52} of the Supplemental Material.

\begin{lemma}\label{coro: gaussian approximation}
Suppose that Conditions E and M are satisfied. Then
\begin{equation}\label{eq: gaussian approximation ks metric}
\sup_{x\in\mathbb{R}}|\Pr(T_n\leq x)-\Pr(T^G_n\leq x)|\leq C\left(\frac{B_n^2\log^5(pn)}{n}\right)^{1/4},
\end{equation}
where $C$ is a constant depending only on $b_1$ and $b_2$.  
\end{lemma}

\begin{proof}
Without loss of generality, we may assume that \eqref{eq: bn restriction once again} holds and that $n$ is large enough so that $n\geq n_0$ for $n_0$ from Lemma \ref{lem: all conditions}, since otherwise the conclusion of the lemma is trivial by taking $C$ large enough. This will justify an application of Lemma \ref{lem: all conditions} when needed. In addition, by again taking $C$ large enough, we may assume that $1/n^4 + 2/n + 3v_n < 1$.

Let $\mathcal{A}_n$ be the event that \eqref{eq: subgaussian bound on x} and \eqref{eq: upper and lower bounds for second moment}--\eqref{eq: third order deviation} hold jointly. By Lemmas \ref{lem: subgaussian growth} and \ref{lem: all conditions}, $\Pr(\mathcal A_n)\geq 1 - 1/(2n^4) - 1/n -3v_n > 0$. Further, let $e_1,\dots,e_n$ be independent standardized Beta$(1/2,3/2)$ random variables, standardized in such a way that they have mean zero and unit variance (cf. Corollary \ref{coro:beta-comparison}), that are independent of $X_{1:n} = (X_1,\dots,X_n)$. It is not difficult to check that  $\Ep[e_i^3]=1$ for all $i=1,\dots,n$. 

Let $T_n^*$ be the multiplier bootstrap statistic with weights $e_1,\dots,e_n$.  
Condition on $X_{1:n}$ such that $\mathcal A_n$ holds. 
Then, by Corollary \ref{coro:beta-comparison} and the definition of $\mathcal A_n$, we have
\begin{equation}\label{beta-to-gauss}
\sup_{y\in\mathbb{R}^p}\left|\Pr\left( \frac{1}{\sqrt n}\sum_{i=1}^n e_i(X_i - \bar X_n) \leq y \mid X_{1:n} \right)-\Pr(\hat{G}\leq y\mid X_{1:n})\right|\leq C_1\delta_n, 
\end{equation}
while by Proposition \ref{coro:g-g-comparison}, we have
\begin{equation*}
\sup_{y\in\mathbb{R}^p}|\Pr(\hat{G}\leq y\mid X_{1:n})-\Pr(G\leq y)|\leq C_2\delta_n,
\end{equation*}
where  $C_1$ and $C_2$ are constants depending only on $b_1$ and $b_2$. 

Next, we shall invoke Theorem \ref{cor: max} to  compare the distribution of $T_n$ with the conditional distribution of $T_n^*$. Formally, let $Y_1,\dots,Y_n$ be independent copies of $X_1,\dots,X_n$ that are independent of $X_{1:n}$, and define $T_n'$ by $T_n$ with $X_i$'s replaced by $Y_i$'s. Then, $\Pr (T_n \le x) = \Pr (T_n' \le x \mid X_{1:n})$. Condition on $X_{1:n}$ such that $\mathcal A_n$ holds and apply Theorem \ref{cor: max} with $V_i = Y_i$ and $Z_i = e_i \tilde X_i$ for all $i=1,\dots,n$. Since $\Ep[e_i]=0$ and $\Ep[e_i^2] = \Ep[e_i^3]=1$ for all $i=1,\dots,n$, it is not difficult to see from the definition of $\mathcal A_n$ that Conditions V, P, and B, as well as inequalities (\ref{eq: bn bounds 1}) and (\ref{eq: bn bounds 2}) of Theorem \ref{cor: max} are satisfied with appropriate constants $C_v$, $C_p$, $C_b$, and $C_m$ that depend only on $b_1,b_2$. 
It remains to verify Condition A in Theorem \ref{cor: max}. 
Observe that for  any $y\in\mathbb R^p$ and $t>0$, 
\begin{equation}
 \label{beta-anti}
\begin{split}
&\Pr\left(\frac{1}{\sqrt n}\sum_{i=1}^n e_i(X_i - \bar X_n) \leq y + t \mid X_{1:n}\right)  \leq \Pr\left(\hat G \leq y+t \mid X_{1:n}\right) + C_1\delta_n \quad (\text{by (\ref{beta-to-gauss})}) \\
&\quad  \leq \Pr\left(\hat G \leq y\mid X_{1:n}\right) + K_1t\sqrt{\log p} + C_1\delta_n \quad (\text{by Lemma \ref{lem: anticoncetration} and (\ref{eq: upper and lower bounds for second moment})}) \\
&\quad  \leq\Pr\left(\frac{1}{\sqrt n}\sum_{i=1}^n e_i(X_i - \bar X_n) \leq y \mid X_{1:n}\right) 
+K_1t\sqrt{\log p} + 2C_1\delta_n, \quad (\text{by (\ref{beta-to-gauss})})
\end{split}
\end{equation}
where $K_1>0$ is a constant depending only on $b_1$. Thus, applying Theorem \ref{cor: max}, we conclude that
\begin{equation*}
\sup_{x\in\mathbb{R}}|\Pr(T_n\leq x)-\Pr(T^*_n\leq x\mid X_{1:n})| = \sup_{x\in\mathbb{R}}|\Pr(T_n'\leq x \mid X_{1:n})-\Pr(T^*_n\leq x\mid X_{1:n})|\leq C_3\delta_n
\end{equation*}
 for some constant $C_3$ depending only on $b_1$ and $b_2$. The asserted claim follows from these bounds via the triangle inequality by noting that the left-hand side of \eqref{eq: gaussian approximation ks metric} is non-stochastic, so that if \eqref{eq: gaussian approximation ks metric} holds with strictly positive probability (recall that $\Pr(\mathcal A_n)>0$), then it holds with probability one.
\end{proof}

\begin{lemma}\label{thm: empirical anticoncentration}
Suppose that Conditions E and M are satisfied. Then for any $y\in\mathbb R^p$ and $t > 0$,
\begin{align*}
&\Pr\left(\frac{1}{\sqrt n}\sum_{i=1}^n X_i \leq y+t\right) - \Pr\left(\frac{1}{\sqrt n}\sum_{i=1}^n X_i \leq y\right) \leq C\left(t\sqrt{\log p} + \left(\frac{B_n^2\log^5(p n)}{n}\right)^{1/4}\right),
\end{align*}
where $C$ is a constant depending only on $b_1$ and $b_2$.
\end{lemma}

\begin{proof}
Fix $y\in\mathbb R^p$ and $t>0$. Then for some constant $C$ depending only on $b_1$ and $b_2$,
\begin{align*}
\Pr\left(\frac{1}{\sqrt n}\sum_{i=1}^n X_i \leq y + t\right) &\leq \Pr(G \leq y + t) + C\delta_n
\leq \Pr(G \leq y ) + Ct\sqrt{\log p} + C\delta_n\\
&\leq \Pr\left(\frac{1}{\sqrt n}\sum_{i=1}^n X_i \leq y\right) + Ct\sqrt{\log p} + 2C\delta_n,
\end{align*}
where the first and the third inequalities follow from Lemma \ref{coro: gaussian approximation} and the second from Lemma \ref{lem: anticoncetration} of the Supplemental Material. This gives the asserted claim.
\end{proof}


\begin{lemma}\label{thm: empirical bootstrap}
Suppose that Conditions E and M are satisfied and that the random variables $X_1^*,\dots,X_n^*$ are obtained via the empirical bootstrap. Then with probability at least $1 - 2/n-3\upsilon_n$, we have
$$
\sup_{x\in\mathbb R}\left| \Pr\left(T_n \leq x\right) - \Pr\left(T_n^* \leq x \mid X_{1:n}\right) \right| \leq C\left(\frac{B_n^2\log^5(p n)}{n}\right)^{1/4},
$$
where $C$ is a constant depending only on $b_1$ and $b_2$.
\end{lemma}

\begin{proof}
As before, we may assume that \eqref{eq: bn restriction once again} holds and that $n$ is large enough so that $n\geq n_0$ for $n_0$ from Lemma \ref{lem: all conditions}, since otherwise the conclusion of the lemma is trivial by taking $C$ large enough. This will justify an application of Lemma \ref{lem: all conditions} when needed.

Let $Y_1,\dots,Y_n$ be vectors in $\mathbb R^p$ such that 
\begin{equation}\label{eq: Y infty bound}
\|Y_i\|_{\infty} \leq 10B_n\log(p n)\quad\text{for all }i=1,\dots,n,
\end{equation}
\begin{equation}\label{eq: Y2 bound lower and upper}
b_1^2/2\leq \frac{1}{n}\sum_{i=1}^n Y_{i j}^2 \quad \text{and} \quad \frac{1}{n}\sum_{i=1}^n Y_{i j}^4 \leq 2B_n^2 b_2^{2},\quad\text{for all }j=1,\dots,p,
\end{equation}
\begin{equation}\label{eq: Y2 bound plus}
\max_{1\leq j,k\leq p}\left|\frac{1}{\sqrt n}\sum_{i=1}^n(Y_{i j}Y_{i k} - \Ep[X_{ij}X_{i k}])\right| \leq C_mB_n\sqrt{\log(p n)},
\end{equation}
and
\begin{equation}\label{eq: Y3 bound}
\max_{1\leq j,k,l\leq p}\left|\frac{1}{\sqrt n}\sum_{i=1}^n(Y_{i j}Y_{i k}Y_{i l} - \Ep[X_{ij}X_{i k}X_{i l}])\right| \leq C_mB_n^2\sqrt{\log^3(p n)},
\end{equation}
where $C_m$ is the constant $C$ from Lemma \ref{lem: all conditions}. Also, let $Y_1^*,\dots,Y_n^*$ be independent random vectors with each $Y_i^*$ having uniform distribution on $\{Y_1,\dots,Y_n\}$. 

To prove the asserted claim, we will apply Theorem \ref{cor: max} with $V_i = Y_i^*$ and $Z_i = X_i$ for all $i=1,\dots,n$. Conditions V, P, and B with constants $C_v$, $C_p$, and $C_b$ depending only on $b_1$ and $b_2$ follow immediately from Conditions E and M, Lemma \ref{lem: subgaussian growth}, and the inequalities in \eqref{eq: Y infty bound} and \eqref{eq: Y2 bound lower and upper}. Also, Condition A with $\delta=\delta_n$ and $C_a$ depending only on $b_1$ and $b_2$ follows from Lemma \ref{thm: empirical anticoncentration}. Hence, an application of Theorem \ref{cor: max} is justified if we can verify \eqref{eq: bn bounds 1} and \eqref{eq: bn bounds 2} but these inequalities follow from \eqref{eq: Y2 bound plus} and \eqref{eq: Y3 bound} by noting that
$$
\frac{1}{\sqrt n}\sum_{i=1}^n(\Ep[V_{i j}V_{i k}] - Y_{i j}Y_{i k}) = 0 \quad \text{and} \quad \frac{1}{\sqrt n}\sum_{i=1}^n(\Ep[V_{i j}V_{i k}V_{i l}] - Y_{i j}Y_{i k}Y_{i l}) = 0
$$
for all $j,k,l=1,\dots,p$. 
Now, applying Theorem \ref{cor: max} shows that for all $y\in\mathbb R^p$, we have
$$
\left|\Pr\left(\frac{1}{\sqrt n}\sum_{i=1}^n V_i\leq y\right) -  \Pr\left( \frac{1}{\sqrt n}\sum_{i=1}^n X_i \leq y \right)\right|\leq K_1\left(\frac{B_n^2\log^5(p n)}{n}\right)^{1/4}
$$
for some constant $K_1$ depending only on $b_1$, $b_2$, and $C_m$. The asserted claim follows from this bound by setting $Y_i = X_i - \bar X_n$ for all $i = 1,\dots,n$, and noting that in this case \eqref{eq: Y infty bound} holds with probability at least $1 - 1/(2n^4)$ by Lemma \ref{lem: subgaussian growth} and  \eqref{eq: Y2 bound lower and upper}, \eqref{eq: Y2 bound plus}, and \eqref{eq: Y3 bound} hold jointly with probability at least $1 - 1/n-3\upsilon_n$ by Lemma \ref{lem: all conditions}.
\end{proof}

\begin{lemma}\label{eq: second order matching multiplier bootstrap ks metric}
Suppose that Conditions E and M are satisfied and that the random variables $X_1^*,\dots,X_n^*$ are obtained via the multiplier bootstrap with weights $e_1,\dots,e_n$ satisfying \eqref{eq: multiplier bootstrap simplification}. Then with probability at least $1-2/n-3v_n$, we have
\[
\sup_{x\in\mathbb{R}}|\Pr(T_n\leq x)-\Pr(T^*_n\leq x\mid X_{1:n})|\leq C\left(\frac{B_n^2\log^5(pn)}{n}\right)^{1/4},
\]
where $C$ is a constant depending only on $\Ep[e_1^3]$, $b_1$ and $b_2$.  
\end{lemma}

\begin{remark}
The constant $C$ in this result depends on $\Ep[e_1^3]$ continuously, and so we can take $C$ independent of $\Ep[e_1^3]$ under the implicitly maintained assumption that \eqref{eq: multiplier bootstrap simplification} holds. 
\end{remark}

\begin{proof}
As before, we may assume that \eqref{eq: bn restriction once again} holds and that $n$ is large enough so that $n\geq n_0$ for $n_0$ from Lemma \ref{lem: all conditions}, since otherwise the conclusion of the lemma is trivial by taking $C$ large enough. This will justify an application of Lemma \ref{lem: all conditions} when needed.

Let $\mathcal{A}_n$ be the event that \eqref{eq: subgaussian bound on x} and \eqref{eq: upper and lower bounds for second moment}--\eqref{eq: third order deviation} hold jointly. 
By Lemmas \ref{lem: subgaussian growth} and \ref{lem: all conditions}, we have $\Pr(\mathcal{A}_n)\geq1-2/n-3\upsilon_n$. Moreover, by Proposition \ref{coro:g-g-comparison},
\begin{equation}\label{eq: lem 56 beginning}
\sup_{y\in\mathbb{R}^p}|\Pr(\hat G\leq y\mid X_{1:n})-\Pr(G \leq y)|\leq C_1\delta_n
\end{equation}
on the event $\mathcal A_n$, where $C_1$ is a constant depending only on $b_1$ and $b_2$. 

Next, we claim that the case with $\sigma_e > 0$ can be reduced to the case with $\sigma_e = 0$ (and the constant 3 appearing in \eqref{eq: multiplier bootstrap simplification} replaced by some other universal constant). To prove this claim, define random variables $e'_1,\dots,e'_n$ as in Corollary \ref{coro:beta-comparison} with $\alpha=\beta=1$ such that they are independent of everything else. Then on the event $\mathcal A_n$, by Corollary \ref{coro:beta-comparison}, we have that
\[
\sup_{y\in\mathbb{R}^p}\left|\Pr\left(\frac{1}{\sqrt n}\sum_{i=1}^n e'_i\tilde X_i\leq y\mid X_{1:n}\right)-\Pr\left(\frac{1}{\sigma_e\sqrt n}\sum_{i=1}^ne_{i,1}\tilde X_i\leq y\mid X_{1:n}\right)\right|\leq C_2\delta_n,
\]
where $\tilde X_i=X_i - \bar X_n$ for all $i=1,\dots,n$ and $C_2$ is a constant depending only on $b_1$ and $b_2$. Therefore, noting that the sequences $\{e_{i,1}\}_{i=1}^n$, $\{e_{i,2}\}_{i=1}^n$, and $\{e'_{i}\}_{i=1}^n$ are independent, we have on $\mathcal A_n$ that
\begin{align*}
&\sup_{y\in\mathbb{R}^p}\left|\Pr\left(\frac{1}{\sqrt n}\sum_{i=1}^ne_{i}\tilde X_i\leq y\mid X_{1:n}\right)-\Pr\left(\frac{1}{\sqrt n}\sum_{i=1}^n(\sigma_ee'_i+e_{i,2})\tilde X_i\leq y\mid X_{1:n}\right)\right|\\
&\quad\leq\Ep\Bigg[\sup_{y\in\mathbb{R}^p}\Bigg|\Pr\left(\frac{1}{\sqrt n}\sum_{i=1}^ne_{i,1}\tilde X_i\leq y\mid X_{1:n},\{e_{i,2}\}_{i=1}^n\right) \\
&\qquad\qquad\qquad-\Pr\left(\frac{1}{\sqrt n}\sum_{i=1}^n\sigma_e e'_i\tilde X_i\leq y\mid X_{1:n},\{e_{i,2}\}_{i=1}^n\right)\Bigg|\mid X_{1:n}\Bigg] \leq C_2\delta_n.
\end{align*}
Thus, it suffices to prove the asserted claim with $e_i$'s replaced by $\sigma_ee'_i+e_{i,2}$'s, which are bounded by a universal constant (note that $\sigma_e \le 1$ since $e_i$ has unit variance).

Further, define the function $f\colon(0,1)\to\mathbb{R}$ by
\[
f(\alpha)=\frac{2\sqrt{2}(1-2\alpha)}{3\sqrt{\alpha(1-\alpha)}},\qquad\text{for all }\alpha\in(0,1).
\]
One can directly check that $f(\alpha)$ is the skewness of the Beta$(\alpha,1-\alpha)$ distribution for all $\alpha\in(0,1)$. Since $\lim_{\alpha\to0}f(\alpha)=\infty$, $\lim_{\alpha\to1}f(\alpha)=-\infty$ and $f$ is continuous, there is an $\alpha^*\in(0,1)$ satisfying $f(\alpha^*)=\Ep[e_{1}^3]$. 
We define random variables $\tilde e_1,\dots,\tilde e_n$ as in Corollary \ref{coro:beta-comparison} with $\alpha=\alpha^*$ and $\beta=1-\alpha^*$ such that they are independent of everything else. It is then easy to check that $\Ep[\tilde e_i]=0$, $\Ep[\tilde e_i^2]=1$, and $\Ep[\tilde e_i^3]=\Ep[e_i^3]$ for all $i = 1,\dots,n$. Also, applying Corollary \ref{coro:beta-comparison}, we have on $\mathcal{A}_n$ that
\begin{equation}\label{eq: third order to gaussian multipliers}
\sup_{y\in\mathbb{R}^p}\left|\Pr\left(\frac{1}{\sqrt n}\sum_{i=1}^n \tilde e_i \tilde X_i \leq y\mid X_{1:n}\right)-\Pr\left(\hat G \leq y\mid X_{1:n}\right)\right|\leq C_3\delta_n,
\end{equation}
where $C_3$ is a constant depending only on $\alpha^*$, $b_1$ and $b_2$. 

We now apply Theorem \ref{cor: max} with $V_i = e_i \tilde X_i$ and $Z_i = \tilde e_i \tilde X_i$ for all $i=1,\dots,n$ conditional on $X_{1:n}$ on the event $\mathcal A_n$. Conditions V, P, and B with $C_v$, $C_p$, and $C_b$ depending only on $\alpha^*$, $b_1$ and $b_2$ follow immediately from the inequalities \eqref{eq: subgaussian bound on x} and \eqref{eq: upper and lower bounds for second moment} and the boundedness of $e_{i}$'s and $\tilde e_i$'s. Condition A with $\delta=\delta_n$ follows from \eqref{eq: third order to gaussian multipliers} and the derivation in \eqref{beta-anti}. Moreover, \eqref{eq: bn bounds 1} and \eqref{eq: bn bounds 2} are evident by construction. Thus, by Theorem \ref{cor: max}, we have on $\mathcal{A}_n$  that
\begin{equation}\label{eq: lemma 56 end}
\sup_{y\in\mathbb{R}^p}\left|\Pr\left(\frac{1}{\sqrt n}\sum_{i=1}^n \tilde e_i \tilde X_i \leq y\mid X_{1:n}\right)-\Pr\left(\frac{1}{\sqrt n}\sum_{i=1}^n  e_i \tilde X_i  \leq y\mid X_{1:n}\right)\right|\leq C_4\delta_n,
\end{equation}
where $C_4$ is a constant depending only on $\alpha^*$, $b_1$ and $b_2$. The asserted claim now follows from combining \eqref{eq: lem 56 beginning}, \eqref{eq: third order to gaussian multipliers}, and \eqref{eq: lemma 56 end} via the triangle inequality and using Lemma \ref{coro: gaussian approximation}.
\end{proof}

We are now in the position to prove the main results from Section \ref{sec: main results}.

\begin{proof}[Proof of Theorem \ref{thm: gaussian approximation main}]
The asserted claim follows immediately from Lemma \ref{coro: gaussian approximation} by applying \eqref{eq: gaussian approximation ks metric} with $x = c_{1-\alpha}^G$.
\end{proof}

\begin{proof}[Proof of Theorem \ref{cor: rejection probabilities}]
Let $C_1$, $C_2$, and $C_3$ be the constants $C$ in Lemmas \ref{thm: empirical anticoncentration}, \ref{thm: empirical bootstrap}, and \ref{eq: second order matching multiplier bootstrap ks metric}, respectively. Set
$$
\beta_n = (1\vee C_1\vee C_2\vee C_3)\left( \frac{B_n^2\log^5(p n)}{n} \right)^{1/4}.
$$
By Lemmas \ref{thm: empirical bootstrap} and \ref{eq: second order matching multiplier bootstrap ks metric}, we have 
$
\sup_{x\in\mathbb R}\left| \Pr(T_n \leq x) - \Pr(T_n^* \leq x\mid X_{1:n}) \right| \leq \beta_n
$
with probability at least $1 - 2/n-3\upsilon_n$. Hence, letting $c_{1-\gamma}$ be the $(1-\gamma)$th quantile of $T_n$ for all $\gamma\in(0,1)$, we have with the same probability that
$$
\Pr(T_n^* \leq c_{1-\alpha + \beta_n}\mid X_{1:n}) \geq \Pr(T_n \leq c_{1-\alpha + \beta_n}) - \beta_n \geq 1 - \alpha, \quad \text{and}
$$
\begin{align*}
\Pr(T_n^* \leq c_{1-\alpha - 3\beta_n} \mid X_{1:n}) & \leq \Pr(T_n \leq c_{1-\alpha - 3\beta_n}) + \beta_n\\
& \leq 1-\alpha - 2\beta_n + C_{1}\left(\frac{B_n^2\log^5(p n)}{n}\right)^{1/4} < 1 - \alpha,
\end{align*}
where the second inequality follows from Lemma \ref{thm: empirical anticoncentration}. Therefore,
$$
\Pr(c_{1-\alpha - 3\beta_n} < c_{1-\alpha}^B \leq c_{1-\alpha + \beta_n}) \geq 1 - 2/n-3\upsilon_n\geq1-5\upsilon_n,
$$
so that 
$$
\Pr(T_n > c_{1-\alpha}^B) \leq \Pr(T_n > c_{1-\alpha - 3\beta_n}) + 5\upsilon_n \leq \alpha + 3\beta_n + 5\upsilon_n \leq \alpha + 8\beta_n \quad \text{and} 
$$
\begin{align*}
\Pr(T_n > c_{1-\alpha}^B) &\geq \Pr(T_n > c_{1-\alpha + \beta_n}) - 5\upsilon_n \\
& \geq \alpha - \beta_n - C_{1}\left( \frac{B_n^2\log^5(p n)}{n} \right)^{1/4} - 5\upsilon_n \geq \alpha - 7\beta_n,
\end{align*}
where the second inequality follows from Lemma \ref{thm: empirical anticoncentration}. Combining these inequalities gives the asserted claim.
\end{proof}

\begin{acks}[Acknowledgments]
We are grateful to Tim Armstrong, Matias Cattaneo, Xiaohong Chen, and Tengyuan Liang for helpful discussions. We also thank seminar participants at  the University of Pennsylvania and Yale University.
\end{acks}

\begin{funding}
 K. Kato is suport by the NSF DMS-1952306 and DMS-2014636.
\end{funding}

\begin{supplement}
The Supplementary Material contains proofs omitted in the main text as well as several technical tools, and the simulation results. 
\end{supplement}

\newpage
\centerline{\textbf{\Large{Supplementary Material}}}

\appendix
\section{Proofs for Section \ref{sec: polynomial moment conditions}}
\label{sec: proof of poly moment}

\subsection{Technical lemma}
The proofs of Theorems \ref{cor: ga under polynomial conditions} and \ref{cor: boot app pol mom cond} rely on the following lemma combined with a truncation argument.  
Lemma \ref{lem:dirty} below is a version of the high-dimensional CLT and would be of independent interest. The proof of Lemma \ref{lem:dirty} relies on some elements of the proof of Lemma \ref{lem: main} and Stein's exchangeable pair approach. 
In what follows,  $\| \cdot \|_{\infty}$ denotes the $\ell^\infty$-norm for vectors, i.e., $\| y \|_{\infty} = \max_{1 \le j \le p} |y_j|$ for $y=(y_1,\dots,y_p)' \in \R^p$. 
\begin{lemma}\label{lem:dirty}
Let $X_{1:n}=(X_1,\dots, X_n)$ be a sequence of centered independent random vectors in $\mathbb{R}^p$. 
Set $S_n=n^{-1/2}\sum_{i=1}^n X_i$ and $\Sigma=n^{-1}\sum_{i=1}^n\Ep[X_iX_i']$. 
Let $\tilde X_{1:n}=(\tilde X_1,\dots,\tilde X_n)$ be an independent copy of $X_{1:n}$ and set $Y_i=(\tilde X_i-X_i)/\sqrt n$ for $i=1,\dots,n$. 
Let $\tilde\Sigma$ be a $d\times d$ positive semidefinite symmetric matrix such that $\min_{1\leq j\leq p}\tilde\Sigma_{jj}\geq c$ for some constant $c>0$. 
Then for any positive constant $\phi$,
\begin{equation}\label{eq:dirty}
\begin{split}
&\sup_{y\in\mathbb R^p}\left|\Pr(S_n\leq y) - \Pr(Z\leq y)\right|\\
&\leq C\left(
\phi(\log p)^{3/2}\Delta_0
+\phi(\log p)^2\sqrt{\Delta_1}+\phi^3(\log p)^{7/2}\Delta_1\right.\\
&\left.+(\log p)\sqrt{\Delta_2(\phi)}+\phi(\log p)^{3/2}\Delta_2(\phi)
+\sqrt{(\log p)^{3}\Delta_3(\phi)}
+\frac{\sqrt{\log p}}{\phi}
\right),
\end{split}
\end{equation}
where $C>0$ is a constant depending only on $c$, $Z$ is a centered Gaussian vector in $\mathbb{R}^p$ with covariance matrix $\tilde\Sigma$, and
\begin{align*}
\Delta_0&=\max_{1\leq j,k\leq p}|\Sigma_{jk}-\tilde\Sigma_{jk}|,\qquad
\Delta_1=\frac{1}{n^2}\max_{1\leq j\leq p}\sum_{i=1}^n\Ep[X_{ij}^4],\\
\Delta_2(\phi)&=\max_{1\leq j\leq p}\sum_{i=1}^n\Ep\left[Y_{ij}^21\{\|Y_i\|_\infty>(\phi\log p)^{-1}\}\right],\\
\Delta_3(\phi)&=\Ep\left[\max_{1\leq i\leq n}\|Y_{i}\|_\infty^21\{\|Y_i\|_\infty>(\phi\log p)^{-1}\}\right].
\end{align*}
In particular, if there exists a constant $\kappa_n>0$ such that $\|X_i\|_\infty\leq \kappa_n$ for every $i=1,\dots,n$,
then 
\begin{equation}\label{bdd-version}
\begin{split}
&\sup_{y\in\mathbb R^p}\left|\Pr(S_n\leq y) - \Pr(Z\leq y)\right|\\
&\leq C'\left(
\sqrt{\Delta_0}\log p
+\left(\Delta_1\log^5p\right)^{1/4}
+\frac{\kappa_n(\log p)^{3/2}}{\sqrt n}
\right),
\end{split}
\end{equation}
where $C'>0$ is a constant depending only on $c$.
\end{lemma}

The proof of this lemma is differed to Appendix \ref{sec: proof of tech lemma} below. 

\subsection{Proof of Theorem \ref{cor: ga under polynomial conditions}}
Without loss of generality, we may assume
\ben{\label{wlog0}
\left(\frac{B_n^2\log^5p}{n}\right)^{1/4}+\sqrt{\frac{B_n^2(\log p)^{3-2/q}}{n^{1-2/q}}}\leq1.
} 
Also, we will use the symbol $\lesssim$ to denote inequalities that hold up to constants depending only on $q$, $b_1$, and $b_2$.

Set $S_n = n^{-1/2}\sum_{i=1}^n X_i$. Recall that $G \sim N(0,\Sigma_n)$ with $\Sigma_n = \Ep[S_nS_n']$. We will show that 
\begin{equation}
\sup_{y\in\mathbb R^p}\left|\Pr(S_n\leq y) - \Pr(G\leq y)\right|\lesssim\left(\frac{B_n^2\log^5p}{n}\right)^{1/4}+\sqrt{\frac{B_n^2(\log p)^{3-2/q}}{n^{1-2/q}}},
\label{eq: poly rectangle}
\end{equation}
which implies the desired result. 
We apply \eqref{eq:dirty} with $\tilde\Sigma=\Sigma_n$. 
Under the assumptions of Theorem \ref{cor: ga under polynomial conditions}, 
\ba{
\Delta_1&\leq\frac{B_n^2}{n},&
\Delta_2(\phi)&\lesssim(\phi\log p)^{q-2}\frac{B_n^q}{n^{q/2-1}},&
\Delta_3(\phi)&\lesssim(\phi\log p)^{q-2}\frac{B_n^q}{n^{q/2}}.
}
Hence,
\ba{
&\sup_{y\in\mathbb R^p}\left|\Pr(S_n\leq y) - \Pr(G\leq y)\right|\\
&\lesssim \phi(\log p)^2\frac{B_n}{\sqrt n}
+\phi^3(\log p)^{7/2}\frac{B_n^2}{n}
+\phi^{q/2-1}(\log p)^{q/2}\sqrt{\frac{B_n^q}{n^{q/2-1}}}\\
&\qquad+\phi^{q-1}(\log p)^{q-1/2}\frac{B_n^q}{n^{q/2-1}}
+\sqrt{\phi^{q-2}(\log p)^{q+1}\frac{B_n^q}{n^{q/2}}}
+\frac{\sqrt{\log p}}{\phi}.
}
Choose the constant $\phi$ such that
\[
\phi^{-1}=\left(\frac{B_n^2\log^3p}{n}\right)^{1/4}+\frac{B_n(\log p)^{1-1/q}}{n^{1/2-1/q}}.
\]
Then we have
\ba{
 &\phi(\log p)^2\frac{B_n}{\sqrt n}
+\phi^3(\log p)^{7/2}\frac{B_n^2}{n}\lesssim\left(\frac{B_n^2\log^5p}{n}\right)^{1/4},\\
&\phi^{q/2-1}(\log p)^{q/2}\sqrt{\frac{B_n^q}{n^{q/2-1}}}+\phi^{q-1}(\log p)^{q-1/2}\frac{B_n^q}{n^{q/2-1}}
\lesssim\sqrt{\frac{B_n^2(\log p)^{3-2/q}}{n^{1-2/q}}},\\
&\sqrt{\phi^{q-2}(\log p)^{q+1}\frac{B_n^q}{n^{q/2}}}\lesssim\frac{B_n(\log p)^{2-1/q}}{n^{1-1/q}}
\leq\sqrt{\frac{B_n^2(\log p)^{3-2/q}}{n^{1-2/q}}},
}
and
\ba{
\frac{\sqrt{\log p}}{\phi}\lesssim\left(\frac{B_n^2\log^5p}{n}\right)^{1/4}+\sqrt{\frac{B_n^2\log^{3-2/q}p}{n^{1-2/q}}}.
}
Combining these bounds leads to \eqref{eq: poly rectangle}. 
\qed

\subsection{Proof of Theorem \ref{cor: boot app pol mom cond}}
Without loss of generality, we may assume
\ben{\label{wlog}
\delta_{n,q} := \left(\frac{B_n^2\log^5(pn)}{n}\right)^{1/4}+\sqrt{\frac{B_n^2\log^{3-2/q}(pn)}{n^{1-2/q}}}\leq1.
} 
Also, we will use the symbol $\lesssim$ to denote inequalities that hold up to constants depending only on $q$, $b_1$, and $b_2$. 

Set $S_n^* = n^{-1/2}\sum_{i=1}^n X_i^*$ . In view of the proof of Theorem \ref{cor: rejection probabilities}, it suffices to show that
 \begin{equation}
\sup_{y \in \R^p}\left|\Pr(S_n^*\le y \mid X_{1:n}) - \Pr(G \le y)\right| \le  C \delta_{n,q} 
\label{eq: bootstrap rectangle}
\end{equation}
holds with probability at least $1-C\delta_{n,q}$ for some constant $C$ that depends only on $q,b_1$, and $b_2$. 
Set $\kappa_n=B_n(n/\log(pn))^{1/q}$ and for $i=1,\dots,n$ and $j=1,\dots,p$, define 
\[
\wh{X}_{ij}=X_{ij}1\{\|X_{i}\|_{\infty}\leq\kappa_n\}.
\] 
Also,  define $\wh S_n^*$ in the same way as $S_n^*$ with $X_{ij}$ replaced by $\wh X_{ij}$. Since
\begin{align*}
\Pr(X_{ij}\neq\wh X_{ij}\text{ for some }i,j)
&\leq\Pr(\max_{1\leq i\leq n}\|X_i\|_\infty>\kappa_n)\\
&\leq\kappa_n^{-q}B_n^q=\frac{\log (pn)}{n},
\end{align*}
we have that
\[
\sup_{y \in \R^p}\left|\Pr(S_n^*\le y \mid X_{1:n}) - \Pr(G \le y)\right|
=\sup_{y \in \R^p}\left|\Pr(\wh S_n^* \le y \mid X_{1:n}) - \Pr(G \le y)\right|
\]
with probability at least $1-\log (pn)/n$. We will show below that
\begin{equation}\label{boot-aim}
\sup_{y \in \R^p}\left|\Pr(\wh S_n^* \le y \mid X_{1:n}) - \Pr(G \le y)\right|
\lesssim \delta_{n,q}
\end{equation}
with probability at least $1-4/n-B_n^q/n^{q/2-1}$.  Since 
\[
B_n^q/n^{q/2-1}=(B_n^2/n^{1-2/q})^{q/2} \le \sqrt{\frac{B_n^2\log^{3-2/q}(pn)}{n^{1-2/q}}},
\]
these results imply \eqref{eq: bootstrap rectangle}.

\medskip

\noindent\textbf{Gaussian multiplier bootstrap}. 
Applying Proposition \ref{coro:g-g-comparison} conditional on the data, we have 
\ben{\label{gmb-1}
\sup_{y \in \R^p}\left|\Pr(\wh S_n^* \le y \mid X_{1:n}) - \Pr(G \le y)\right|
\lesssim \sqrt{\Delta_{n,r}}\log p,
}
where
\[
\Delta_{n,r}:=\max_{1\leq j,k\leq p}\left|\frac{1}{n}\sum_{i=1}^n(\wh X_{ij}-\wh X^{ave}_{n,j})(\wh X_{ik}-\wh X^{ave}_{n,k})-\Sigma_{n,jk}\right|
\]
with $\wh X^{ave}_{n,j}=n^{-1}\sum_{i=1}^n\wh X_{ij}$. 
Since
\begin{equation}\label{eq:cov-diff}
\begin{split}
&\max_{1\leq j,k\leq p}\left|\frac{1}{n}\sum_{i=1}^n\left(\Ep\left[\wh{X}_{ij}\wh{X}_{ik}\right]-\Ep\left[X_{ij}X_{ik}\right]\right)\right| \\
&\leq\max_{1\leq j\leq p}\max_{1\leq i\leq n}2\Ep\left[X_{ij}^21\{\|X_i\|_{\infty}>\kappa_n\}\right] \leq \frac{2B_n^q}{\kappa_n^{q-2}}
=2\frac{B_n^2(\log p)^{1-2/q}}{n^{1-2/q}},
\end{split}
\end{equation}
we can bound $\Delta_{n,r}$ as
\ben{\label{gmb-2}
\Delta_{n,r}\leq\Delta^{(1)}_{n,r}+\{\Delta^{(2)}_{n,r}\}^2+2\frac{B_n^2(\log p)^{1-2/q}}{n^{1-2/q}},
}
where
\[
\Delta^{(1)}_{n,r}:=\max_{1\leq j,k\leq p}\left|\frac{1}{n}\sum_{i=1}^n\left(\wh X_{ij}\wh X_{ik}-\Ep[\wh X_{ij}\wh X_{ik}]\right)\right|,\quad
\Delta^{(2)}_{n,r}:=\max_{1\leq j\leq p}|\wh X^{ave}_{n,j}|.
\]

First we bound $\Delta^{(1)}_{n,r}$. Observe that
\ba{
\max_{1\leq j,k\leq p}\frac{1}{n^2}\sum_{i=1}^n\Ep[\wh X_{ij}^2\wh X_{ik}^2]
\leq\max_{1\leq j\leq p}\frac{1}{n^2}\sum_{i=1}^n\Ep[X_{ij}^4]\leq\frac{B_n^2}{n}
}
and
\ba{
\frac{1}{n}\max_i\max_{1\leq j,k\leq p}|\wh X_{ij}\wh X_{ik}|
\leq\frac{\kappa_n^2}{n}.
}
Thus, by Lemma \ref{lem: maximal ineq},
\ba{
\Ep[\Delta^{(1)}_{n,r}]
&\lesssim\sqrt{\frac{B_n^2\log p}{n}}+\frac{\kappa_n^2\log p}{n}.
}
Also, by Lemma E.2 in \cite{CCK17}, there is a universal constant $K>0$ such that, for any $t>0$, we have 
$\Delta^{(1)}_{n,r}\lesssim \Ep[\Delta^{(1)}_{n,r}]+t$ with probability at least $1-\exp(-nt^2/(3B_n^2))-3\exp(-tn/(K\kappa_n^2))$. 
Choosing
\[
t=\sqrt{\frac{3B_n^2\log(pn)}{n}}+\frac{2K\kappa_n^2\log(pn)}{n},
\]
we have
\ben{\label{gmb-3}
\Delta^{(1)}_{n,r}
\lesssim\sqrt{\frac{B_n^2\log(pn)}{n}}+\frac{\kappa_n^2\log(pn)}{n}
=\sqrt{\frac{B_n^2\log(pn)}{n}}+\frac{B_n^2\log^{1-2/q}(pn)}{n^{1-2/q}}
}
with probability at least $1-2/n$. 

Next we bound $\Delta^{(2)}_{n,r}$. 
Similarly to the above argument, we have 
\[
\max_{1\leq j\leq p}|\wh X^{ave}_{n,j}-\Ep[\wh X^{ave}_{n,j}]|\lesssim\sqrt{\frac{B_n\log(pn)}{n}}+\frac{B_n\log^{1-1/q}(pn)}{n^{1-1/q}}
\]
with probability at least $1-2/n$. 
Also, since $\Ep[X_{ij}]=0$, 
\ba{
|\Ep[\wh X_{ij}]|=|-\Ep[X_{ij}1\{\|X_i\|_\infty>\kappa_n\}]|\leq\kappa_n^{-q+1}B_n^q=\frac{B_n(\log p)^{1-1/q}}{n^{1-1/q}}.
}
Hence
\ben{\label{gmb-4}
\Delta^{(2)}_{n,r}\lesssim\sqrt{\frac{B_n\log(pn)}{n}}+\frac{B_n(\log p)^{1-1/q}}{n^{1-1/q}}
}
with probability at least $1-2/n$. By \eqref{wlog} and \eqref{gmb-1}--\eqref{gmb-4}, \eqref{boot-aim} holds with probability at least $1-4/n$.
\medskip


\noindent\textbf{Empirical bootstrap}. 
Applying \eqref{bdd-version} conditional on the data, we obtain
\begin{align}
&\sup_{y \in \R^p}\left|\Pr(\wh S_n^* \le y \mid X_{1:n}) - \Pr(G \le y)\right|
\notag\\
&\lesssim
\sqrt{\Delta_{n,r}}\log p
+\left(\frac{\Delta'_{n,r}\log^5p}{n}\right)^{1/4}
+\frac{\kappa_n(\log p)^{3/2}}{\sqrt n},\label{eb-1}
\end{align}
where, with $\wh X_{ij}^*:=X_{ij}^*1\{\|X_{i}^*\|_\infty\leq\kappa_n\}$,
\ba{
\Delta_{n,r}'=\frac{1}{n}\max_{1\leq j\leq p}\sum_{i=1}^n\Ep[(\wh X_{ij}^*-\wh X^{avg}_{n,j})^4\mid X_{1:n}].
}
We have
\begin{align*}
\Delta_{n,r}'
\lesssim\frac{1}{n}\max_{1\leq j\leq p}\sum_{i=1}^n\wh X_{ij}^4.
\end{align*}
Observe that
\[
\max_{1\leq j\leq p}\sum_{i=1}^n\Ep[\wh X_{ij}^4]
\leq\max_{1\leq j\leq p}\sum_{i=1}^n\Ep[X_{ij}^4]\leq nB_n^2
\]
and
\[
\Ep\left[\max_{1\leq i\leq n}\|\wh X_{i}\|_\infty^{2q}\right]
\leq\kappa_n^q\Ep\left[\max_{1\leq i\leq n}\|\wh X_{i}\|_\infty^{q}\right]
\leq\kappa_n^qB_n^q.
\]
Since $2q>4$, the second bound particularly yields
\[
\Ep\left[\max_{1\leq i\leq n}\|\wh X_{i}\|_\infty^{4}\right]
\leq\kappa_n^2B_n^2.
\]
Thus, by Lemma 9 in \cite{CCK15},
\begin{align*}
\Ep\left[\max_{1\leq j\leq p}\sum_{i=1}^n\wh X_{ij}^4\right]
\lesssim nB_n^2+B_n^2\kappa_n^2(\log p).
\end{align*}
By \eqref{wlog},
\[
\kappa_n^2(\log p)\leq\frac{B_n^2\log^{1-2/q}(pn)}{n^{-2/q}}\leq n.
\]
Hence
\begin{align*}
\Ep\left[\max_{1\leq j\leq p}\sum_{i=1}^n\wh X_{ij}^4\right]
\lesssim nB_n^2.
\end{align*}
Then, by Lemma E.4(ii) in \cite{CCK15} with $s=q/2$, there is a constant $K'>0$ depending only on $q$ such that, for any $t>0$,
\ba{
\max_{1\leq j\leq p}\sum_{i=1}^n\wh X_{ij}^4
\lesssim nB_n^2+t
}
holds with probability at least $1-K'(B_n^2\kappa_n^2/t)^{q/2}$. Choosing $t=nB_n^2(K')^{2/q}$, we have
\ba{
\max_{1\leq j\leq p}\sum_{i=1}^n\wh X_{ij}^4
\lesssim nB_n^2
}
with probability at least $1-(\kappa_n^2/n)^{q/2}\geq1-B_n^q/n^{q/2-1}$. All together, 
\ben{\label{eb-2}
\Delta_{n,r}'
\lesssim B_n^2
}
holds with probability at least $1-B_n^q/n^{q/2-1}$. In addition, 
\[
\frac{\kappa_n(\log p)^{3/2}}{\sqrt n}
=\frac{B_n\log^{3/2-1/q}(np)}{n^{1/2-1/q}}
=\sqrt{\frac{B_n^2\log^{3-2/q}(p n)}{n^{1-2/q}}}.
\]
Consequently, \eqref{wlog}, \eqref{gmb-2}--\eqref{gmb-4} and \eqref{eb-1} imply that \eqref{boot-aim} holds with probability at least $1-4/n-B_n^q/n^{q/2-1}$. 
\qed

\section{Sharpness of Gaussian-to-Gaussian Comparison in Proposition \ref{coro:g-g-comparison}}\label{sec: sharpness}

\begin{proposition}[Sharpness of Proposition \ref{coro:g-g-comparison}]\label{ggcomp:lower-bound}
Let $(\zeta_i)_{i=1}^\infty$ and $(\eta_j)_{j=1}^\infty$ be two independent sequences of independent $N(0,1)$ random variables.  Also, let $\sigma=\sigma_p$ be a sequence of positive constants such that $\sigma\log p\to0$ and $\sigma p^c\to\infty$ as $p\to\infty$ for any $c>0$. 
Define $Z_{1,ij}:=\zeta_i+\sigma\eta_j$ and $Z_{2,ij}:=\zeta_i$ for all $i,j=1,\dots,p$. Then we have
$$
\liminf_{p\to\infty}\frac{1}{\sqrt{\Delta}\log p}\sup_{y\in\mathbb{R}}\left|\Pr\left(\max_{1\leq i,j\leq p}Z_{1,ij}\leq y\right)-\Pr\left(\max_{1\leq i,j\leq p}Z_{2,ij}\leq y\right)\right|>0,
$$
where $\Delta:=\max_{1\leq i,j,k,l\leq p}\left|\Ep[Z_{1,ij}Z_{1,kl}]]-\Ep[Z_{2,ij}Z_{2,kl}]]\right|=\sigma^2$. 
\end{proposition}
\begin{proof}[Proof of Proposition \ref{ggcomp:lower-bound}]
Without loss of generality, we may assume $\sigma\log p\leq(\sqrt{2}+1/24)^{-1}$. 
Set $M_p:=\max_{1\leq j\leq p}\eta_j$ and $A_p:=\{|M_p-\Ep[M_p]|\leq\sqrt{\log p}/24\}$. 
Then, by equation (1.5) in \cite{Tanguy15}, there is a universal constant $c_1>0$ such that
\[
\Pr(A_p^c)=\Pr\left(|M_p-\Ep[M_p]|>\sqrt{\log p}/24\right)\leq 6e^{-c_1\log p}=6/p^{c_1}.
\]
Thus, by assumption we obtain
\begin{equation}\label{eq:max-dev}
\Pr(A_p^c)=o(\sigma)\qquad\text{as }p\to\infty.
\end{equation}
Now, note that $\max_{1\leq i,j\leq p}Z_{1,ij}=\max_{1\leq i\leq p}\zeta_{i}+\sigma M_p$. 
Then, for every $y\in\mathbb{R}$, the triangle inequality yields
\begin{align*}
&\left|\Pr\left(\max_{1\leq i,j\leq p}Z_{1,ij}\leq y\right)-\Pr\left(\max_{1\leq i,j\leq p}Z_{2,ij}\leq y\right)\right|\\
&\geq\left|\Pr\left(\left\{\max_{1\leq i\leq p}\zeta_i+\sigma M_p\leq y\right\}\cap A_p\right)-\Pr\left(\left\{\max_{1\leq i\leq p}\zeta_i\leq y\right\}\cap A_p\right)\right|-2\Pr(A_p^c).
\end{align*}
Since $\sqrt{\log p}/12\leq\Ep[M_p]\leq\sqrt{2\log p}$ by the fourth inequality on page 58 of \cite{CCK15}, we have $\sqrt{\log p}/24\leq M_p\leq(\sqrt{2}+1/24)\sqrt{\log p}$ on $A_p$. 
In particular, $M_p>0$ on $A_p$, so we have
\[
\begin{split}
&\left|\Pr\left(\left\{\max_{1\leq i\leq p}\zeta_i+\sigma M_p\leq y\right\}\cap A_p\right)-\Pr\left(\left\{\max_{1\leq i\leq p}\zeta_i\leq y\right\}\cap A_p\right)\right|\\
&\quad =\Ep\left[\int_{y-\sigma M_p}^yf_p(x)dx;A_p\right],
\end{split}
\]
where $f_p$ denotes the density of $\max_{1\leq i\leq p}\zeta_i$. 
Then, by the second inequality on page 58 of \cite{CCK15}, there is a universal constant $c_2>0$ such that $f_p(x)\geq c_2\sqrt{2\log p}$ for all $x\in[d_p-1/\sqrt{\log p},d_p+1/\sqrt{\log p}]$, where
\[
d_p:=\sqrt{2\log p}-\frac{\log(4\pi)+\log\log p}{2\sqrt{2\log p}}.
\]
Since $\sigma M_p\leq1/\sqrt{\log p}$ on $A_p$ by assumption, we deduce that
\begin{align*}
&\sup_{t\in\mathbb{R}}\left|\Pr\left(\max_{1\leq i,j\leq p}Z_{1,ij}\leq y\right)-\Pr\left(\max_{1\leq i,j\leq p}Z_{2,ij}\leq y\right)\right|\\
&\geq\Ep\left[\int_{d_p-\sigma M_p}^{d_p}f_p(x)dx;A_p\right]-2\Pr(A_p^c)
\geq\Ep\left[c_2\sigma M_p\sqrt{2\log p};A_p\right]-2\Pr(A_p^c).
\end{align*}
Since $M_p/\sqrt{2\log p}\to1$ almost surely as $p\to\infty$, we obtain
\[
\liminf_{p\to\infty}\frac{1}{\sigma\log p}\Ep\left[c_2M_p\sqrt{2\log p};A_p\right]
\geq 2c_2
\]
by \eqref{eq:max-dev} and Fatou's lemma. Combining this with \eqref{eq:max-dev} completes the proof of the proposition. 
\end{proof}

\section{Proof of Lemma \ref{lem: main}}\label{sec: proof of main lemma}
Since we assume throughout the paper that $p\geq 2$, the asserted claim is trivial if $\phi<1$. We will therefore assume in the proof that $\phi\geq 1$. In turn, $\phi\geq 1$ together with \eqref{eq: phi restriction} imply that
\begin{equation}\label{eq: lazy condition}
C_p B_n \log^2(p n)\leq \sqrt n.
\end{equation}
This condition will be useful in the proof.

Fix $d = 0,\dots,D-1$ and $e^d\in\{0,1\}^n$ such that if $\epsilon^d = e^d$, then $\mathcal A_d$ holds. All arguments in this proof will be conditional on $\epsilon^d = e^d$. For brevity of notation, however, we make this conditioning implicit and write $\Pr(\cdot)$ and $\Ep[\cdot]$ instead of $\Pr(\cdot\mid \epsilon^d = e^d)$ and $\Ep[\cdot\mid \epsilon^d = e^d]$, respectively. 

Fix any five-times continuously differentiable and decreasing function $g_0\colon\mathbb R\to\mathbb R$ such that (i) $g_0(t)\geq 0$ for all $t\in\mathbb R$, (ii) $g_0(t) = 0$ for all $t\geq 1$, and (iii) $g_0(t) = 1$ for all $t\leq 0$. For this function, there exists a constant $C_g>0$ such that
$$
\sup_{t\in\mathbb R}\Big(|g_0^{(1)}(t)|\vee|g_0^{(2)}(t)|\vee|g_0^{(3)}(t)|\vee|g_0^{(4)}(t)|\vee|g_0^{(5)}(t)|\Big) \leq C_g.
$$
In this proof, we will use the symbol $\lesssim$ to denote inequalities that hold up to a constant depending only on $C_v$, $C_p$, $C_b$, $C_a$, and $C_g$. Since $g_0$ can be chosen to be universal, we say that the inequality for $\varrho_{\epsilon^d}$ in the statement of the lemma holds up to a constant depending only on $C_v$, $C_p$, $C_b$, and $C_a$.

Fix $\phi \geq 1$ and set
$
\beta = \phi\log p.
$
Define functions $g\colon \mathbb R\to\mathbb R$ and $F\colon \mathbb R^p\to\mathbb R$ by $g(t) = g_0(\phi t)$ for all $t\in\mathbb R$ and
$$
F(w)=\beta^{-1}\log\left( \sum_{j=1}^p\exp(\beta w_j) \right),\quad\text{for all }w\in\mathbb R^p.
$$
It is immediate that the function $g$ satisfies
\begin{equation}
g(t)=\begin{cases}
1 & \text{if }t\leq0,\\
0 & \text{if }t\geq\phi^{-1}.
\end{cases}\label{eq: g property}
\end{equation}
It is also straightforward to check that the function $F$ has the following property:
\begin{equation}\label{eq: f properties}
\max_{1\leq j\leq p}w_j  \leq F(w)  \leq  \max_{1\leq j\leq p}w_j + \phi^{-1},\ \text{for all }w\in\mathbb R^p;
\end{equation}
see \cite{CCK13} for details.
Also, for all $y\in\mathbb R^p$, define the function $m^y\colon\mathbb R^p\to\mathbb R$ by
$$
m^y(w) = g(F(w - y)),\quad\text{for all }w\in\mathbb R^p.
$$
Below, we will need partial derivatives of $m^y$ up to the fifth order. For brevity of notation, we will use indices to denote these derivatives. For example, for any $j,k,l,r,h=1,\dots,p$, we will write
$$
m_{jklrh}^y(w)=\frac{\partial^5m^y(w)}{\partial w_j \partial w_k \partial w_l \partial w_r \partial w_h},\quad\text{for all }w\in\mathbb R^p.
$$
Using straightforward but lengthy algebra, we can show that the function $m^y$ has the following property: for all $j,k,l,r,h=1,\dots,p$, there exist functions $U_{jk}^y\colon\mathbb R^p\to\mathbb R$, $U_{jkl}^y\colon\mathbb R^p\to\mathbb R$, $U_{jklr}^y\colon\mathbb R^p\to\mathbb R$, and $U_{jklrh}^y\colon\mathbb R^p\to\mathbb R$ such that (i) for all $w\in\mathbb R^p$, we have
\begin{equation}\label{eq: property 1 - 2 and 3}
|m^y_{jk}(w)| \leq U^y_{jk}(w), \quad |m^y_{jkl}(w)| \leq U^y_{jkl}(w),
\end{equation}
\begin{equation}\label{eq: property 1 - 3 and 4}
|m^y_{jklr}(w)| \leq U^y_{jklr}(w), \quad |m^y_{jklrh}(w)| \leq U^y_{jklrh}(w),
\end{equation}
(ii) for all $w_1\in\mathbb R^p$ and $w_2\in\mathbb R^p$ such that $\beta\|w_2\|_{\infty}\leq 1$, we have
\begin{equation}\label{eq: property 2 - 4 and 5}
U^y_{jklr}(w_1+w_2)\lesssim U^y_{jklr}(w_1), \quad U^y_{jklrh}(w_1+w_2)\lesssim U^y_{jklrh}(w_1),
\end{equation}
and (iii) for all $w\in\mathbb R^p$,
\begin{equation}\label{eq: property 3 - 2 and 3}
\sum_{j,k=1}^p U_{jk}^y(w) \lesssim \phi^2\log p, \quad \sum_{j,k,l=1}^p U_{jkl}^y(w) \lesssim \phi^3\log^2p,
\end{equation}
\begin{equation}\label{eq: property 3 - 3 and 4}
\sum_{j,k,l,r=1}^p U_{jklr}^y(w) \lesssim \phi^4\log^3p, \quad \sum_{j,k,l,r,h=1}^p U_{jklrh}^y(w) \lesssim \phi^5\log^4p.
\end{equation}
For example, we can set
\begin{align*}
U^y_{jk}(w)& = C_g(\phi^2+\phi\beta)\frac{\exp(\beta(w_j - y_j))\exp(\beta(w_k - y_k))}{\left( \sum_{i=1}^p \exp(\beta(w_i - y_i)) \right)^2}\\
& \quad + C_g\phi\beta 1\{j = k\} \frac{\exp(\beta(w_j - y_j))}{\sum_{i=1}^p \exp(\beta(w_i - y_i))},\quad\text{for all }w\in\mathbb R^p;
\end{align*}
see \cite{CCK13} and \cite{CCK17} for more details.

Further, for all $y\in\mathbb R^p$, define
\begin{equation}\label{eq: iy}
\mathcal I^{y} =m^y( S_{n,\epsilon^d}^V ) - m^y( S_n^Z )
\end{equation}
and
\begin{equation}\label{eq: h definition}
h^y(Y; x) = 1\left\{ -x < \max_{1\leq j\leq p}(Y_j - y_j) \leq x \right\},\ \text{for all }x\geq 0\text{ and }Y\in\mathbb R^p.
\end{equation}
Also, denote
\begin{equation}\label{eq: W definition}
W = \frac{1}{\sqrt n}\sum_{i=1}^n\Big(\epsilon_i^{d + 1}V_i + (1-\epsilon_i^{d+1})Z_i\Big).
\end{equation}

For the rest of the proof, we proceed in five steps. In the first step, we show that
\begin{align}
\sup_{y\in\mathbb R^p} |\Ep[\mathcal I^y]| &\lesssim \frac{B_n^2\phi^4\log^{5}(p n)}{n^2} +  \left( \Ep[\varrho_{\epsilon^{d+1}}] + \frac{\sqrt{\log p}}{\phi}+\delta \right)\nonumber\\
&\quad \times \left( \frac{\mathcal B_{n,1,d}\phi^2\log p}{\sqrt n} + \frac{\mathcal B_{n,2,d}\phi^3\log^2p}{n} + \frac{B_n^2\phi^4\log^3(p n)}{n} \right).\label{eq: EIy bound}
\end{align}
In the second step, we show that
\begin{equation}\label{eq: anticoncentration application}
\varrho_{\epsilon^d} \lesssim \frac{\sqrt{\log p}}{\phi} +\delta+ \sup_{y\in\mathbb R^p}|\Ep[\mathcal I^{y}]|.
\end{equation}
Combining two steps, we obtain the asserted claim. In Steps 3, 4, and 5, we provide some auxiliary calculations.

\medskip
\noindent
{\bf Step 1.} Here, we prove \eqref{eq: EIy bound}. Recalling that $I_d = \{i=1,\dots,n\colon \epsilon_i^d = 1\}$, let $\mathcal S_n$ be the set of all one-to-one functions mapping $\{1,\dots,|I_d|\}$ to $I_d$, and let $\sigma$ be a random function with uniform distribution on $\mathcal S_n$ such that $\sigma$ is independent of $V_1,\dots,V_n$, $Z_1,\dots,Z_n$, and $\epsilon^{d+1}$.

Denote
$$
W_i^{\sigma} = \frac{1}{\sqrt n}\sum_{j=1}^{i-1}V_{\sigma(j)} + \frac{1}{\sqrt n}\sum_{j = i + 1}^{|I_d|} Z_{\sigma(j)} + \frac{1}{\sqrt n}\sum_{j\notin I_d}Z_j,\ \text{for all }i=1,\dots,|I_d|.
$$
Note that for any function $m\colon \mathbb R^p\to\mathbb R$ and any $i\in I_d$, it follows from Lemma \ref{lem: rand lind} that
$$
\Ep\left[ \frac{\sigma^{-1}(i)}{|I_d|+1}m\Big(W^{\sigma}_{\sigma^{-1}(i)} + \frac{V_i}{\sqrt n}\Big) + \Big(1 - \frac{\sigma^{-1}(i)}{|I_d|+1}\Big)m\Big(W^{\sigma}_{\sigma^{-1}(i)} +\frac{ Z_i}{\sqrt n}\Big) \right]
$$
is equal to $\Ep[m(W)]$. Here, Lemma \ref{lem: rand lind} is applied to the first two terms of $W_i^\sigma$ conditional on the set $I_d$ and $\{Z_j : j \notin I_d\}$. We will use this property extensively below without explicit mentioning.

Now, fix $y\in\mathbb R^p$ and observe that
$$
\mathcal I^{y} = \sum_{i=1}^{|I_d|}\left( m^y\left(W_i^{\sigma} + \frac{V_{\sigma(i)}}{\sqrt n}\right) - m^y\left(W_i^{\sigma} + \frac{Z_{\sigma(i)}}{\sqrt n}\right) \right).
$$
Hence, letting $f\colon[0,1]\to\mathbb R$ be a function defined by
$$
f(t) = \sum_{i=1}^{|I_d|}\Ep\left[ m^y\left(W_i^{\sigma} + \frac{t V_{\sigma(i)}}{\sqrt n}\right) - m^y\left(W_i^{\sigma} + \frac{t Z_{\sigma(i)}}{\sqrt n}\right) \right],\ \text{for all }t\in[0,1],
$$
it follows that
$
\Ep[\mathcal I^{y}] = f(1)
$
and by Taylor's expansion,
$$
f(1) = f(0) + f^{(1)}(0) + \frac{f^{(2)}(0)}{2} + \frac{f^{(3)}(0)}{6} + \frac{f^{(4)}(\tilde t)}{24},\ \text{where }\tilde t\in(0,1).
$$
Here, $f(0) = 0$ by construction and $f^{(1)}(0) = 0$ because $\Ep[V_{i j}] = \Ep[Z_{i j}] = 0$ for all $i\in I_d$ and $j=1,\dots,p$. We thus need to bound $|f^{(2)}(0)|$, $|f^{(3)}(0)|$, and $|f^{(4)}(\tilde t)|$. To this end, we show in Steps 3, 4, and 5 that
\begin{align}
|f^{(2)}(0)| & \lesssim \frac{B_n^{2}\phi^4\log^{5}(p n)}{n^{2}} \nonumber\\
&\quad + \Big(\Ep[\varrho_{\epsilon^{d+1}}] + \frac{\sqrt{\log p}}{\phi}+\delta\Big) \Big(\frac{\mathcal B_{n,1,d}\phi^2\log p}{\sqrt n} + \frac{B_n^2\phi^4\log^3(p n)}{n}\Big),\label{eq: f2 bound}
\end{align}
\begin{align}
|f^{(3)}(0)| &\lesssim \frac{B_n^{2}\phi^4 \log^{5}(p n)}{n^{2}} \nonumber\\
&\quad + \Big(\Ep[\varrho_{\epsilon^{d+1}}] + \frac{\sqrt{\log p}}{\phi}+\delta\Big)\Big(\frac{\mathcal B_{n,2,d}\phi^3\log^2p}{n} + \frac{B_n^3\phi^5\log^5(pn)}{n^{3/2}}\Big),\label{eq: f3 bound}
\end{align}
and
\begin{equation}
|f^{(4)}(\tilde t)| \lesssim \frac{B_n^2\phi^4\log^3 p}{n^2} + \left( \Ep[\varrho_{\epsilon^{d+1}}] + \frac{\sqrt{\log p}}{\phi}+\delta \right) \frac{B_n^2\phi^4\log^3 p}{n},\label{eq: f4 bound}
\end{equation}
respectively. Combining these inequalities gives \eqref{eq: EIy bound} and completes Step 1.

\medskip
\noindent
{\bf Step 2.} Here, we prove \eqref{eq: anticoncentration application}. Fix $y\in\mathbb R^p$ and observe that
\begin{align*}
&\Pr(S_{n,\epsilon^d}^V \leq y) \leq\Pr(F( S_{n,\epsilon^d}^V - y - \phi^{-1}) \leq 0)
\leq \Ep[ m^{y + \phi^{-1}}( S_{n,\epsilon^d}^V) ]\\
&\quad \leq \Ep[ m^{y + \phi^{-1}}( S_n^{Z})] + |\Ep[\mathcal I^{y+\phi^{-1}}]|
 \leq \Pr( S_n^{Z} \leq y + 2\phi^{-1}) + |\Ep[\mathcal I^{y+\phi^{-1}}]|\\
&\quad \leq \Pr( S_n^{Z} \leq y) + 2C_a\phi^{-1}\sqrt{\log p} +C_a\delta+ |\Ep[\mathcal I^{y+\phi^{-1}}]|,
\end{align*}
where the first inequality follows from \eqref{eq: f properties}, the second from $m^{y + \phi^{-1}}(\cdot) = g(F(\cdot - y - \phi^{-1}))$ and \eqref{eq: g property}, the third from \eqref{eq: iy}, the fourth from \eqref{eq: g property} and \eqref{eq: f properties}, and the fifth from Condition A. Similarly,
\begin{align*}
&\Pr(S_{n,\epsilon^d}^V \leq y) = \Pr(S_{n,\epsilon^d}^V - y \leq 0)\\
&\quad \geq \Pr(F(S_{n,\epsilon^d}^V - y + \phi^{-1}) \leq \phi^{-1})
\geq \Ep[ m^{y - \phi^{-1}}( S_{n,\epsilon^d}^V) ]\\
&\quad \geq \Ep[ m^{y - \phi^{-1}}( S_n^{Z})] - |\Ep[\mathcal I^{y-\phi^{-1}}]|
 \geq \Pr( S_n^{Z} \leq y - 2\phi^{-1}) - |\Ep[\mathcal I^{y-\phi^{-1}}]|\\
&\quad \geq \Pr( S_n^{Z} \leq y) - 2C_a\phi^{-1}\sqrt{\log p} -C_a\delta - |\Ep[\mathcal I^{y-\phi^{-1}}]|.
\end{align*}
Combining the presented bounds gives \eqref{eq: anticoncentration application} and completes Step 2.

\medskip
\noindent
{\bf Step 3.} Here, we prove \eqref{eq: f2 bound}. We have
\begin{align}
f^{(2)}(0) &= \frac{1}{n}\sum_{i=1}^{|I_d|}\sum_{j, k=1}^p \Ep\Big[ m^y_{j k}(W_i^{\sigma})(V_{\sigma(i) j}V_{\sigma(i) k} - Z_{\sigma(i) j}Z_{\sigma(i) k})\Big]\nonumber\\
&= \frac{1}{n}\sum_{i\in I_d}\sum_{j, k=1}^p \Ep\Big[ m^y_{j k}(W_{\sigma^{-1}(i)}^{\sigma})(V_{i j}V_{i k} - Z_{i j}Z_{i k})\Big]\nonumber\\
& = \frac{1}{n}\sum_{i\in I_d}\sum_{j, k = 1}^p\Ep[ m^y_{j k}(W^{\sigma}_{\sigma^{-1}(i)})] (\mathcal E_{i,j k}^V - \mathcal E_{i,j k}^Z),\label{eq: independence argument}
\end{align}
where the third line follows from observing that conditional on $\sigma$, $W_{\sigma^{-1}(i)}^{\sigma}$ is independent of $V_{i j}V_{i k} - Z_{i j}Z_{i k}$. Thus, denoting
\begin{align*}
R_{i,jk}^{\sigma} & = m_{jk}^y(W_{\sigma^{-1}(i)}^{\sigma}) -\frac{\sigma^{-1}(i)}{|I_d|+1}m^y_{jk}\left(W_{\sigma^{-1}(i)}^{\sigma} + \frac{V_{i}}{\sqrt n}\right)\\
&\quad - \left(1-\frac{\sigma^{-1}(i)}{|I_d|+1}\right)m^y_{jk}\left(W_{\sigma^{-1}(i)}^{\sigma} + \frac{Z_{i}}{\sqrt n}\right),
\end{align*}
for all $i \in I_d$ and $j,k = 1,\dots,p$, we have $f^{(2)}(0) = \mathcal I_{2,1} + \mathcal I_{2,2}$, where
$$
\mathcal I_{2,1} = \frac{1}{n}\sum_{i\in I_d}\sum_{j,k=1}^p\Ep[m^y_{jk}(W)](\mathcal E_{i, j k}^V - \mathcal E_{i,j k}^Z)
$$
and
$$
\mathcal I_{2,2} = \frac{1}{n}\sum_{i\in I_d}\sum_{j,k=1}^p\Ep[R_{i,jk}^{\sigma}](\mathcal E_{i, j k}^V - \mathcal E_{i,j k}^Z)
$$
by our discussion in the beginning of Step 1. To bound $\mathcal I_{2,1}$, we have
\begin{align*}
|\mathcal I_{2,1}| &\leq \sum_{j,k=1}^p\Ep[|m^y_{j k}(W)|]\max_{1\leq j,k\leq p}\left|\frac{1}{n}\sum_{i=1}^n\epsilon_i^d(\mathcal E_{i, j k}^V - \mathcal E_{i,j k}^Z)\right|\\
& \leq \frac{\mathcal B_{n,1,d}}{\sqrt n} \sum_{j,k=1}^p\Ep[|m^y_{j k}(W)|]
\end{align*}
by the definition of $\mathcal A_d$. In addition, by the definition of $m^y$, we have $m^y_{jk}(W) = 0$ if
$$
\max_{1\leq j\leq p}(W_j - y_j) \leq -\phi^{-1}\quad\text{or}\quad\max_{1\leq j\leq p}(W_j  - y_j) > \phi^{-1},
$$
which means that 
\begin{equation}\label{eq: h switching}
m^y_{jk}(W) = h^y(W; \phi^{-1})m^y_{jk}(W)
\end{equation}
by the definition of $h^y$ in \eqref{eq: h definition}. Thus, since
$$
\mathcal P = \Pr\left( -\phi^{-1} < \max_{1\leq j\leq p}\frac{1}{\sqrt n}\sum_{i=1}^n (Z_{i j} - y_j) \leq \phi^{-1} \right) \leq C_a\left(\frac{2\sqrt{\log p}}{\phi}+\delta\right)
$$
by Condition A, it follows that $\sum_{j,k=1}^p\Ep[|m^y_{j k}(W)|]$ is equal to
\begin{align}
& \sum_{j,k=1}^p\Ep[h(W; \phi^{-1})|m^y_{j k}(W)|]
\leq \sum_{j,k=1}^p\Ep[h(W; \phi^{-1})U_{j k}(W)] \nonumber \\
&\quad \lesssim (\phi^2\log p)\Pr\left( -\phi^{-1} < \max_{1\leq j\leq p}(W_j - y_j) \leq \phi^{-1} \right)\nonumber\\
&\quad \leq (\phi^2\log p)\left(2\Ep[\varrho_{\epsilon^{d+1}}] + \mathcal P\right) \lesssim (\phi^2\log p)\left( \Ep[\varrho_{\epsilon^{d+1}}] + \frac{\sqrt{\log p}}{\phi} +\delta\right),\label{eq: induction trick}
\end{align}
where the first inequality follows from \eqref{eq: property 1 - 2 and 3}, the second from \eqref{eq: property 3 - 2 and 3}, and the third from the definition of $\varrho_{\epsilon^{d+1}}$ in \eqref{eq: rho eps definition} and \eqref{eq: W definition}. Hence,
$$
|\mathcal I_{2,1}| \lesssim \frac{\mathcal B_{n,1,d}\phi^2\log p}{\sqrt n}\left( \Ep[\varrho_{\epsilon^{d+1}}] + \frac{\sqrt{\log p}}{\phi} +\delta\right).
$$
To bound $\mathcal I_{2,2}$, by another Taylor's expansion, for all $i \in I_d$ and $j,k=1,\dots,p$, we have
\begin{align*}
|\Ep[R_{i,jk}^{\sigma}]| 
&\leq \sum_{l,r=1}^p\Ep\left[ \left| m^y_{jklr}\left( W_{\sigma^{-1}(i)}^{\sigma} + \frac{\hat t V_i}{\sqrt n} \right)\frac{V_{il}V_{ir}}{n}\right| \right] \\
&\quad +  \sum_{l,r=1}^p\Ep\left[\left| m^y_{jklr}\left( W_{\sigma^{-1}(i)}^{\sigma} + \frac{\hat t Z_i}{\sqrt n} \right)\frac{Z_{il}Z_{ir}}{n}\right| \right] 
\end{align*}
for some $\hat t\in(0,1)$, possibly depending on $i$, $j$, and $k$, and so $|\mathcal I_{2,2}| \leq \mathcal I_{2,2,1} + \mathcal I_{2,2,2}$, where
$$
\mathcal I_{2,2,1} = \frac{1}{n^2}\sum_{i \in I_d}\sum_{j,k,l,r=1}^p\Ep\left[ \left|m_{jklr}^y\left( W_{\sigma^{-1}(i)}^{\sigma} + \frac{\hat t V_i}{\sqrt n} \right)V_{i l}V_{i r}\right| \right]\times|\mathcal E_{i,jk}^V - \mathcal E_{i,jk}^Z|
$$
and
$$
\mathcal I_{2,2,2} = \frac{1}{n^2}\sum_{i\in I_d}\sum_{j,k,l,r=1}^p\Ep\left[ \left|m_{jklr}^y\left( W_{\sigma^{-1}(i)}^{\sigma} + \frac{\hat t Z_i}{\sqrt n} \right)Z_{i l}Z_{i r}\right| \right]\times|\mathcal E_{i,jk}^V - \mathcal E_{i,jk}^Z|.
$$
Below, we bound $\mathcal I_{2,2,1}$ and note that the same argument also applies to $\mathcal I_{2,2,2}$. Denote $x = C_pB_n\log(p n)/\sqrt n + \phi^{-1}$ and $\tilde V_i = 1\{\|V_i\|_{\infty}\leq C_pB_n\log(p n)\}$ for all $i \in I_d$. Then for all $i \in I_d$ and $j,k = 1,\dots,p$, we have
\begin{align}
&\sum_{l,r=1}^p\Ep\left[ \tilde V_i \left| m^y_{jklr}\left( W_{\sigma^{-1}(i)}^{\sigma} + \frac{\hat t V_i}{\sqrt n} \right)V_{il}V_{ir}\right| \right]\nonumber\\
&\quad = \sum_{l,r=1}^p\Ep\left[ \tilde V_i h^y(W^{\sigma}_{\sigma^{-1}(i)}; x) \left| m^y_{jklr}\left( W_{\sigma^{-1}(i)}^{\sigma} + \frac{\hat t V_i}{\sqrt n} \right)V_{il}V_{ir}\right| \right]\nonumber\\
&\quad\leq \sum_{l,r=1}^p\Ep\left[ \tilde V_i h^y(W^{\sigma}_{\sigma^{-1}(i)}; x) U^y_{jklr}\left( W_{\sigma^{-1}(i)}^{\sigma} + \frac{\hat t V_i}{\sqrt n} \right)|V_{il}V_{ir}|\right]\nonumber\\
&\quad\lesssim \sum_{l,r=1}^p\Ep\left[ h^y(W^{\sigma}_{\sigma^{-1}(i)}; x) U^y_{jklr}(W^{\sigma}_{\sigma^{-1}(i)})|V_{il}V_{ir}|\right],\label{eq: intermediate 1}
\end{align}
where the equality follows from the same argument as that leading to \eqref{eq: h switching}, the first inequality follows from \eqref{eq: property 1 - 3 and 4} and the second from \eqref{eq: phi restriction} and \eqref{eq: property 2 - 4 and 5}. In turn, denoting $\tilde Z_i = 1\{\|Z_i\|_{\infty} \leq C_p B_n\log(p n)\}$, it follows that for all $l,r=1,\dots,p$, the expectation in \eqref{eq: intermediate 1} is equal to
\begin{align}
&\Ep\Big[ h^y(W^{\sigma}_{\sigma^{-1}(i)}; x) U^y_{jklr}(W^{\sigma}_{\sigma^{-1}(i)})\Big]\Ep[|V_{il}V_{ir}|]\nonumber\\
&\quad\lesssim \Ep\Big[ \tilde V_i \tilde Z_i h^y(W^{\sigma}_{\sigma^{-1}(i)}; x) U^y_{jklr}(W^{\sigma}_{\sigma^{-1}(i)})\Big]\Ep[|V_{il}V_{ir}|]\nonumber\\
&\quad\lesssim \Ep\Big[ \tilde V_i \tilde Z_ih^y(W; 2x) U^y_{jklr}(W^{\sigma}_{\sigma^{-1}(i)})\Big]\Ep[|V_{il}V_{ir}|]\nonumber\\
&\quad\lesssim \Ep\Big[ h^y(W; 2x) U^y_{jklr}(W)\Big]\Ep[|V_{il}V_{ir}|],\label{eq: intermediate 2 x}
\end{align}
where the first inequality follows from Condition P, the second from the definitions of $h^y$ in \eqref{eq: h definition}, $W$ in \eqref{eq: W definition}, and $W^{\sigma}_{\sigma^{-1}(i)}$ in the beginning of Step 1, and the third from \eqref{eq: phi restriction} and \eqref{eq: property 2 - 4 and 5}. Thus,
\begin{align*}
&\frac{1}{n^2}\sum_{i\in I_d}\sum_{j,k,l,r = 1}^p \Ep\left[ \tilde V_i \left| m^y_{jklr}\left( W_{\sigma^{-1}(i)}^{\sigma} + \frac{\hat t V_i}{\sqrt n} \right)V_{il}V_{ir}\right| \right]\times|\mathcal E_{i,jk}^V - \mathcal E_{i,jk}^Z|\\
&\quad \lesssim \frac{1}{n^2}\sum_{i\in I_d}\sum_{j,k,l,r=1}^p \Ep\Big[ h^y(W; 2x) U^y_{jklr}(W)\Big]\Ep[|V_{il}V_{ir}|]\times|\mathcal E_{i,jk}^V - \mathcal E_{i,jk}^Z|\\
&\quad \lesssim \frac{B_n^2}{n}\sum_{j,k,l,r = 1}^p\Ep\Big[ h^y(W; 2x) U^y_{jklr}(W)\Big] \lesssim \frac{B_n^2\phi^4\log^3p}{n}\left(\Ep[\varrho_{\epsilon^{d+1}}] + \frac{\sqrt{\log p}}{\phi} +\delta\right) 
\end{align*}
where the second inequality follows from Condition V since by H\"{o}lder's inequality,
\begin{align*}
&\max_{1\leq j,k,l,r\leq p}\frac{1}{n}\sum_{i=1}^n\Ep[|V_{il}V_{i r}|]\times|\mathcal E_{i,jk}^V - \mathcal E_{i,jk}^Z|\\
&\qquad \lesssim \max_{1 \leq j,k \leq p} \frac{1}{n}\sum_{i=1}^n\Ep[|V_{ij} V_{ik}|^2] + \max_{1 \leq j,k \leq p} \frac{1}{n}\sum_{i=1}^n\Ep[|Z_{ij} Z_{ik}|^2]\lesssim B_n^2,
\end{align*}
and the third inequality follows from \eqref{eq: phi restriction} and the the same arguments as those leading to \eqref{eq: induction trick}. In addition,
\begin{align}
&\frac{1}{n^2}\sum_{i\in I_d}\sum_{j,k,l,r = 1}^p \Ep\Big[ (1 - \tilde V_i)\Big| m^y_{jklr}\Big( W_{\sigma^{-1}(i)}^{\sigma} + \frac{\hat t V_i}{\sqrt n} \Big)V_{il}V_{ir} \Big| \Big]\times|\mathcal E_{i, j k}^V - \mathcal E^Z_{i,jk}|\nonumber\\
&\quad \lesssim \frac{B_n\phi^4\log^3 p}{n^{3/2}}\sum_{i=1}^n\Ep\Big[ (1 - \tilde V_i)\|V_{i}\|_\infty^2\Big]
\nonumber\\
&\quad \lesssim \frac{B_n\phi^4\log^3 p}{n^{1/2}}\max_{1\leq i\leq n}\Ep\left[(1 - \tilde V_i)\|V_i\|_{\infty}^2\right]\lesssim \frac{B_n^{3}\phi^4\log^{5}(p n)}{n^{5/2}}\lesssim \frac{B_n^{2}\phi^4\log^{5}(p n)}{n^{2}},\label{eq: intermediate 3 x}
\end{align}
where the first inequality follows from the fact that Condition V implies that $|\mathcal E_{i,jk}^V| + |\mathcal E_{i,jk}^Z|\lesssim \sqrt n B_n$ as well as inequalities in \eqref{eq: property 1 - 3 and 4} and \eqref{eq: property 3 - 3 and 4}, the third from noting that $\Ep[(1 - \tilde V_i)\|V_i\|_{\infty}^2] \leq (\Ep[1 - \tilde V_i])^{1/2}(\Ep[\|V_i\|_{\infty}^4])^{1/2}$ by H\"{o}lder's inequality and using Conditions P and B, and the fourth from \eqref{eq: lazy condition}. This shows that
$$
\mathcal I_{2,2,1}\lesssim \frac{B_n^2\phi^4\log^3p}{n}\left(\Ep[\varrho_{\epsilon^{d+1}}] + \frac{\sqrt{\log p}}{\phi} +\delta\right) + \frac{B_n^{2}\phi^4\log^{5}(p n)}{n^{2}}
$$
and since the same bound holds for $\mathcal I_{2,2,2}$ as well, it follows that
\begin{align*}
|\mathcal I_{2,2}| \lesssim \frac{B_n^2\phi^4\log^3p}{n}\left(\Ep[\varrho_{\epsilon^{d+1}}] + \frac{\sqrt{\log p}}{\phi} +\delta\right) + \frac{B_n^{2}\phi^4\log^{5}(p n)}{n^{2}}.
\end{align*}
Combining the bounds on $\mathcal I_{2,1}$ and $\mathcal I_{2,2}$ gives \eqref{eq: f2 bound} and completes Step 3.

\medskip
\noindent
{\bf Step 4.} Here, we prove \eqref{eq: f3 bound}. We have

\begin{align*}
f^{(3)}(0) = \frac{1}{n^{3/2}}\sum_{i\in I_d}\sum_{j, k, l=1}^p \Ep[m^y_{j k l}(W_{\sigma^{-1}(i)}^{\sigma})](\mathcal E^V_{i, j k l} - \mathcal E_{i,jkl}^Z)
\end{align*}
by the same argument as that in \eqref{eq: independence argument}. Further, denoting
\begin{align*}
R_{i,jkl}^{\sigma} & = m^y_{jkl}(W_{\sigma^{-1}(i)}^{\sigma}) -\frac{\sigma^{-1}(i)}{|I_d|+1}m^y_{jkl}\left(W_{\sigma^{-1}(i)}^{\sigma} + \frac{V_{i}}{\sqrt n}\right)\\
&\quad - \left(1-\frac{\sigma^{-1}(i)}{|I_d|+1}\right)m^y_{jkl}\left(W_{\sigma^{-1}(i)}^{\sigma} + \frac{Z_{i}}{\sqrt n}\right),
\end{align*}
for all $i \in I_d$ and $j,k,l = 1,\dots,p$, we have $f^{(3)}(0) = \mathcal I_{3,1} + \mathcal I_{3,2}$, where
$$
\mathcal I_{3,1} = \frac{1}{n^{3/2}}\sum_{i\in I_d}\sum_{j,k,l=1}^p\Ep[m^y_{jkl}(W)](\mathcal E^V_{i, j k l} - \mathcal E_{i,jkl}^Z)
$$
and
$$
\mathcal I_{3,2} = \frac{1}{n^{3/2}}\sum_{i\in I_d}\sum_{j,k,l=1}^p\Ep[R_{i,jkl}^{\sigma}](\mathcal E^V_{i, j k l} - \mathcal E_{i,jkl}^Z).
$$
Here, $|\mathcal I_{3,1}|$ can be bounded using the same arguments as those used to bound $|\mathcal I_{2,1}|$ in the previous step. This gives
$$
|\mathcal I_{3,1}| \lesssim \frac{\mathcal B_{n,2,d}\phi^3\log^2 p}{n}\left( \Ep[\varrho_{\epsilon^{d+1}}] + \frac{\sqrt{\log p}}{\phi} +\delta\right).
$$
To bound $|\mathcal I_{3,2}|$, we have like in the case of $|\mathcal I_{2,2}|$ in the previous step that $|\mathcal I_{3,2}|\leq\mathcal I_{3,2,1} + \mathcal I_{3,2,2}$, where
$$
\mathcal I_{3,2,1} = \frac{1}{n^{5/2}}\sum_{i \in I_d}\sum_{j,k,l,r,h=1}^p\Ep\left[ \left|m_{jklrh}^y\left( W_{\sigma^{-1}(i)}^{\sigma} + \frac{\hat t V_i}{\sqrt n} \right)V_{i r}V_{i h}\right| \right]\times|\mathcal E_{i,jkl}^V - \mathcal E_{i,jkl}^Z|
$$
and
$$
\mathcal I_{3,2,2} = \frac{1}{n^{5/2}}\sum_{i\in I_d}\sum_{j,k,l,r,h=1}^p\Ep\left[ \left|m_{jklrh}^y\left( W_{\sigma^{-1}(i)}^{\sigma} + \frac{\hat t Z_i}{\sqrt n} \right)Z_{i r}Z_{i h}\right| \right]\times|\mathcal E_{i,jkl}^V - \mathcal E_{i,jkl}^Z|.
$$
Further, since
\begin{align}
|\mathcal E_{i,jkl}^V| &\leq \Ep[|V_{ij}V_{ik}V_{il}|] = \Ep\Big[\tilde V_i |V_{ij}V_{ik}V_{il}|\Big] + \Ep\Big[(1 - \tilde V_i) |V_{ij}V_{ik}V_{il}|\Big]\nonumber\\
&\lesssim B_n\log(p n)\Ep[|V_{ij}V_{ik}|] + (\Ep[1-\tilde V_i])^{1/2}(\Ep[\|V_i\|_{\infty}^6])^{1/2}\nonumber\\
&\lesssim B_n\log(p n)\Ep[|V_{ij}V_{ik}|] + B_n^3\log^{3}(p n)/n^2\label{eq: details moment calculations}
\end{align}
and similarly
$$
|\mathcal E_{i,jkl}^Z| \lesssim B_n\log(p n)\Ep[|Z_{ij}Z_{ik}|] + B_n^3\log^{3}(p n)/n^2
$$
by Conditions P and B, we have by the same argument as in the previous step that
\begin{align*}
&\frac{1}{n^{5/2}}\sum_{i\in I_d}\sum_{j,k,l,r,h = 1}^p \Ep\left[ \tilde V_i \left| m^y_{jklrh}\left( W_{\sigma^{-1}(i)}^{\sigma} + \frac{\hat t V_i}{\sqrt n} \right)V_{ir}V_{ih}\right| \right]\times|\mathcal E_{i,jkl}^V - \mathcal E_{i,jkl}^Z|\\
&\quad \lesssim \frac{1}{n^{5/2}}\sum_{i\in I_d}\sum_{j,k,l,r,h=1}^p \Ep\Big[ h^y(W; 2x) U^y_{jklrh}(W)\Big]\Ep[|V_{ir}V_{ih}|]\times|\mathcal E_{i,jkl}^V - \mathcal E_{i,jkl}^Z|\\
&\quad \lesssim  \left( \frac{B_n^3\phi^5\log^5(pn)}{n^{3/2}} + \frac{B_n^4\phi^5\log^{7}(pn)}{n^{7/2}} \right)\left(\Ep[\varrho_{\epsilon^{d+1}}] + \frac{\sqrt{\log p}}{\phi} +\delta\right)\\
& \quad \lesssim \frac{B_n^3\phi^5\log^5(pn)}{n^{3/2}}\left(\Ep[\varrho_{\epsilon^{d+1}}] + \frac{\sqrt{\log p}}{\phi} +\delta\right),
\end{align*}
where in the second inequality, we used $n^{-1}\sum_{i\in I_d}\Ep[|V_{ir}V_{ih}|]\lesssim B_n$, which follows from  Condition V, and  the third inequality follows from \eqref{eq: lazy condition} since $B_n\geq 1$. Also, again like in the previous step,
\begin{align*}
&\frac{1}{n^{5/2}}\sum_{i\in I_d}\sum_{j,k,l,r,h = 1}^p \Ep\left[ (1 - \tilde V_i) \left| m^y_{jklrh}\left( W_{\sigma^{-1}(i)}^{\sigma} + \frac{\hat t V_i}{\sqrt n} \right)V_{ir}V_{ih}\right| \right] \\
& \quad \times|\mathcal E_{i,jkl}^V - \mathcal E_{i,jkl}^Z|
\lesssim \frac{B_n^{3/2}\phi^5\log^4 p}{n^{3/4}}\max_{1\leq i\leq n}\Ep\Big[ (1 - \tilde V_i)\|V_i\|_{\infty}^2 \Big]\\
&\quad \lesssim \frac{B_n^{7/2}\phi^5 \log^{6}(p n)}{n^{11/4}} \lesssim \frac{B_n^2\phi^4\log^5(p n)}{n^2},
\end{align*}
where the first inequality follows from the fact that Condition V implies that $|\mathcal E_{i,jkl}^V| + |\mathcal E_{i,jkl}^Z|\lesssim n^{3/4} B_n^{3/2}$ as well as inequalities in \eqref{eq: property 1 - 3 and 4} and \eqref{eq: property 3 - 3 and 4}, the second from noting that $\Ep[(1 - \tilde V_i)\|V_i\|_{\infty}^2] \leq (\Ep[1 - \tilde V_i])^{1/2}(\Ep[\|V_i\|_{\infty}^4])^{1/2}$ by H\"{o}lder's inequality and using Conditions P and B, and the third from \eqref{eq: phi restriction}. Thus,
$$
\mathcal I_{3,2,1} \lesssim \frac{B_n^3\phi^5\log^5(pn)}{n^{3/2}}\Big(\Ep[\varrho_{\epsilon^{d+1}}] + \frac{\sqrt{\log p}}{\phi}+\delta\Big) + \frac{B_n^2\phi^4\log^5(p n)}{n^2}
$$
and since the same bound holds for $\mathcal I_{3,2,2}$, it follows that
$$
\mathcal I_{3,2} \lesssim \frac{B_n^3\phi^5\log^5(pn)}{n^{3/2}}\Big(\Ep[\varrho_{\epsilon^{d+1}}] + \frac{\sqrt{\log p}}{\phi}+\delta\Big) + \frac{B_n^2\phi^4\log^5(p n)}{n^2}.
$$
Combining these bounds gives \eqref{eq: f3 bound} and completes Step 4.

\medskip
\noindent
{\bf Step 5.} Here, we prove \eqref{eq: f4 bound}. We have
$
f^{(4)}(\tilde t) = \mathcal I_{4,1} - \mathcal I_{4,2},
$
where
$$
\mathcal I_{4,1} =\frac{1}{n^2}\sum_{i\in I_d}\sum_{j,k,l,r=1}^p\Ep\left[ m^y_{jklr}\left(W_{\sigma^{-1}(i)}^{\sigma} + \frac{\tilde t V_i}{\sqrt n}\right)V_{ij}V_{ik}V_{il}V_{ir} \right]
$$
and
$$
\mathcal I_{4,2} = \frac{1}{n^2}\sum_{i\in I_d}\sum_{j,k,l,r=1}^p\Ep\left[ m^y_{jklr}\left(W_{\sigma^{-1}(i)}^{\sigma} + \frac{\tilde t Z_i}{\sqrt n}\right) Z_{ij}Z_{ik}Z_{il}Z_{ir} \right].
$$
Here, again denoting $x = C_pB_n\log(p n)/\sqrt n + \phi^{-1}$, we have
\begin{align*}
&\frac{1}{n^2} \sum_{i\in I_d}\sum_{j,k,l,r=1}^p\Ep\left[ \tilde V_i \left|m^y_{jklr}\left(W_{\sigma^{-1}(i)}^{\sigma} + \frac{\tilde t V_i}{\sqrt n}\right) V_{ij}V_{ik}V_{il}V_{ir}\right| \right]\\
& \quad \lesssim \frac{1}{n^2}\sum_{i\in I_d}\sum_{j,k,l,r=1}^p\Ep\Big[ \tilde V_i h^y(W_{\sigma^{-1}(i)}^{\sigma}; x)U^y_{jklr}(W_{\sigma^{-1}(i)}^{\sigma}) |V_{ij}V_{ik}V_{il}V_{ir}| \Big]\\
& \quad \leq \frac{1}{n^2}\sum_{i\in I_d}\sum_{j,k,l,r=1}^p\Ep\Big[ h^y(W_{\sigma^{-1}(i)}^{\sigma}; x) U^y_{jklr}(W_{\sigma^{-1}(i)}^{\sigma})\Big] \Ep[|V_{ij}V_{ik}V_{il}V_{ir}| ],
\end{align*}
where the first inequality follows from the same argument as that leading to \eqref{eq: intermediate 1}. In addition, for all $i \in I_d$ and $j,k,l,r = 1,\dots,p$, we have
\begin{align*}
\Ep\Big[ h^y(W_{\sigma^{-1}(i)}^{\sigma}; x) U^y_{jklr}(W_{\sigma^{-1}(i)}^{\sigma})\Big]
  \lesssim \Ep\Big[ h^y(W; 2x) U^y_{jklr}(W)\Big]
\end{align*}
by the same argument as that leading to \eqref{eq: intermediate 2 x}.
Hence,
\begin{align*}
&\frac{1}{n^2}\sum_{i\in I_d}\sum_{j,k,l,r=1}^p\Ep\left[ \tilde V_i \left|m^y_{jklr}\left(W_{\sigma^{-1}(i)}^{\sigma} + \frac{\tilde t V_i}{\sqrt n}\right) V_{ij}V_{ik}V_{il}V_{ir}\right| \right]\\
& \quad \lesssim \frac{1}{n^2}\sum_{j,k,l,r=1}^p\Ep[ h^y(W; 2x) U^y_{jklr}(W)] \max_{1\leq j,k,l,r\leq p} \sum_{i=1}^n \Ep[|V_{ij}V_{ik}V_{il}V_{ir}| ]\\
& \quad \lesssim \frac{B_n^2\phi^4\log^3p}{n}\left( \Ep[\varrho_{\epsilon^{d+1}}] + \frac{\sqrt{\log p}}{\phi} +\delta\right),
\end{align*}
where the second inequality follows from \eqref{eq: phi restriction}, \eqref{eq: property 3 - 3 and 4}, Condition V, and the same arguments as those leading to \eqref{eq: induction trick}.
In addition,
\begin{align*}
&\frac{1}{n^2}\sum_{i\in I_d}\sum_{j,k,l,r = 1}^p \Ep\left[ (1 - \tilde V_i)\left|m^y_{jklr}\left(W_{\sigma^{-1}(i)}^{\sigma} + \frac{\tilde t V_i}{\sqrt n}\right) V_{ij}V_{ik}V_{il}V_{ir}\right| \right]\\
&\quad \lesssim \frac{\phi^4\log^3 p}{n^2}\sum_{i=1}^n\Ep[(1 - \tilde V_i)\|V_i\|_{\infty}^4]
\lesssim \frac{B_n^2\phi^4\log^3p}{n^2}
\end{align*}
by \eqref{eq: lazy condition} and the arguments similar to those leading to \eqref{eq: intermediate 3 x}. Therefore,
\begin{align*}
|\mathcal I_{4,1}| \lesssim \frac{B_n^2\phi^4\log^3p}{n}\left( \Ep[\varrho_{\epsilon^{d+1}}] + \frac{\sqrt{\log p}}{\phi} +\delta\right) + \frac{B_n^2\phi^4\log^3p}{n^2},
\end{align*}
and since the same bound holds for $|\mathcal I_{4,2}|$ as well, it follows that \eqref{eq: f4 bound} holds, which completes Step 5 and the proof of the lemma.

\section{Proof of Theorem \ref{thm: stein kernel}}\label{sec: proof of kernel result}
In this proof, we will use the same notation as that used in the proof of Lemma \ref{lem: main}. In particular, we will use the constant $\phi>0$ and the functions $m^y$ and $h^y$. Moreover, we will use indices to denote partial derivatives, e.g. $m_{jk}^y(w) = \partial^2m^y(w)/\partial w_j\partial w_k$. Throughout the proof, we will assume, without loss of generality, that $V$ and $Z$ are independent. We proceed in two steps.

\medskip
\noindent
{\bf Step 1.} Denote
$$
\mathcal I^y = m^y(V) - m^y(Z),\quad\text{for all }y\in\mathbb R^p.
$$
It then follows from exactly the same arguments as those in Step 2 of the proof of Lemma \ref{lem: main}, with Lemma \ref{lem: anticoncetration} playing the role of Condition A, that
$$
\sup_{y\in\mathbb R^p}\Big|\Pr(V\leq y) - \Pr(Z\leq y)\Big| \leq C_1\left(\frac{\sqrt{\log p}}{\phi} + \sup_{y\in\mathbb R^p}|\Ep[\mathcal I^y]|\right),
$$
where $C_1$ is a constant depending only on $c$. In addition, we will prove in Step 2 below that
\begin{equation}\label{eq: stein kernel step 2}
\sup_{y\in\mathbb R^p}|\Ep[\mathcal I^y]| \leq C_2\phi\Delta\log^{3/2}p,
\end{equation}
where $C_2$ is another constant depending only on $c$. Combining these inequalities and substituting $\phi = 1/(\Delta\log p)^{1/2}$ gives the asserted claim.

\medskip
\noindent
{\bf Step 2.} Here, we prove \eqref{eq: stein kernel step 2}. Fix $y\in\mathbb R^p$ and denote
$$
\Psi(t)=\Ep[m^y(\sqrt t V + \sqrt{1-t} Z)],\quad\text{for all }t\in[0,1].
$$
Then
$$
\Ep[\mathcal I^y] = \Psi(1) - \Psi(0) = \int_0^1\Psi'(t)dt,
$$
where
$$
\Psi'(t) = \frac{1}{2}\sum_{j=1}^p\Ep\left[ m_j^y(\sqrt t V + \sqrt{1-t} Z)\left(\frac{V_j}{\sqrt t} - \frac{Z_j}{\sqrt{1-t}}\right) \right].
$$
Also, since $\tau$ is the Stein kernel for $V$,
$$
\sum_{j=1}^p\Ep\left[ m_j^y(\sqrt t V + \sqrt{1-t} Z) \frac{V_j}{\sqrt t} \right] = \sum_{j,k=1}^p \Ep[m_{jk}^y(\sqrt t V + \sqrt{1-t} Z)\tau_{jk}(V) ]
$$
and by the multivariate Stein identity,
$$
\sum_{j=1}^p\Ep\left[ m_j^y(\sqrt t V + \sqrt{1-t} Z) \frac{Z_j}{\sqrt{1-t}} \right] = \sum_{j,k=1}^p \Ep[m_{jk}^y(\sqrt t V + \sqrt{1-t} Z)\Sigma_{jk} ].
$$
Therefore,
$$
\Psi'(t) = \frac{1}{2}\sum_{j,k=1}^p \Ep\Big[ m_{jk}^y(\sqrt t V + \sqrt{1-t} Z)(\tau_{j k}(V) - \Sigma_{jk}) \Big].
$$
In addition, by \eqref{eq: h switching},
$$
m_{jk}^y(w) = h^y(w;\phi^{-1})m_{jk}^y(w)
$$
for all $w\in \mathbb R^p$. Substituting here $w = \sqrt t V + \sqrt{1-t} Z$ and using the definition of $h^y$ in \eqref{eq: h definition}, we obtain
$$
m_{jk}^y(\sqrt t V + \sqrt{1-t} Z) = h^{y(V,t)}\left(Z,\frac{1}{\phi\sqrt{1 - t}}\right)m_{jk}^y(\sqrt t V + \sqrt{1-t} Z),
$$
where
$$
y(V,t) = \frac{y - \sqrt t V}{\sqrt{1 - t}}.
$$
Hence, using \eqref{eq: property 1 - 2 and 3} and \eqref{eq: property 3 - 2 and 3}, we have
\begin{align*}
&|\Psi'(t)|
 \leq K_1\phi^2(\log p) \Ep\left[ h^{y(V,t)}\left(Z,\frac{1}{\phi\sqrt{1-t}}\right)\times\max_{1\leq j,k\leq p}|\tau_{jk}(V) - \Sigma_{jk}| \right]\\
&\quad = K_1\phi^2(\log p)\Ep\left[\Ep\left[ h^{y(V,t)}\left(Z,\frac{1}{\phi\sqrt{1-t}}\right) \mid V\right]\times \max_{1\leq j,k\leq p}|\tau_{jk}(V) - \Sigma_{jk}|\right]\\
&\quad \leq \frac{K_2\phi\log^{3/2}p}{\sqrt{1-t}}\Ep\left[ \max_{1\leq j,k\leq p}|\tau_{jk}(V) - \Sigma_{jk}| \right]
\end{align*}
by the law of iterated expectations and Lemma \ref{lem: anticoncetration}, where $K_1$ is a universal constant and $K_2$ is a constant depending only on $c$. Conclude that
$$
|\Ep[\mathcal I^y]| \leq \int_0^1\Psi'(t)dt \leq 2K_2\phi\Delta\log^{3/2}p,
$$
which gives \eqref{eq: stein kernel step 2} and completes Step 2 and the proof of the theorem.

\section{Proof of Lemma \ref{lem: all conditions}}\label{sec: proof lemma 52}

Fix $m\in\{1,2,3,4\}$ and let $\mathcal P = \{1,\dots,p\}^m$. Also, for any $y = (y_1,\dots,y_p)'\in\mathbb R^p$ and $h = (h_1,\dots,h_m)'\in\mathcal P$, denote $y^h = y_{h_1}\cdots y_{h_m}$. Then note that there exists a constant $A_1\geq 1$ depending only on $b_2$ such that
\begin{align*}
\max_{h\in\mathcal P}\frac{1}{n}\sum_{i=1}^n\Ep[(X_i^h - \Ep[X_i^h])^2] 
&\leq \max_{h\in\mathcal P}\frac{1}{n}\sum_{i=1}^n\Ep[(X_i^h)^2] \\
& \leq A_1^2 B_n^{\{2(m-1)\}\vee1}\log^{2(m-2)\vee0}(p n)
\end{align*}
by Conditions M and E, Lemma \ref{lem: subgaussian growth}, and calculations similar to those in \eqref{eq: details moment calculations}. Also, by standard calculations (see Lemma 2.2.2 and discussion on page 95 of \cite{VW96}), for some universal constant $A_2\geq 1$,
$$
\Ep\Big[\max_{1\leq i\leq n}\max_{h\in\mathcal P}(X_i^h - \Ep[X_i^h])^2\Big] \leq A_2^2 B_n^{2m}\log^{2m}(p n)
$$
by Condition E. Hence, by Lemma \ref{lem: maximal ineq},
\begin{align*}
&\Ep\left[\max_{h\in\mathcal P}\left| \frac{1}{\sqrt n}\sum_{i=1}^n(X_i^h - \Ep[X_i^h]) \right|\right] \\
&\quad \leq K_1\left( A_1 B_n^{(m-1)\vee(1/2)}\log^{(m-3/2)\vee(1/2)}(p n) + \frac{A_2 B_n^m\log^{m+1}(pn)}{\sqrt n} \right)\\
&\quad  \leq K_1B_n^{(m-1)\vee(1/2)}(A_1 + A_2)\log^{(m-3/2)\vee(1/2)}(pn)
\end{align*}
for some universal constant $K_1\geq 1$, where the second inequality follows from \eqref{eq: bn restriction once again}. Thus, applying Lemma \ref{lem: fuk-nagaev} with $\eta = 1$, $\beta = 1/4$, and $t = 3K_1 B_n^{(m-1)\vee(1/2)} (A_1 + A_2)\log^{(m-3/2)\vee(1/2)}(pn)$ shows that
$$
\max_{h\in\mathcal P}\left|\frac{1}{\sqrt n}\sum_{i=1}^n(X_i^h - \Ep[X_i^h])\right| > 5K_1 B_n^{(m-1)\vee(1/2)} (A_1 + A_2)\log^{(m-3/2)\vee(1/2)}(pn)
$$
with probability at most
\begin{align*}
&\exp(-3\log(p n)) + 3\exp\left( -K_2\left(\frac{B_n^{(m-1)\vee(1/2)}\log^{(m-3/2)\vee(1/2)}(pn) }{ B_n^m\log^{m}(p n)/\sqrt n} \right)^{1/4}\right)\\
&\quad \leq (p n)^{-3} + 3\exp\left( -K_2/(\sqrt c\upsilon_n)^{1/8}\right) \\
&\quad \leq(p n)^{-3} + 3\exp\left( -8/\upsilon_n^{1/8}\right)
\leq 1/(4n)+3\upsilon_n/4,
\end{align*}
for some universal constant $K_2>0$, where the first inequality follows from \eqref{eq: bn restriction once again}, the second holds if $c=(1\wedge(K_2/8))^{16}$, and the third holds since $4^{1/8}x\leq\exp(x)$ for all $x\geq 0$. Thus, for $A_3 = 5K_1(A_1 + A_2)$, letting $\mathcal A$ be the event that the inequalities
$$
\max_{1\leq j\leq p}\left|\frac{1}{\sqrt n}\sum_{i=1}^n X_{i j}\right| \leq A_3\sqrt{B_n\log(p n)},
$$
$$
\max_{1\leq j,k\leq p}\left|\frac{1}{\sqrt n}\sum_{i=1}^n(X_{i j}X_{i k} - \Ep[X_{i j}X_{i k}])\right| \leq A_3B_n\sqrt{\log(p n)},
$$
$$
\max_{1\leq j,k,l\leq p}\left|\frac{1}{\sqrt n}\sum_{i=1}^n(X_{i j}X_{i k}X_{il} - \Ep[X_{i j}X_{i k}X_{il}])\right| \leq A_3B_n^2\log^{3/2}(p n),
$$
$$
\max_{1\leq j,k,l,r\leq p}\left|\frac{1}{\sqrt n}\sum_{i=1}^n(X_{i j}X_{i k}X_{il}X_{ir} - \Ep[X_{i j}X_{i k}X_{il}X_{ir}])\right| \leq A_3B_n^3\log^{5/2}(p n)
$$
hold jointly, we have that the probability of $\mathcal A$ is at least $1 - 1/n-3\upsilon_n$. On the other hand, given that
\begin{align*}
&\max_{1\leq j,k\leq p}\left|\frac{1}{\sqrt n}\sum_{i=1}^n(\tilde X_{i j}\tilde X_{i k} - \Ep[X_{i j}X_{i k}])\right|\\
&\quad\leq\max_{1\leq j,k\leq p}\left|\frac{1}{\sqrt n}\sum_{i=1}^n(X_{i j} X_{i k} - \Ep[X_{i j}X_{i k}])\right| + \sqrt n \max_{1\leq j\leq p}|\bar X_{n j}|^2
\end{align*}
and
\begin{align*}
&\max_{1\leq j,k,l\leq p}\left|\frac{1}{\sqrt n}\sum_{i=1}^n(\tilde X_{i j}\tilde X_{i k}\tilde X_{i l} - \Ep[X_{ij}X_{i k}X_{i l}])\right|\\
&\quad\leq \max_{1\leq j,k,l\leq p}\left|\frac{1}{\sqrt n}\sum_{i=1}^n(X_{i j}X_{i k}X_{i l} - \Ep[X_{ij}X_{i k}X_{i l}])\right|\\
&\quad\quad + 2\sqrt n\max_{1\leq j\leq p}|\bar X_{n j}|^3 + \max_{1\leq j,k,l\leq p}|\bar X_{nl}|\times \left| \frac{3}{\sqrt n}\sum_{i=1}^n X_{i j}X_{i k} \right|,
\end{align*}
it follows that the inequalities \eqref{eq: second order deviation} and \eqref{eq: third order deviation}  with some constant $C$ depending only on $b_2$ hold on $\mathcal A$.

In addition, it follows from Condition M that the first part of \eqref{eq: upper and lower bounds for second moment} holds as long as
$$
\max_{1\leq j\leq p}\left| \frac{1}{n}\sum_{i=1}^n(\tilde X_{i j}^2 - \Ep[X_{i j}^2]) \right|\leq b_1^2/2.
$$
However, on the event \eqref{eq: second order deviation}, we have
\begin{align*}
\max_{1\leq j\leq p}\left| \frac{1}{n}\sum_{i=1}^n(\tilde X_{i j}^2 - \Ep[X_{i j}^2]) \right|
& \leq \max_{1\leq j,k\leq p}\left| \frac{1}{n}\sum_{i=1}^n(\tilde X_{i j}\tilde X_{i k} - \Ep[X_{i j}X_{i k}]) \right|\\
& \leq \frac{CB_n\sqrt{\log(p n)}}{\sqrt n} \leq \frac{C}{\log^2(p n)} \leq b_1^2/2,
\end{align*}
where the last inequality holds as long as $n\geq n_{0}$ for some constant $n_{0}$ depending only on $b_1$ and $C$. 

Finally, it follows from Condition M that the second part of \eqref{eq: upper and lower bounds for second moment} holds as long as
$$
\max_{1\leq j\leq p}\left| \frac{1}{n}\sum_{i=1}^n\tilde X_{i j}^4 - \Ep[X_{ij}^4] \right| \leq B_n^2 b_2^{2},
$$
which holds on $\mathcal A$ for all $n\geq n_0$ and some $n_0$ depending only on $b_2$ by the same arguments as those used above.  The asserted claim follows.

\section{Proof of Theorem \ref{thm: gauss and rademacher}}\label{sec: proof of theorem 23}

\begin{lemma}\label{lem: not all conditions}
Suppose that Conditions E and M are satisfied. 
Then there exist a universal constant $c\in(0,1]$ and constants $C_1\geq1$ and $n_0\in\mathbb N$ depending only on $b_1$ and $b_2$ such that for all $n\geq n_0$, the inequality
\begin{equation}\label{eq: bn restriction once again 2}
B_n^2\log^{3}(p n)\leq cn
\end{equation}
implies that 
\begin{equation}\label{eq: event randomization proof}
\|\bar X_n\|_{\infty} \leq C_1\sqrt{\frac{B_n\log(p n)}{n}}\ \text{and}\ \frac{b_1^{2}}{2}\leq \frac{1}{n}\sum_{i=1}^n X_{i j}^2,\ \text{for all }j=1,\dots,p
\end{equation}
with probability at least $1 - 1/n$.
\end{lemma}

\begin{proof}
First, by a similar argument to the proof of Lemma \ref{lem: all conditions}, there exist a universal constant $c\in(0,1]$ and a constant $C_1\geq1$ depending only on $b_1$ and $b_2$ such that, if \eqref{eq: bn restriction once again 2} holds, the first part of \eqref{eq: event randomization proof} holds with probability at least $1-1/(2n)$. 
Next, by one-sided Bernstein's inequality (cf. equation (2.23) in \cite{Wa19}), we have for any $t>0$
\begin{align*}
\Pr\left(\frac{1}{n}\sum_{i=1}^n X_{i j}^2\leq\frac{1}{n}\sum_{i=1}^n \Ep[X_{i j}^2]-t\right)
&\leq\exp\left(-\frac{nt^2}{\frac{2}{n}\sum_{i=1}^n \Ep[X_{i j}^4]}\right)\\
&\leq\exp\left(-\frac{nt^2}{2b_2^2B_n^2}\right),
\end{align*}
where the second inequality follows from Condition (M).  Taking $t=\sqrt{4b_2^2B_n^2\log(pn)/n}$, we obtain
\begin{align*}
\Pr\left(\frac{1}{n}\sum_{i=1}^n X_{i j}^2\leq\frac{1}{n}\sum_{i=1}^n \Ep[X_{i j}^2]-t\right)
\leq(pn)^{-2}.
\end{align*}
Therefore, if $t\leq b^2_1/2$, then
\begin{align*}
\Pr\left(\min_{1\leq j\leq p}\frac{1}{n}\sum_{i=1}^n X_{i j}^2\leq\frac{b_1^2}{2}\right)
\leq p(pn)^{-2}\leq1/(2n).
\end{align*}
Thus the second part of \eqref{eq: event randomization proof} holds with probability at least $1-1/(2n)$ as long as $16b_2^2B_n^2\log(pn)\leq b_1^4n$, completing the proof. 
\end{proof}

\begin{proof}[Proof of Theorem \ref{thm: gauss and rademacher}]
Let $e_1,\dots,e_n$ be Rademacher weights and assume that $B_n^2\log^3(p n)\leq cn$ and $n\geq n_0$ with the same constants $c$ and $n_0$ as those in Lemma \red{\ref{lem: not all conditions}} since otherwise the asserted claim is trivial. 

Further, for all $\gamma\in(0,1)$, let $c_{1-\gamma}^{B,0}$ be the $(1-\gamma)$th quantile of the conditional distribution of 
$$
T_n^{*,0} = \max_{1\leq j\leq p}\frac{1}{\sqrt n}\sum_{i=1}^n (e_i X_{i j} + a_j)
$$
given $X_1,\dots,X_n$. Note that $T_n^{*,0}$ is defined as $T_n^*$ with replaced $X_i-\bar X_n$ by $X_i$. Observe that by Lemma \ref{lem: rademacher anticoncentration}, there exists a constant $C_2\geq 1$ depending only on $b_1$ and $b_2$ such that on the event that \eqref{eq: subgaussian bound on x} and \eqref{eq: event randomization proof} hold jointly, we have for any $t>0$ that
\begin{align}
&\sup_{x\in\mathbb R}\Pr(x\leq T_n^{*,0} \leq x+ t \mid X_{1:n})
 \leq C_2\left( t\sqrt{\log p} +  \sqrt{\frac{B_n^2\log^3(p n)}{n}}\right).\label{rad-anti}
\end{align}
Hence, given that \eqref{eq: subgaussian bound on x} holds with probability at least $1 - 1/n$ by Lemma \ref{lem: subgaussian growth}, applying Lemma \ref{lem: rand test}, which is justified by Condition S, we obtain
\begin{equation}\label{eq: app rand lemma}
\sup_{\gamma\in(0,1)}|\Pr(T_n > c_{1-\gamma}^{B,0}) - \gamma| \leq C_2\sqrt{\frac{B_n^2\log^{3}(p n)}{n}} + \frac{2}{n}.
\end{equation}
In fact, for every $s\in\{-1,1\}^n$, define the function $g_s:(\mathbb R^p)^n\to(\mathbb R^p)^n$ by $g_s(x_1,\dots,x_n)=(s_1x_1,\dots,s_nx_n)$ for $x_1,\dots,x_n\in\mathbb R^p$, and set $G=\{g_s:s\in\{-1,1\}^n\}$. Thanks to Condition S, we can check that $X=(X_1,\dots,X_n)$ and $G$ satisfy the assumptions in Lemma \ref{lem: rand test}. 
Also, denoting by $T(x)$ the value of $T_n$ when $X=x$, we have $\phi(X)=1\{T_n>c_{1-\alpha}^{B,0}\}$, where $\phi$ is the function defined in Lemma \ref{lem: rand test}. Moreover, the quantity $\Ep[\chi(X)]$ in Lemma \ref{lem: rand test} is bounded by the right-hand side of \eqref{eq: app rand lemma} because \eqref{rad-anti} holds with probability $1-2/n$. Hence, \eqref{eq: app rand lemma} indeed follows from Lemma \ref{lem: rand test}. 

In addition, for
$$
\beta_n = C_2\sqrt{\frac{B_n^2\log^{3}(p n)}{n}} + C_1C_2\sqrt{\frac{2B_n\log^2(p n)\log n}{n}},
$$
where $C_1$ is the same as in Lemma \ref{lem: not all conditions}, we have on the event that \eqref{eq: subgaussian bound on x} and \eqref{eq: event randomization proof} hold jointly that
$$
c_{1-\gamma + \beta_n}^{B,0} - c_{1-\gamma}^{B,0} \geq C_1 \sqrt{\frac{2B_n\log(p n)\log n}{n}},\quad\text{for all }\gamma\in(\beta_n,1)
$$
since otherwise we would have by \eqref{rad-anti} that
$$
\Pr(c_{1-\gamma}^{B,0} \leq T_n^{*,0} \leq c_{1-\gamma + \beta_n}^{B,0}\mid X_{1:n}) < \beta_n
$$
and simultaneously
\begin{align*}
&\Pr(c_{1-\gamma}^{B,0} \leq T_n^{*,0} \leq c_{1-\gamma + \beta_n}^{B,0}\mid X_{1:n})\\
&\quad  = \Pr(T_n^{*,0} \leq c_{1-\gamma + \beta_n}^{B,0}\mid X_{1:n}) - \Pr(T_n^{*,0} < c_{1-\gamma}^{B,0}\mid X_{1:n})\\
&\quad\geq 1-\gamma + \beta_n - (1-\gamma) = \beta_n,
\end{align*}
which is a contradiction.

Thus, on the event  that \eqref{eq: subgaussian bound on x} and \eqref{eq: event randomization proof} hold jointly, we have
\begin{align*}
&\Pr(T_n^* \leq c_{1 - \alpha + 2\beta_n}^{B,0}\mid X_{1:n})\\
& \quad \geq \Pr\left( T_n^{*,0} + C_1\left|\frac{1}{\sqrt n}\sum_{i=1}^n e_i\right|\sqrt{\frac{B_n\log(p n)}{n}} \leq c_{1 - \alpha + 2\beta_n}^{B,0}\mid X_{1:n}  \right)\\
& \quad \geq \Pr\left( T_n^{*,0} + C_1\sqrt{\frac{2B_n\log(p n)\log n}{n}} \leq c_{1 - \alpha + 2\beta_n}^{B,0}\mid X_{1:n}  \right) - 2/n\\
& \quad \geq \Pr(T_n^{*,0} \leq c_{1 - \alpha + \beta_n}^{B,0} \mid X_{1:n}) - 2/n \geq 1 - \alpha + \beta_n - 2/n > 1 - \alpha,
\end{align*}
where the first inequality follows from \eqref{eq: event randomization proof} and the second from the Hoeffding inequality. In addition, by the same arguments, again on the event  that \eqref{eq: subgaussian bound on x} and \eqref{eq: event randomization proof} hold jointly, we have
\begin{align*}
\Pr(T_n^* \leq c_{1-\alpha - 2\beta_n}^{B,0}\mid X_{1:n})
&\leq \Pr(T_n^{*,0} \leq c_{1-\alpha - \beta_n}^{B,0} \mid X_{1:n}) + 2/n\\
&\leq 1 - \alpha - \beta_n + 2/n + C_2\sqrt{\frac{B_n^2\log^{3}(p n)}{n}} < 1-\alpha.
\end{align*}
Hence,
$$
\Pr(c_{1-\alpha - 2\beta_n}^{B,0} < c_{1-\alpha}^B \leq c_{1-\alpha + 2\beta_n}^{B,0}) \geq 1 - 2/n.
$$
Conclude that
\begin{align*}
\Pr(T_n > c_{1-\alpha}^B)
&\leq \Pr(T_n > c_{1-\alpha -2\beta_n}^{B,0}) + 2/n\\
& \leq \alpha +2\beta_n +  2/n+ \beta_n \leq \alpha + 4\beta_n
\end{align*}
and
\begin{align*}
\Pr(T_n > c_{1-\alpha}^B)
& \geq \Pr(T_n > c_{1-\alpha +2\beta_n}^{B,0}) - 2/n\\
& \geq \alpha - 2\beta_n - 2/n -  \beta_n \geq \alpha - 4\beta_n,
\end{align*}
where the second lines follow from \eqref{eq: app rand lemma}. The asserted claim follows.
\end{proof}

\section{Proof of Theorem  \ref{thm: infinitesimal factors}}\label{sec: proof of theorem 24}

We proceed in three steps.

\medskip
\noindent
{\bf Step 1.} Here, we show that in the setting of Section \ref{sec: main arguments}, if Conditions V, P, and B are satisfied, then for all $y\in\mathbb R^p$, we have
\begin{align*}
&\Pr\left(\frac{1}{\sqrt n}\sum_{i=1}^n V_i \leq y\right) - \Pr\left(\frac{1}{\sqrt n}\sum_{i=1}^n Z_i \leq y + \eta/2 \right)\\
&\quad \lesssim (1\vee\eta^{-4})\left(\frac{\mathcal B_{n,1}\log p}{\sqrt n} + \frac{\mathcal B_{n,2}\log^2 p}{n} + \frac{B_n^2\log^3(p n)}{n} + \frac{B_n^2\log^{5}(p n)}{n^2}\right)
\end{align*}
up to a constant depending only on $C_v$, $C_p$, and $C_b$, where $\mathcal B_{n,1}$ and $\mathcal B_{n,2}$ are the left-hand sides of \eqref{eq: bn bounds 1} and \eqref{eq: bn bounds 2}, respectively. To show this result, note that for all $y\in\mathbb R^p$, we have
$$
\Pr\left(\frac{1}{\sqrt n}\sum_{i=1}^n V_i \leq y\right) \leq \Pr\left(\frac{1}{\sqrt n}\sum_{i=1}^n Z_i \leq y + 2/\phi\right) + |\Ep[\mathcal I^{y+\phi}]|
$$ 
by Step 2 of the proof of Lemma \ref{lem: main}, where all the notations are the same as in Lemma \ref{lem: main}. So, we set $\phi = 4/\eta$ and bound $ |\Ep[\mathcal I^{y+\phi}]|$ using Steps 1, 3, 4, and 5 of the proof of Lemma \ref{lem: main} with the only difference that all terms like
$$
\Pr\left( -\phi^{-1} < \max_{1\leq j\leq p}(W_j - y_j) \leq \phi^{-1} \right)
$$
are now upper bounded by one rather than by $2\Ep[\varrho_{\epsilon^1}] + \mathcal P$. This gives the claim of this step.

\medskip
\noindent
{\bf Step 2.} Here, we show that there exists a constant $C\geq1$ depending only on $b_1$ and $b_2$ such that with probability at least $1 - 2/n-3\upsilon_n$, we have for all $x\in\mathbb R$ that
\begin{equation}\label{eq: this is tough}
\mathcal P_x = \Pr\left(T_n^* \leq x\mid X_{1:n}\right) - \Pr\left(T_n \leq x + \eta/2\right)\leq C(1\vee\eta^{-4})\left(\frac{B_n^2\log^s(p n)}{n}\right)^{1/2}
\end{equation}
where $s = 3$ if $c_{1-\alpha}^B$ is obtained via either the empirical bootstrap or the multiplier bootstrap with weights satisfying both \eqref{eq: multiplier bootstrap simplification} and \eqref{eq: third order matching multipliers} and $s = 5$ if $c_{1-\alpha}^B$ is obtained via the multiplier bootstrap with weights satisfying \eqref{eq: multiplier bootstrap simplification} only. To do so, we proceed as in the proof of Lemma \ref{thm: empirical bootstrap}, with the following differences: (i) we now use Step 1 instead of Theorem \ref{cor: max}, and (ii) in the case of the multiplier bootstrap with weights violating \eqref{eq: third order matching multipliers}, instead of \eqref{eq: Y3 bound}, we use the bound
$$
\max_{1\leq j,k,l\leq p}\left| \frac{1}{\sqrt n}\sum_{i=1}^n Y_{i j}Y_{i k}Y_{i l} \right| \leq 4 b_2^{2}B_n\sqrt{5 n\log(p n)},
$$
which follows from \eqref{eq: Y infty bound} and the additional assumption $n^{-1}\sum_{i=1}^n\Ep[X_{ij}^2]\leq b_2^2$. This leads to the following inequality: with probability at least $1 - 2 / n-3\upsilon_n$, for all $x\in\mathbb R^p$,
$$
\mathcal P_x \lesssim (1\vee\eta^{-4})\left(\frac{\mathcal B_{n,1}\log p}{\sqrt n} + \frac{\mathcal B_{n,2}\log^2 p}{n} + \frac{B_n^2\log^3(p n)}{n} + \frac{B_n^2\log^{5}(p n)}{n^2}\right)
$$
up to a constant depending only on $b_1$ and $b_2$, where $\mathcal B_{n,1} = B_n\sqrt{\log(p n)}$ in all cases and $\mathcal B_{n,2} = B_n^2\log(p n)$ in the case of the empirical and the multiplier bootstrap with weights satisfying \eqref{eq: third order matching multipliers} and $\mathcal B_{n,2} = B_n\sqrt{n\log(p n)}$ in the case of the multiplier bootstrap with weights violating \eqref{eq: third order matching multipliers}. The asserted claim of this step follows.

\medskip
\noindent
{\bf Step 3.} Here, we complete the proof. Let $\beta_n$ be the right-hand side of \eqref{eq: this is tough} and for all $\gamma\in(0,1)$, let $c_{1-\gamma}$ be the $(1-\gamma)$th quantile of $T_n$. Then by Step 2, with probability at least $1 - 2/n-3\upsilon_n$, we have
$$
\Pr(T_n^* \leq c_{1-\alpha - 2\beta_n} - \eta\mid X_{1:n}) \leq \Pr(T_n < c_{1-\alpha - 2\beta_n}) + \beta_n < 1 - \alpha.
$$
Therefore,
$$
\Pr(c_{1-\alpha}^B \geq c_{1-\alpha - 2\beta_n} - \eta) \geq 1 - 2/n-3\upsilon_n\geq1-5\upsilon_n,
$$
and so
$$
\Pr(T_n > c_{1-\alpha}^B + \eta) \leq \Pr(T_n > c_{1-\alpha-2\beta_n} ) + 5\upsilon_n \leq \alpha + 7\beta_n.
$$
The asserted claim follows.

\section{Proof of Lemma \ref{lem:dirty}}
\label{sec: proof of tech lemma}

First we derive \eqref{bdd-version} from \eqref{eq:dirty}. Choose the constant $\phi$ such that
\[
\phi^{-1}=\left(\Delta_1\log^3p\right)^{1/4}+\frac{2\kappa_n\log p}{\sqrt n}.
\]
Then we obtain
\ba{
&\phi(\log p)^2\sqrt{\Delta_1}+\phi^3(\log p)^{7/2}\Delta_1\lesssim\left(\Delta_1\log^5p\right)^{1/4},\\
&(\log p)\sqrt{\Delta_2(\phi)}+\phi(\log p)^{3/2}\Delta_2(\phi)
+\sqrt{(\log p)^{3}\Delta_3(\phi)}=0,\\
&\frac{\sqrt{\log p}}{\phi}\lesssim \left(\Delta_1\log^5p\right)^{1/4}
+\frac{\kappa_n(\log p)^{3/2}}{\sqrt n}.
}
Combining these bounds with \eqref{eq:dirty} gives \eqref{bdd-version}. 

Next we prove \eqref{eq:dirty}. 
Without loss of generality, we may assume $Z$ is independent of $X_{1:n}$ and $\tilde X_{1:n}$. 
Fix a non-increasing $C^4$ function $g_0\colon\mathbb R\to\mathbb R$ such that (i) $g_0(t)\geq 0$ for all $t\in\mathbb R$, (ii) $g_0(t) = 0$ for all $t\geq 1$, and (iii) $g_0(t) = 1$ for all $t\leq 0$. For this function, there exists a constant $C_g>0$ such that
$$
\sup_{t\in\mathbb R}\Big(|g_0^{(1)}(t)|\vee|g_0^{(2)}(t)|\vee|g_0^{(3)}(t)|\vee|g_0^{(4)}(t)|\Big) \leq C_g.
$$
In this proof, we will use the symbol $\lesssim$ to denote inequalities that hold up to a constant depending only on $c$ and $C_g$. Since $g_0$ can be chosen to be universal, we say that the inequality in the statement of the theorem holds up to a constant depending only on $c$.

Set 
$
\beta = \phi\log p.
$
Define functions $g\colon \mathbb R\to\mathbb R$ and $F\colon \mathbb R^p\to\mathbb R$ as in Appendix \ref{sec: proof of main lemma}.  
Also, for all $y\in\mathbb R^p$, define a function $m^y\colon\mathbb R^p\to\mathbb R$ as
$$
m^y(w) = g(F(w - y)),\quad\text{for all }w\in\mathbb R^p.
$$
Below, we will need partial derivatives of $m^y$ up to the fourth order. For brevity of notation, we will use indices to denote these derivatives. For example, for any $j,k,l,r=1,\dots,p$, we will write
$$
m_{jklr}^y(w)=\frac{\partial^4m^y(w)}{\partial w_j \partial w_k \partial w_l \partial w_r},\quad\text{for all }w\in\mathbb R^p.
$$

Further, for all $y\in\mathbb R^p$, define 
\begin{equation*}
\mathcal I^{y} =m^y(S_n) - m^y(Z)
\end{equation*}
and
\begin{equation*}
h^y(W; x) = 1\left\{ -x < \max_{1\leq j\leq p}(W_j - y_j) \leq x \right\},\ \text{for all }x\geq 0\text{ and }W\in\mathbb R^p.
\end{equation*}
By the same argument as those in Step 2 of the proof of Lemma \ref{lem: main}, with Lemma \ref{lem: anticoncetration} playing the role of Condition A therein, we have
\[
\sup_{y\in\mathbb R^p}\Big|\Pr(S_n\leq y) - \Pr(Z\leq y)\Big| \lesssim \frac{\sqrt{\log p}}{\phi} + \sup_{y\in\mathbb R^p}|\Ep[\mathcal I^y]|.
\]
Below we will prove
\begin{multline}\label{aim}
\sup_{y\in\mathbb R^p}|\Ep[\mathcal I^y]|
\lesssim\phi(\log p)^{3/2}\Delta_0
+\phi(\log p)^2\sqrt{\Delta_1}+\phi^3(\log p)^{7/2}\Delta_1\\
+\phi(\log p)^{3/2}\Delta_2(\phi)+(\log p)\sqrt{\Delta_2(\phi)}
+\sqrt{(\log p)^3\Delta_3(\phi)}
+\frac{\sqrt{\log p}}{\phi}.
\end{multline}
%
\smallskip

\noindent\textbf{Step 1.} 
Define a function $\psi^y:\mathbb{R}^p\to\mathbb{R}$ by
\[
\psi^y(w)=\int_0^1\frac{1}{2t}\Ep[m^y(\sqrt{t}w+\sqrt{1-t}Z)-m^y(Z)]dt,\quad w\in\mathbb{R}^p.
\]
$\psi^y$ is a solution to the following Stein equation (cf.~Lemma 1 in \cite{M09}):
\begin{equation}\label{eq:stein}
m^y(w)-\Ep[m^y(Z)]=w\cdot\nabla\psi^y(w)-\sum_{j,k=1}^p\tilde\Sigma_{jk}\partial_{jk}\psi^y(w).
\end{equation}
Also, from \eqref{eq: g property}--\eqref{eq: f properties}, we have (cf. \eqref{eq: h switching})
\[
m^y_{jk}(w)=h^y(w;\phi^{-1})m^y_{jk}(w)
\]
for all $w\in\mathbb{R}^p$. Hence we obtain
\begin{align*}
\partial_{jk}\psi^y(w)
&=\frac{1}{2}\int_0^1\Ep[m^y_{jk}(\sqrt{t}w+\sqrt{1-t}Z)]dt\\
&=\frac{1}{2}\int_0^1\Ep\left[h^{y(w,t)}\left(Z,\frac{1}{\phi\sqrt{1-t}}\right)m^y_{jk}(\sqrt{t}w+\sqrt{1-t}Z)\right]dt,
\end{align*}
where
\[
y(w,t):=\frac{y-\sqrt{t}w}{\sqrt{1-t}}.
\]
By \eqref{eq: property 1 - 2 and 3} and \eqref{eq: property 3 - 2 and 3}, we have
\[
\sum_{j,k=1}^p|m^y_{jk}(w)|\lesssim\phi^2\log p
\]
for all $w\in\mathbb{R}^p$, so we obtain
\begin{align}
\sum_{j,k=1}^p|\partial_{jk}\psi^y(w)|
&\lesssim\phi^2(\log p)\int_0^1\Ep\left[h^{y(w,t)}\left(Z,\frac{1}{\phi\sqrt{1-t}}\right)\right]dt
\nonumber\\
&\lesssim\phi(\log p)^{3/2}\int_0^1\frac{1}{\sqrt{1-t}}dt
\lesssim\phi(\log p)^{3/2},
\label{est-psi2}
\end{align}
where the second estimate follows by Lemma \ref{lem: anticoncetration}. 
Noting Lemmas A.3--A.4 of \cite{CCK13}, we can similarly prove
\begin{equation}\label{est-psi1}
\sum_{j=1}^p|\partial_{j}\psi^y(w)|\lesssim \sqrt{\log p}.
\end{equation}
\smallskip

\noindent\textbf{Step 2.} 
Let $I$ be a uniform random variable on $\{1,\dots,n\}$ independent of everything else. 
Define $\tilde S_n:=S_n+Y_I$. It is well-known that $(S_n,\tilde S_n)$ is an exchangeable pair and satisfies 
\begin{align}
\Ep[\tilde S_n-S_n\mid X_{1:n}]&=-\frac{1}{n}S_n.\label{eq:lr}
\end{align}
See the proof of \cite[Theorem 4.1]{CCK14ga} for details. 
By exchangeability we have, with $D:=\tilde S_n-S_n=Y_I$,
\begin{align*}
0&=\frac{n}{2}\Ep[D\cdot(\nabla \psi^y(\tilde S_n)+\nabla \psi^y(S_n))1\{\|D\|_\infty\leq\beta^{-1}\}]\\
&=\Ep\left[\frac{n}{2}D\cdot(\nabla \psi^y(\tilde S_n)-\nabla \psi^y(S_n))1\{\|D\|_\infty\leq\beta^{-1}\}+nD\cdot\nabla \psi^y(S_n)1\{\|D\|_\infty\leq\beta^{-1}\}\right]\\
&=\Ep\left[\frac{n}{2}\sum_{j,k=1}^p\tilde D_j\tilde D_k\partial_{jk}\psi^y(S_n)+n\tilde D\cdot\nabla \psi^y(S_n)\right]+R,
\end{align*}
where $\tilde D=D1\{\|D\|_\infty\leq\beta^{-1}\}$,
\[
R=\frac{n}{2}\sum_{j,k,l=1}^p\Ep[\tilde D_j\tilde D_k\tilde D_lU\partial_{jkl}\psi^y(S_n+(1-U)\tilde D)],
\]
and $U$ is a standard uniform random variable independent of everything else. 
Combining this with \eqref{eq:stein}--\eqref{eq:lr}, we obtain
\begin{equation}\label{est-I}
\Ep[\mathcal{I}^y]
\lesssim\sqrt{\log p}\cdot H_1+\phi(\log p)^{3/2}H_2+|R|,
\end{equation}
where
\ba{
H_1&=\Ep[\|n\Ep[D1\{\|D\|_\infty>\beta^{-1}\}\mid X_{1:n}]\|_\infty],\\
H_2&=\Ep\left[\max_{1\leq j,k\leq p}\left|\frac{n}{2}\Ep[\tilde D_j\tilde D_k\mid X_{1:n}]-\tilde\Sigma_{jk}\right|\right].
}
\smallskip

\noindent\textbf{Step 3.} It is straightforward to deduce
\ba{
H_1\leq\Ep\left[\left\|\sum_{i=1}^nY_i1\{\|Y_i\|_\infty>\beta^{-1}\}\right\|_\infty\right].
}
Hence, by Lemma \ref{lem: maximal ineq},
\ben{\label{I-est}
\sqrt{\log p}\cdot H_1\lesssim(\log p)\sqrt{\Delta_2(\phi)}+\sqrt{(\log p)^3\Delta_3(\phi)}.
}
\smallskip

\noindent\textbf{Step 4.} We bound $H_2$ as $H_2\leq H_{21}+H_{22}+\Delta_0$, where
\ba{
H_{21}&=\Ep\left[\max_{1\leq j,k\leq p}\left|\frac{n}{2}(\Ep[\tilde D_j\tilde D_k\mid X_{1:n}]-\Ep[\tilde D_j\tilde D_k])\right|\right],\\
H_{22}&=\Ep\left[\max_{1\leq j,k\leq p}\left|\frac{n}{2}\Ep[\tilde D_j\tilde D_k]-\Sigma_{jk}\right|\right].
}
It is straightforward to deduce
\ba{
H_{21}\leq\Ep\left[\max_{1\leq j,k\leq p}\left|\sum_{i=1}^n(\tilde Y_{ij}\tilde Y_{ik}-\Ep[\tilde Y_{ij}\tilde Y_{ik}])\right|\right],
}
where $\tilde Y_i=Y_i1\{\|Y_i\|_\infty\leq\beta^{-1}\}$. 
Hence, by Lemma \ref{lem: maximal ineq},
\ba{
H_{21}
&\lesssim\sqrt{(\log p)\max_{1\leq j,k\leq p}\sum_{i=1}^n\Ep\left[\tilde Y_{ij}^2\tilde Y_{ik}^2\right]}
+(\log p)\sqrt{\Ep\left[\max_{1\leq i\leq n}\max_{1\leq j,k\leq p}\tilde Y_{ij}^2\tilde Y_{ik}^2\right]}\\
&\lesssim\sqrt{(\log p)\Delta_1}+\beta^{-2}\log p.
}
Meanwhile, since $\Sigma_{jk}=2^{-1}\sum_{i=1}^n\Ep[Y_{ij}Y_{ik}]$,
\ba{
H_{22}&=\max_{1\leq j,k\leq p}\left|\frac{1}{2}\sum_{i=1}^n\Ep[Y_{ij}Y_{ik}1\{\|Y_i\|_\infty>\beta^{-1}\}]\right|
=\frac{1}{2}\Delta_2(\phi).
}
Consequently,
\ben{\label{II-est}
\phi(\log p)^{3/2}H_{2}\lesssim \phi(\log p)^2\sqrt{\Delta_1}+\frac{\sqrt{\log p}}{\phi}+\phi(\log p)^{3/2}\{\Delta_0+\Delta_2(\phi)\}.
}
\smallskip

\noindent\textbf{Step 5.} 
By exchangeability we have
\begin{align*}
&\Ep[\tilde D_j\tilde D_k\tilde D_lU\partial_{jkl}\psi^y(S_n+(1-U)\tilde D)]\\
&=\Ep[D_jD_kD_lU\partial_{jkl}\psi^y(S_n+(1-U)D)1\{\|D\|_\infty\leq\beta^{-1}\}]\\
&=-\Ep[D_jD_kD_lU\partial_{jkl}\psi^y(\tilde S_n-(1-U)D)1\{\|D\|_\infty\leq\beta^{-1}\}]\\
&=-\Ep[D_jD_kD_lU\partial_{jkl}\psi^y(S_n+UD)1\{\|D\|_\infty\leq\beta^{-1}\}]\\
&=-\Ep[\tilde D_j\tilde D_k\tilde D_lU\partial_{jkl}\psi^y(S_n+U\tilde D)].
\end{align*}
Hence we have
\begin{align*}
R&=\frac{n}{4}\sum_{j,k,l=1}^p\Ep[\tilde D_j\tilde D_k\tilde D_lU\{\partial_{jkl}\psi^y(S_n+(1-U)\tilde D)-\partial_{jkl}\psi^y(S_n+U\tilde D)\}]\\
&=\frac{n}{4}\sum_{j,k,l,r=1}^p\Ep[\tilde D_j\tilde D_k\tilde D_l\tilde D_rU(1-2U)\partial_{jklr}\psi^y(S_n+U\tilde D+U'(1-2U)\tilde D)]\\
&=\frac{1}{4}\sum_{i=1}^n\sum_{j,k,l,r=1}^p\Ep[\tilde Y_{ij}\tilde Y_{ik}\tilde Y_{il}\tilde Y_{ir}U(1-2U)\partial_{jklr}\psi^y(S_n+\hat Y_i)],
\end{align*}
where $\hat Y_i=U\tilde Y_i+U'(1-2U)\tilde Y_i$ and $U'$ is a standard uniform random variable independent of everything else. 
Note that $|U+U'(1-2U)|\leq U\vee(1-U)\leq1$ and thus $\|\hat Y_i\|_\infty\leq\|\tilde Y_i\|_{\infty}\leq\beta^{-1}$. 
From \eqref{eq: g property}--\eqref{eq: f properties}, we have for any $w_1,w_2\in\mathbb{R}^p$ with $\|w_2\|_\infty\leq\beta^{-1}$, 
\begin{align*}
&\partial_{jklr}\psi^y(w_1+w_2)\\
&=\frac{1}{2}\int_0^1\sqrt t\Ep[m^y_{jklr}(\sqrt{t}(w_1+w_2)+\sqrt{1-t}Z)]dt\\
&=\frac{1}{2}\int_0^1\sqrt t\Ep\left[h^{y(w_1,t)}\left(Z,\frac{\phi^{-1}+\beta^{-1}}{\sqrt{1-t}}\right)m^y_{jklr}(\sqrt{t}(w_1+w_2)+\sqrt{1-t}Z)\right]dt.
\end{align*}
By \eqref{eq: property 1 - 3 and 4} and \eqref{eq: property 2 - 4 and 5}, 
\begin{align*}
&|\partial_{jklr}\psi^y(w_1+w_2)|\\
&\lesssim\int_0^1\Ep\left[h^{y(w_1,t)}\left(Z,\frac{\phi^{-1}+\beta^{-1}}{\sqrt{1-t}}\right)U^y_{jklr}(\sqrt{t}w_1+\sqrt{1-t}Z)\right]dt.
\end{align*}
Hence 
\ba{
|R|&\lesssim\sum_{i=1}^n\sum_{j,k,l,r=1}^p\int_0^1\Ep\left[|\tilde Y_{ij}\tilde Y_{ik}\tilde Y_{il}\tilde Y_{ir}|h^{y(S_n,t)}\left(Z,\frac{\phi^{-1}+\beta^{-1}}{\sqrt{1-t}}\right)U^y_{jklr}(\sqrt{t}S_n+\sqrt{1-t}Z)\right]dt\\
&\lesssim\phi^4(\log p)^3\int_0^1\Ep\left[\max_{1\leq j\leq p}\sum_{i=1}^n\tilde Y_{ij}^4h^{y(S_n,t)}\left(Z,\frac{\phi^{-1}+\beta^{-1}}{\sqrt{1-t}}\right)\right]dt\\
&\lesssim\phi^3(\log p)^{7/2}\Ep\left[\max_{1\leq j\leq p}\sum_{i=1}^n\tilde Y_{ij}^4\right],
}
where the second line follows by \eqref{eq: property 3 - 3 and 4} and the last one by Lemma \ref{lem: anticoncetration} (cf.~\eqref{est-psi2}). 
By Lemma 9 in \cite{CCK15},
\ba{
\Ep\left[\max_{1\leq j\leq p}\sum_{i=1}^n\tilde Y_{ij}^4\right]
&\lesssim \max_{1\leq j\leq p}\Ep\left[\sum_{i=1}^n\tilde Y_{ij}^4\right]
+(\log p)\Ep\left[\max_{1\leq i\leq p}\|\tilde Y_{i}\|_\infty^4\right]\\
&\lesssim\Delta_1+\beta^{-4}(\log p).
}
Consequently, 
\ben{\label{est-r}
|R|\lesssim\phi^3(\log p)^{7/2}\Delta_1+\frac{\sqrt{\log p}}{\phi}.
}

Combining \eqref{est-I}, \eqref{I-est}, \eqref{II-est} and \eqref{est-r}, we obtain \eqref{aim}. 
\qed

\section{Technical Lemmas}\label{lem: technical lemmas}

\begin{lemma}[Exponential Inequality for Weighted Sums of Exchangeable Random Variables]\label{lem: exponential inequality}
Let $a_1,\dots, a_n$ be some constants in $\mathbb R$ and let $X_1,\dots,X_n$ be exchangeable random variables such that $|X_i|\leq 1$ almost surely for all $i = 1,\dots,n$. Then
$$
\Pr\left( \left| \sum_{i=1}^n a_i X_i \right| > \left| \sum_{i=1}^n a_i \right| + t \right) \leq 2\exp\left(-\frac{t^2}{32\sum_{i=1}^n a_i^2}\right),\quad\text{for all }t>0.
$$
\end{lemma}
\begin{proof}
Since the random variables $X_i$ are exchangeable, we can and will, without loss of generality, assume that
\begin{equation}\label{eq: decreasing order}
|a_{1}| \geq |a_{2}| \geq \dots\geq |a_{n}|.
\end{equation}
Next, define the sample mean
$
\bar X_n = n^{-1}\sum_{i=1}^n X_i
$
and observe that $|\bar X_n|\leq 1$. Hence, denoting
$$
Y_i = X_i - \bar X_n,\quad\text{for all }i=1,\dots, n,
$$
we have by the triangle inequality that
\begin{equation}\label{eq: decomposition}
\left|\sum_{i=1}^n a_i X_i\right| = \left|\sum_{i=1}^n a_i(X_i - \bar X_n) + \sum_{i=1}^n a_i \bar X_n\right| \leq \left|\sum_{i=1}^n a_i Y_i \right| + \left|\sum_{i=1}^n a_i\right|.
\end{equation}
Now, observe that $Y_1,\dots,Y_n$ are exchangeable random variables, and so for all $i=1,\dots,n$,
$$
\Ep[Y_i\mid Y_1,\dots,Y_{i-1}] = \Ep\left[\frac{1}{n-i+1}\sum_{j=i}^n Y_j\mid Y_1,\dots,Y_{i-1}\right].
$$
Hence, denoting
\begin{equation}\label{eq: martingale}
R_i = Y_i + \frac{1}{n-i+1}\sum_{j=1}^{i-1} Y_j,\quad\text{for all }i=1,\dots, n,
\end{equation}
it follows that for all $i = 1,\dots,n$,
\begin{align*}
\Ep[R_i\mid R_1,\dots,R_{i-1}] 
& = \Ep[R_i\mid Y_1,\dots, Y_{i-1}]\\
& = \Ep\left[Y_i + \frac{1}{n-i+1}\sum_{j=1}^{i-1} Y_j\mid Y_1,\dots,Y_{i-1}\right]\\
& = \Ep\left[ \frac{1}{n-i+1}\sum_{j=1}^n Y_j \mid Y_1,\dots,Y_{i-1} \right] = 0.
\end{align*}
Thus, $(R_i,\mathcal F_i)_{i=1}^n$, where $\mathcal F_i = \{R_1,\dots,R_i\}$ for all $i=1,\dots,n$, is a martingale difference sequence. In addition, for all $i=1,\dots,n$,
$$
|R_i| = \left| Y_i + \frac{1}{n - i + 1}\sum_{j=1}^{i-1} Y_j \right| = \left| Y_i - \frac{1}{n-i+1}\sum_{j=i}^n Y_j \right| \leq \max_{1\leq j\leq n}|X_i-X_j| \leq 2.
$$
Moreover, using an induction argument, it follows from \eqref{eq: martingale} that for all $i=1,\dots,n$,
\begin{equation}\label{eq: inversion}
Y_i = R_i - \sum_{j=1}^{i-1} \frac{R_j}{n - j},
\end{equation}
Indeed, \eqref{eq: inversion} holds trivially for $i=1$. Hence, assuming that \eqref{eq: inversion} holds for all $i=1,\dots,k-1$ for some $k=2,\dots,n$, we have that
\begin{align*}
Y_k
& = R_k - \frac{1}{n-k+1}\sum_{j=1}^{k-1} Y_j 
= R_k - \frac{1}{n-k+1}\sum_{j=1}^{k-1}\left( R_j - \sum_{l=1}^{j-1} \frac{R_l}{n-l} \right)\\
& = R_k - \frac{1}{n-k+1}\sum_{j=1}^{k-1}R_j\left( 1 - \frac{k-1-j}{n-j} \right)
= R_k - \sum_{j=1}^{k-1}\frac{R_j}{n-j},
\end{align*}
meaning that \eqref{eq: inversion} holds for $i=k$ as well, and thus for all $i=1,\dots,n$ by induction. In turn, it follows from \eqref{eq: inversion} that
$$
\sum_{i=1}^n a_i Y_i = \sum_{i=1}^n c_i R_i,
$$
where
$$
c_i = a_i - \frac{1}{n-i}\sum_{j=i+1}^n a_j,\quad\text{for all }i=1,\dots,n.
$$
Here, we have
$$
|c_i|\leq 2|a_i|,\quad\text{for all }i=1,\dots,n,
$$
by \eqref{eq: decreasing order}, and so
$$
\sum_{i=1}^n c_i^2 \leq 4 \sum_{i=1}^n a_i^2.
$$
Hence, by the Azuma-Hoeffding inequality, for any $t>0$,
\begin{align*}
\Pr\left(\left| \sum_{i=1}^n a_i Y_i \right| > t\right) & = \Pr\left(\left| \sum_{i=1}^n c_i R_i \right| > t\right) \\
& \leq 2\exp\left(-\frac{t^2}{8\sum_{i=1}^n c_i^2}\right) \leq 2\exp\left(-\frac{t^2}{32\sum_{i=1}^n a_i^2}\right).
\end{align*}
Combining this bound with \eqref{eq: decomposition} gives the asserted claim of the lemma.
\end{proof}

\begin{lemma}[Randomized Lindeberg Interpolation]\label{lem: rand lind}
Let $\mathcal S_n$ be the set of all one-to-one functions mapping $\{1,\dots,n\}$ to $\{1,\dots,n\}$. Also, let $X_1,\dots,X_n$ and $Y_1,\dots,Y_n$ be sequences of vectors in $\mathbb R^p$, $U$ be a random variable with uniform distribution on $[0,1]$, and $\sigma$ be a random function with uniform distribution on $\mathcal S_n$. Assume that $U$ is independent of $\sigma$, and for all $k=1,\dots,n$, denote
$$
W_k^\sigma= \sum_{j=1}^{k-1}X_{\sigma(j)}+\sum_{j=k+1}^{n}Y_{\sigma(j)}
$$
and
$$
W_{k}=\begin{cases}
W_{\sigma^{-1}(k)}^\sigma+X_k & \text{if }U\leq\frac{\sigma^{-1}(k)}{n+1},\\
W_{\sigma^{-1}(k)}^\sigma+Y_k & \text{if }U>\frac{\sigma^{-1}(k)}{n+1}.
\end{cases} 
$$
Then the distribution of $W_k$ is independent of $k$, i.e. there exists a random vector $\epsilon = (\epsilon_1,\dots,\epsilon_n)'$ with values in $\{0,1\}^n$ such that for all $k = 1,\dots,n$, the distribution of $W_k$ is equal to that of 
\begin{equation}\label{eq: eps distribution}
\sum_{i=1}^n\Big(\epsilon_i X_i + (1-\epsilon_i)Y_i\Big).
\end{equation}
Moreover, the random variables $\epsilon_1,\dots,\epsilon_n$ are exchangeable and are such that
$
\Pr(\sum_{i=1}^n \epsilon_i = s) = 1/(n+1)
$
for all $s=0,\dots,n$. In particular, $\Ep[\sum_{i=1}^n \epsilon_i] = n/2$. 
\end{lemma}
\begin{remark}
The first asserted claim of this lemma is the same as Lemma 2 in \cite{DZ17}. We present a self-contained proof of this claim below for reader's convenience.\qed
\end{remark}
\begin{proof}
Fix $k=1,\dots,n$. To show that the distribution of $W_k$ is independent of $k$, it suffices to show that for any subset $S$ of $\{1,\dots,n\}$,
\begin{equation}\label{eq: prob linde}
\Pr\left( W_k = \sum_{i\in S} X_i + \sum_{i\notin S} Y_i \right)
\end{equation}
is independent of $k$. To do so, fix any $S\subset \{1,\dots,n\}$ and denote $s = |S|$. If $k\notin S$, then
\begin{align*}
&\Pr\left( W_k = \sum_{i\in S} X_i + \sum_{i\notin S} Y_i \right)\\
&\quad=\Pr\left( \{\sigma^{-1}(i)\leq s,\forall i\in S\}\cap\{\sigma^{-1}(k) = s+1\}\cap\left\{U>\frac{s + 1}{n+1}\right\} \right)\\
&\quad = \Pr\Big( \{\sigma^{-1}(i)\leq s,\forall i\in S\}\cap\{\sigma^{-1}(k) = s+1\} \Big)\times\Pr\left(U>\frac{s + 1}{n+1}\right)\\
&\quad = \frac{1}{n}\frac{1}{{{n-1}\choose s}}\left(1 - \frac{s+1}{n+1}\right) = \frac{s!(n-s)!}{(n+1)!},
\end{align*}
where we used the fact that $\sigma^{-1}$ is also uniformly distributed on $\mathcal S_n$.
Similarly, if $k\in S$, then
\begin{align*}
&\Pr\left( W_k = \sum_{i\in S} X_i + \sum_{i\notin S} Y_i \right)\\
&\quad = \Pr\left( \{ \sigma^{-1}(i)\leq s,\forall i\in S \}\cap\{ \sigma^{-1}(k) = s \}\cap\left\{ U \leq \frac{s}{n+1} \right\} \right)\\
&\quad = \Pr\Big( \{\sigma^{-1}(i)\leq s,\forall i\in S\}\cap\{ \sigma^{-1}(k)=s \}\Big)\times\Pr\left( U \leq \frac{s}{n+1} \right)\\
&\quad = \frac{1}{n}\frac{1}{{{n-1}\choose{s-1}}}\frac{s}{n+1} = \frac{s!(n-s)!}{(n+1)!}.
\end{align*}
Hence, the probability in \eqref{eq: prob linde} is independent of $k$, and so is the distribution of $V_k$.

Further, since $W_k$ can only take values of the form $\sum_{i\in S}X_i + \sum_{i\notin S}Y_i$, where $S$ is a subset of $\{1,\dots,n\}$, it follows that there exists a random vector $\epsilon = (\epsilon_1,\dots,\epsilon_n)'$ with values in $\{0,1\}^n$ such that the distribution of $W_k$ is equal to that of \eqref{eq: eps distribution}. To see that the random variables $\epsilon_1,\dots,\epsilon_n$ are exchangeable, note that for any subset $S$ of $\{1,\dots,n\}$ with $s = |S|$ elements,
\begin{align*}
&\Pr\left(\epsilon_i =1 \ \forall i\in S\text{ and }\epsilon_i = 0 \ \forall i\notin S\right)\\
&\quad = \Pr\left( W_k = \sum_{i\in S}X_i + \sum_{\notin S}Y_i \right)  = \frac{s!(n-s)!}{(n+1)!},
\end{align*}
which depends on the set $S$ only via $s$. Thus, permuting the random variables $\epsilon_1,\dots,\epsilon_n$ in the vector $\epsilon$ creates a vector with the same distribution, which means that these random variables are exchangeable.

Finally, for any $s = 0,\dots,n$,
$$
\Pr\left( \sum_{i=1}^n \epsilon_i = s \right) = {n \choose s}\frac{s!(n-s)!}{(n+1)!} = \frac{1}{n+1}
$$
and
$$
\Ep\left[\sum_{i=1}^n\epsilon_i\right] = \sum_{s=0}^{n}\frac{s}{n+1} = \frac{n(n+1)}{2(n+1)} = \frac{n}{2}.
$$
This completes the proof of the lemma.
\end{proof}

\begin{lemma}[Third-Order Matching Multipliers with Gaussian Component]\label{lem: multipliers}
Let
$
\gamma\in(0;  1/2 - 1/(2\sqrt 5))
$
be a constant. Then
$$
\sigma=\left(1 - \frac{(1-\gamma)^{1/3}\gamma^{1/3}}{(1-2\gamma)^{2/3}}\right)^{1/2}
$$
is a real number satisfying $\sigma > 0$. Further, denote
$$
a=\frac{(1-\gamma)^{2/3}}{\gamma^{1/3}(1-2\gamma)^{1/3}}\quad\text{and}\quad b=-\frac{\gamma^{2/3}}{(1-\gamma)^{1/3}(1-2\gamma)^{1/3}}
$$
and let $e_1$ and $e_2$ be independent random variables such that $e_1$ has the $N(0,\sigma^2)$ distribution and $e_2$ takes values $a$ and $b$ with probabilities $\gamma$ and $1 - \gamma$, respectively. Then the random variable $e = e_1 + e_2$ has the following properties:
\begin{equation}\label{eq: multiplier equalities}
\Ep[e]=0,\quad \Ep[e^2]=1,\quad\text{and}\quad\Ep[e^3]=1.
\end{equation}
\end{lemma}
\begin{remark}
The distribution of the random variable $e$ constructed in this lemma is different from that used in \cite{DZ17}. It appears that our construction is easier to work with. \qed
\end{remark}
\begin{proof}
To show that $\sigma$ is a real number satisfying $\sigma > 0$, it suffices to show that
$$
(1-2\gamma)^2 > (1-\gamma)\gamma,
$$
which in turn is equivalent to
$$
5\gamma^2 - 5\gamma + 1 > 0,
$$
which holds by the choice of $\gamma$. Further, it is straightforward to check that
$$
\Ep[e_2]=0,\quad \Ep[e_2^2] = 1 - \sigma^2,\quad\text{and}\quad\Ep[e_2^3]=1.
$$
Thus, given that
$$
\Ep[e_1]=0,\quad\Ep[e_1^2]=\sigma^2,\quad\text{and}\quad\Ep[e_1^3]=0,
$$
the equalities in \eqref{eq: multiplier equalities} follow from
$$
\Ep[e] = \Ep[e_1]+\Ep[e_2],\ \Ep[e^2] = \Ep[e_1^2]+\Ep[e_2^2], \text{ and } \Ep[e^3] = \Ep[e_1^3] + \Ep[e_2^3].
$$
This completes the proof of the lemma.
\end{proof}

\begin{lemma}[Randomization Tests with Mass Points]\label{lem: rand test}
Let $\mathcal X$ and $X$ be a set and a random variable taking values in this set. Also, let $G$ be a set of $M$ one-to-one functions mapping $\mathcal X$ onto $\mathcal X$ such that (i) for all $g\in G$, the distribution of $g(X)$ is equal to that of $X$, (ii) for all $g\in G$, we have $g^{-1}\in G$, and (iii) for all $g_1,g_2\in G$, we have $g_2\circ g_1\in G$. Further, let $T$ be a function mapping $\mathcal X$ to $\mathbb R$ and for $\alpha\in(0,1)$, define $\phi\colon \mathcal X\to\{0,1\}$ by
$$
\phi(x)=\begin{cases}
1, & \text{if }\sum_{g\in G}1\{ T(x)>T(g(x))\} \geq M(1-\alpha),\\
0, & \text{if }\sum_{g\in G}1\{ T(x)>T(g(x))\} <M(1-\alpha),
\end{cases}
\ \text{for all }x\in\mathcal X.
$$
Finally, define $\chi\colon\mathcal X\to\mathbb R$ by
$$
\chi(x)=\max_{t\in\mathbb R}|\{g\in G\colon T(g(x)) = t\}|/M,\quad\text{for all }x\in\mathcal X.
$$
Then
$$
\alpha - \Ep[\chi(X)] \leq \Ep[\phi(X)] \leq \alpha.
$$
\end{lemma}
\begin{remark}
If $X$ is observable data and $T(X)$ is a statistic, we can think of $\phi(X)$ as a level $\alpha$ randomization test that exploits symmetries of $X$ with respect to a set of transformations $G$. The result presented here is then similar to Theorem 15.2.1 in \cite{LR05}, with the difference coming from the fact that we do not allow the function $\phi$ to take values in $(0,1)$ and instead quantify how much the test can under-reject because of the mass points.
\end{remark}
\begin{proof}
Define $\phi\colon \mathcal X\times\mathcal X\to\{0,1\}$ by
$$
\phi(x,y)=\begin{cases}
1, & \text{if }\sum_{g\in G}1\{ T(x)>T(g(y))\} \geq M(1-\alpha),\\
0, & \text{if }\sum_{g\in G}1\{ T(x)>T(g(y))\} <M(1-\alpha),
\end{cases}
\ \text{for all }x,y\in\mathcal X,
$$
so that $\phi(x) = \phi(x,x)$ for all $x\in\mathcal X$. Observe that for any $x\in\mathcal X$, we have 
$$
\frac{1}{M}\sum_{g\in G}\phi(g(X),X) \leq \alpha\quad\text{and}\quad \frac{1}{M}\sum_{g\in G}\phi(g(X),X) \geq \alpha - \chi(X)
$$ 
by construction of the function $\phi$. Hence,
\begin{align*}
\alpha
& \geq \frac{1}{M}\sum_{g\in G}\Ep[\phi(g(X),X)] = \frac{1}{M}\sum_{g\in G}\Ep[\phi(g(X),g(X))] \\
& = \frac{1}{M}\sum_{g\in G}\Ep[\phi(X,X)] = \Ep[\phi(X,X)] = \Ep[\phi(X)],
\end{align*} 
where the first equality follows from noting that for all $g_2\in G$, we have $\{T(g_1(X))\}_{g_1\in G} = \{T(g_1(g_2(X)))\}_{g_1\in X}$, and the second from noting that $g(X)$ is equal in distribution to $X$ for all $g\in G$.
Similarly, we also have
\begin{align*}
\alpha - \Ep[\chi(X)]
& \leq \frac{1}{M}\sum_{g\in G}\Ep[\phi(g(X),X)] = \Ep[\phi(X)].
\end{align*} 
Combining these bounds gives the asserted claim.
\end{proof}

\section{Other Useful Lemmas}\label{sec: useful lemmas}
In this section, we collect maximal, deviation, and anti-concentration inequalities that are useful for our analysis. Lemmas \ref{lem: maximal ineq}--\ref{lem: anticoncetration} are taken from \cite{CCK17} where the last one is based on Nazarov's \cite{Na03} work. Lemma \ref{lem: rademacher anticoncentration} is essentially taken from \cite{OST18}.
\begin{lemma}[Maximal Inequality for Centered Sums]
\label{lem: maximal ineq}
Let $X_{1},\dots,X_{n}$ be independent centered random vectors in $\R^{p}$ with $p\geq 2$. Define $Z$, $M$, and $\sigma^2$ by $Z = \max_{1 \leq j \leq p} | \sum_{i=1}^{n} X_{ij}|$, $M = \max_{1 \leq i \leq n} \max_{1 \leq j \leq p} | X_{ij} |$ and $\sigma^{2} = \max_{1 \leq j \leq p} \sum_{i=1}^{n} \Ep [ X_{ij}^{2} ]$. Then
\begin{equation*}
\Ep [Z]  \leq K (\sigma \sqrt{\log p} + \sqrt{\Ep [ M^{2} ]} \log p).
\end{equation*}
where $K$ is a universal constant.
\end{lemma}

\begin{lemma}[Deviation Inequality for Centered Sums]
\label{lem: fuk-nagaev}
Assume the setting of Lemma \ref{lem: maximal ineq}. For every $\eta > 0, \beta \in (0,1]$ and $t>0$, we have
\[
\Pr \{  Z \geq (1+\eta) \Ep [ Z ] + t \} \leq \exp \{ -t^{2}/(3\sigma^{2}) \} +3\exp \{ -( t/(K\| M \|_{\psi_{\beta}}))^{\beta} \},
\]
where $K$ is a constant depending only on $\eta$ and $\beta$. 
\end{lemma}

\begin{lemma}[Gaussian Anti-Concentration Inequality]\label{lem: anticoncetration}
Let $Y = (Y_1,\dots,Y_n)'$ be a centered Gaussian random vector in $\mathbb R^p$ with $p\geq 2$ such that $\Ep[Y_j^2]\geq b$ for all $j = 1,\dots,p$ and some constant $b>0$. Then for every $y\in\mathbb R^p$ and $t>0$,
$$
\Pr(Y\leq y + t) - \Pr(Y\leq y) \leq C t \sqrt{\log p},
$$
where $C$ is a constant depending only on $b$.
\end{lemma}
\begin{lemma}[Rademacher Anti-Concentration Inequality]\label{lem: rademacher anticoncentration}
Let $Z_1,\dots,Z_n$ be vectors in $\mathbb R^p$ with $p\geq 2$ and let $e_1,\dots,e_n$ be independent Rademacher random variables. Define $Y = n^{-1/2}\sum_{i=1}^n e_i Z_{i}$ and assume that for some constants $b,B>0$, (i) $bn\leq \sum_{i=1}^n Z_{i j}^2$ for all $j = 1,\dots,p$ and (ii) $\|Z_i\|_{\infty}\leq B$ for all $i = 1,\dots,n$. Then for every $y\in\mathbb R^p$ and $t>0$,
$$
\Pr(Y\leq y + t) - \Pr(Y\leq y) \leq C (t + B/\sqrt n) \sqrt{\log p},
$$
where $C$ is a constant depending only on $b$.
\end{lemma}
\begin{proof}
Since the result for $t\in(0,B/\sqrt n)$ follows from the result for $t=B/\sqrt n$, it suffices to consider the case $t\geq B/\sqrt n$. Next, by the proof of Theorem 7.1 in \cite{OST18}, there exists a constant $K$ depending only on $b$ such that for all $y\in\mathbb R^p$ and $t\geq B/\sqrt n$, we have
\begin{equation}\label{eq: odonnell anticoncentration}
\Pr(Y\leq y + t) - \Pr(Y\leq y) \leq K t \sqrt{\log p} + \exp(\log p - K/t^2).
\end{equation}
Here, since the asserted claim is trivial if $2t^2(\log(1/t) + \log p) > K$, we can assume that $2t^2(\log(1/t) + \log p) \leq K$, in which case the right-hand side of \eqref{eq: odonnell anticoncentration} is bounded from above by
$$
Kt\sqrt{\log p} + t\exp(-K/(2t^2))\leq Kt\sqrt{\log p} + t.
$$
The asserted claim follows.
\end{proof}

\section{Simulation Results}\label{sec: simulation results}
In this section, we present results of a small-scale Monte Carlo simulation study. The purpose of the simulation study is two-fold. First, it confirms that all approximation methods discussed in the previous section work well in finite samples. Second, it compares the relative performance of different methods in the high-dimensional regime. 

We generate random vectors $X_1,\dots,X_n$ by setting
\begin{equation}\label{eq: asymmetric case}
X_{i j} = F^{-1}(\Phi(Y_{i j})),\quad\text{for all }i=1,\dots,n\text{ and }j=1,\dots,p,
\end{equation}
where random vectors $Y_1,\dots,Y_n$ are sampled independently from the centered Gaussian distribution with covariance matrix $\Sigma$ such that $\Sigma_{j k} = \rho^{|j-k|}$ for all $j,k=1,\dots,p$, $\Phi$ is the cdf of the $N(0,1)$ distribution, and, depending on the experiment, $F^{-1}$ is the quantile function of either the Weibull or the Gamma distribution. For both distributions, we set the scale parameter to be one but we set the shape parameter $k$ to be either $2$, $3$, or $4$ in the case of the Weibull distribution and either $1$, $3$, or $5$ in the case of the Gamma distribution. Depending on the experiment, we set the correlation parameter $\rho$ to be either $0$, $0.25$, $0.5$, or $0.75$. Also, we set $n = 400$ and $p$ to be either $400$ or $800$.

We refer to \eqref{eq: asymmetric case} as the case of asymmetric distributions. In addition, since we obtain better bounds for the multiplier bootstrap with Rademacher weights if Condition S is satisfied, we also consider the case of symmetric distributions by setting
$$
X_{i j} = F^{-1}(\Phi(Y^1_{i j})) - F^{-1}(\Phi(Y^2_{i j})),\quad\text{for all }i=1,\dots,n\text{ and }j=1,\dots,p,
$$
where $Y_1^1,\dots,Y_n^1$ and $Y_1^2,\dots,Y_n^2$ are two independent copies of $Y_1,\dots,Y_n$. Since approximations are better in this case, to differentiate between different types of approximations, we replace the sample size $n=400$ by $n=100$ and we keep the same choices for all other parameters. 
Table \ref{tab:mom} reports the first four moments of $X_{ij}$'s.

For all types of the bootstrap, we calculate the critical value $c_{1-\alpha}^B$ using 500 bootstrap samples. To implement the third-order matching multiplier bootstrap, we sample the weights $e_i$ from the distribution constructed in Lemma \ref{lem: multipliers} with $\gamma = 0.2$. In all cases, we set the nominal level $\alpha = 0.1$. We estimate each rejection probability $\Pr(T_n > c_{1-\alpha}^B)$ using 20,000 simulations. 

The results of our simulations for the Weibull and the Gamma distributions are presented in Tables 1 and 2, respectively, and can be summarized as follows. First, we observe similar patterns in both tables. Second, all methods perform well in most cases even though we consider relatively small sample sizes, with the exception of the multiplier bootstrap with Rademacher weights, which tends to substantially over-reject in the case of the Gamma distributions, especially with small $k$.  Third, in the case of the asymmetric distributions, the empirical and the third-order matching multiplier bootstrap methods clearly outperform the multiplier bootstrap methods with Gaussian and Rademacher weights. This is especially clear, for example, in the case of the Gamma distribution with $k=3$ and $p=400$, where the rejection probabilities $\Pr(T_n > c_{1-\alpha}^B)$ are about $0.09-0.10$ for the empirical and the third-order matching multiplier bootstrap methods but are about $0.13-0.15$ for the multiplier bootstrap methods with Gaussian and Rademacher weights. Fourth, the multiplier bootstrap method with Gaussian weights improves and becomes comparable to the empirical and the third-order matching bootstrap methods in the case of symmetric distributions. However, the multiplier bootstrap method with Rademacher weights improves substantially more and in overall gives the best results among all methods in this case. An especially striking example of this conclusion is the case of the Gamma distribution with $k=1$ and $p = 800$, where the rejection probabilities $\Pr(T_n > c_{1-\alpha}^B)$ are about $0.10-0.11$ for the multiplier bootstrap method with Rademacher weights but are about $0.05-0.07$ for all other bootstrap methods. 

\newpage

\begin{table}[H]\label{tab: 1}
\caption{\small{Results of Monte Carlo experiments for bootstrap rejection probabilities $\Pr(T_n > c_{1-\alpha}^B)$ with $\alpha = 10\%$} and 4 types of bootstrap: multiplier bootstrap with Gaussian weights (GB), empirical bootstrap (EB), multiplier bootstrap with Rademacher weights (RB), and third-order matching multiplier bootstrap (MB). The case of the Weibull distributions.
\medskip
}
 
\hspace*{-1cm}
\begin{tabular}{cccccccccc}
\hline\hline
 \tabularnewline
\multicolumn{9}{c}{Design 1: Asymmetric Distributions, $n=400$}\tabularnewline
 \tabularnewline

\hline\hline
\multirow{2}{*}{$k$} & \multirow{2}{*}{$\rho$} & \multicolumn{4}{c}{$p=400$}\vline & \multicolumn{4}{c}{$p=800$} \tabularnewline
\cline{3-10}
 & & GB & EB & RB & MB & GB & EB & RB & MB   \tabularnewline
\hline\hline

\multirow{4}{*}{$2$}
&.00    &   .117 &   .098 &   .125 &   .099 &  .125 &   .102 &   .133 &   .102\tabularnewline
&.25    &   .121 &   .100 &   .126 &   .099 &  .121 &   .097 &   .129 &   .097\tabularnewline
&.50    &   .114 &   .095 &   .122 &   .096 &  .124 &   .100 &   .133 &   .102\tabularnewline
&.75    &   .117 &   .098 &   .122 &   .099 &  .121 &   .099 &   .128 &   .099\tabularnewline
\cline{1-10}

\multirow{4}{*}{$3$}
&.00    &   .110 &   .105 &   .115 &   .105 &  .106 &   .100 &   .114 &   .101\tabularnewline
&.25    &   .105 &   .101 &   .110 &   .100 &  .107 &   .102 &   .114 &   .099\tabularnewline
&.50    &   .103 &   .098 &   .108 &   .098 &  .107 &   .101 &   .113 &   .100\tabularnewline
&.75    &   .106 &   .103 &   .112 &   .101 &  .104 &   .099 &   .112 &   .098\tabularnewline
\cline{1-10}

\multirow{4}{*}{$4$}
&.00    &   .096 &   .099 &   .101 &   .097 &  .095 &   .099 &   .102 &   .098\tabularnewline
&.25    &   .096 &   .099 &   .102 &   .098 &  .098 &   .102 &   .105 &   .103\tabularnewline
&.50    &   .093 &   .095 &   .097 &   .095 &  .100 &   .102 &   .107 &   .103\tabularnewline
&.75    &   .099 &   .101 &   .103 &   .101 &  .098 &   .102 &   .104 &   .100\tabularnewline
\hline\hline

 \tabularnewline
\multicolumn{9}{c}{Design 2: Symmetric Distributions, $n=100$}\tabularnewline
 \tabularnewline

\hline\hline

\multirow{2}{*}{$k$} & \multirow{2}{*}{$\rho$} & \multicolumn{4}{c}{$p=400$}\vline & \multicolumn{4}{c}{$p=800$} \tabularnewline
\cline{3-10}
 & & GB & EB & RB & MB & GB & EB & RB & MB   \tabularnewline
\hline\hline

\multirow{4}{*}{$2$} 
&.00  &     .088 &   .087 &   .110 &   .087 &    .082 &   .083 &   .108 &   .081\tabularnewline
&.25  &     .083 &   .082 &   .104 &   .082 &    .082 &   .083 &   .108 &   .081\tabularnewline
&.50  &     .089 &   .088 &   .109 &   .087 &    .082 &   .082 &   .109 &   .081\tabularnewline
&.75  &     .090 &   .090 &   .108 &   .089 &    .085 &   .084 &   .108 &   .084\tabularnewline

\cline{1-10}
\multirow{4}{*}{$3$}
&.00  &     .088 &   .090 &   .109 &   .088 &    .086 &   .086 &   .109 &   .084\tabularnewline
&.25  &     .086 &   .088 &   .108 &   .087 &    .085 &   .086 &   .109 &   .085\tabularnewline
&.50  &     .090 &   .090 &   .110 &   .089 &    .087 &   .088 &   .110 &   .086\tabularnewline
&.75  &     .093 &   .095 &   .109 &   .093 &    .089 &   .089 &   .111 &   .089\tabularnewline

\cline{1-10}
\multirow{4}{*}{$4$}
&.00  &     .086 &   .090 &   .108 &   .086 &    .085 &   .086 &   .108 &   .081\tabularnewline
&.25  &     .085 &   .087 &   .105 &   .084 &    .082 &   .081 &   .104 &   .080\tabularnewline
&.50  &     .090 &   .091 &   .109 &   .089 &    .088 &   .088 &   .111 &   .085\tabularnewline
&.75  &     .092 &   .092 &   .107 &   .090 &    .093 &   .092 &   .113 &   .091\tabularnewline

\hline\hline
\end{tabular}
\end{table}

\newpage

\begin{table}[H]\label{tab: 2}
\caption{\small{Results of Monte Carlo experiments for bootstrap rejection probabilities $\Pr(T_n > c_{1-\alpha}^B)$ with $\alpha = 10\%$} and 4 types of bootstrap: multiplier bootstrap with Gaussian weights (GB), empirical bootstrap (EB), multiplier bootstrap with Rademacher weights (RB), and third-order matching multiplier bootstrap (MB). The case of the Gamma distributions.
\medskip
}
 
\hspace*{-1cm}
\begin{tabular}{cccccccccc}
\hline\hline
 \tabularnewline
\multicolumn{9}{c}{Design 1: Asymmetric Distributions, $n=400$}\tabularnewline
 \tabularnewline

\hline\hline
\multirow{2}{*}{$k$} & \multirow{2}{*}{$\rho$} & \multicolumn{4}{c}{$p=400$}\vline & \multicolumn{4}{c}{$p=800$} \tabularnewline
\cline{3-10}
 & & GB & EB & RB & MB & GB & EB & RB & MB   \tabularnewline
\hline\hline

\multirow{4}{*}{$1$}
&.00  &   .143  &  .081 &   .166 &   .087 &   .157 &   .084 &   .190 &   .092\tabularnewline
&.25  &   .151  &  .085 &   .171 &   .093 &   .156 &   .081 &   .190 &   .091\tabularnewline
&.50  &   .142  &  .081 &   .167 &   .087 &   .155 &   .078 &   .185 &   .087\tabularnewline
&.75  &   .143  &  .082 &   .164 &   .088 &   .150 &   .080 &   .179 &   .088\tabularnewline
\cline{1-10}

\multirow{4}{*}{$3$}
&.00   &   .135  &  .096 &   .147 &   .098 &   .136 &   .092 &   .152 &   .096\tabularnewline
&.25   &   .131  &  .092 &   .143 &   .095 &   .140 &   .092 &   .155 &   .095\tabularnewline
&.50   &   .130  &  .092 &   .142 &   .092 &   .134 &   .092 &   .151 &   .096\tabularnewline
&.75   &   .129  &  .096 &   .140 &   .097 &   .130 &   .090 &   .144 &   .093\tabularnewline
\cline{1-10}

\multirow{4}{*}{$5$}
&.00   &   .123  &  .094 &   .134 &   .096 &   .126 &   .093 &   .136 &   .093\tabularnewline
&.25   &   .124  &  .095 &   .133 &   .096 &   .130 &   .094 &   .144 &   .097\tabularnewline
&.50   &   .118  &  .094 &   .130 &   .095 &   .130 &   .094 &   .142 &   .098\tabularnewline
&.75   &   .123  &  .094 &   .132 &   .096 &   .125 &   .092 &   .135 &   .093\tabularnewline
\hline\hline

 \tabularnewline
\multicolumn{9}{c}{Design 2: Symmetric Distributions, $n=100$}\tabularnewline
 \tabularnewline

\hline\hline

\multirow{2}{*}{$k$} & \multirow{2}{*}{$\rho$} & \multicolumn{4}{c}{$p=400$}\vline & \multicolumn{4}{c}{$p=800$} \tabularnewline
\cline{3-10}
 & & GB & EB & RB & MB & GB & EB & RB & MB   \tabularnewline
\hline\hline

\multirow{4}{*}{$1$} 
&.00   &   .070 &   .061 &   .107 &   .068 &  .064 &   .053 &   .110  &  .061\tabularnewline
&.25   &   .066 &   .059 &   .103 &   .064 &  .062 &   .053 &   .108  &  .062\tabularnewline
&.50   &   .071 &   .063 &   .108 &   .069 &  .063 &   .053 &   .108  &  .062\tabularnewline
&.75   &   .074 &   .066 &   .107 &   .072 &  .065 &   .055 &   .104  &  .062\tabularnewline

\cline{1-10}
\multirow{4}{*}{$3$}
&.00   &   .081 &   .078 &   .109 &   .079 &  .073 &   .070 &   .107  &  .071\tabularnewline
&.25   &   .080 &   .077 &   .107 &   .079 &  .076 &   .072 &   .109  &  .074\tabularnewline
&.50   &   .081 &   .077 &   .109 &   .080 &  .076 &   .074 &   .109  &  .076\tabularnewline
&.75   &   .087 &   .085 &   .111 &   .086 &  .082 &   .076 &   .112  &  .079\tabularnewline

\cline{1-10}
\multirow{4}{*}{$5$}
&.00   &   .081 &   .080 &   .105 &   .081 &  .077 &   .076 &   .107  &  .076\tabularnewline
&.25   &   .081 &   .079 &   .105 &   .079 &  .077 &   .075 &   .106  &  .076\tabularnewline
&.50   &   .083 &   .080 &   .107 &   .083 &  .082 &   .079 &   .111  &  .081\tabularnewline
&.75   &   .090 &   .088 &   .112 &   .090 &  .086 &   .084 &   .113  &  .084\tabularnewline

\hline\hline
\end{tabular}
\end{table}

\begin{table}[ht]
\centering
\caption{\small{The first four moments of marginal distributions used in the simulation study.}}
\label{tab:mom}

\begin{tabular}{cccccccccc}
\hline\hline\\[-7.5pt]
& & & \multicolumn{5}{c}{Weibull Distributions} \\[2.5pt]
\hline\hline
\multirow{2}{*}{$k$} & \multicolumn{4}{c}{Asymmetric Distributions} & & \multicolumn{4}{c}{Symmetric Distributions} \\
\cline{2-10}
 & 1st & 2nd & 3rd & 4th & & 1st & 2nd & 3rd & 4th \\ 
  \hline\hline
  2 & 0.886 & 1 & 1.329 & 2 & & 0 & 0.429 & 0 & 0.575 \\ 
  3 & 0.893 & 0.903 & 1 & 1.191 & & 0 & 0.211 & 0 & 0.127\\ 
  4 & 0.906 & 0.886 & 0.919 & 1 & & 0 & 0.129 & 0 & 0.048 \\
  \hline\hline
  \\[-7.5pt]
& & & \multicolumn{5}{c}{Gamma Distributions} \\[2.5pt]
\hline\hline
  \multirow{2}{*}{$k$} & \multicolumn{4}{c}{Asymmetric Distributions} & & \multicolumn{4}{c}{Symmetric Distributions} \\
\cline{2-10}
& 1st & 2nd & 3rd & 4th & & 1st & 2nd & 3rd & 4th \\ 
\hline\hline
1 & 1 & 2 & 6 & 24  & & 0 & 2 & 0 & 24 \\ 
  3 & 3 & 12 & 60 & 360 &  & 0 & 6 & 0 & 144\\ 
  5 & 5 & 30 & 210 & 1680  &  & 0 & 10 & 0 & 360 \\ 
   \hline
\end{tabular}
\end{table}

\end{document}